\documentclass[12pt]{article}

\usepackage{times}

\usepackage[utf8]{inputenc}
\usepackage{commath}
\usepackage{amsthm, amscd}
\usepackage{amsmath}
\usepackage{amssymb}
\usepackage{cite}
\usepackage{setspace}

\usepackage{bigints}
\usepackage{comment}
\usepackage{mathtools}
\usepackage{mathrsfs}
\usepackage{fancyhdr}

\allowdisplaybreaks[4]

\numberwithin{equation}{section}
\usepackage{etoolbox}
\patchcmd{\thebibliography}
  {\settowidth}
  {\setlength{\itemsep}{0pt plus -10pt}\settowidth}
  {}{}
\apptocmd{\thebibliography}
  {
  }
  {}{}
\makeatletter
\newtheorem*{rep@theorem}{\rep@title}
\newcommand{\newreptheorem}[2]{%
\newenvironment{rep#1}[1]{%
 \def\rep@title{#2 \ref{##1}}%
 \begin{rep@theorem}}%
 {\end{rep@theorem}}}
\makeatother

\theoremstyle{theorem}

\newreptheorem{theorem}{Theorem}
\newtheorem{thm}{Theorem}[section]
\newtheorem*{thm*}{Theorem}
\theoremstyle{definition}
\newtheorem{prop}[thm]{Proposition}
\newtheorem*{prop*}{Proposition}
\newtheorem{defn}[thm]{Definition}

\newtheorem{cor}[thm]{Corollary}
\newtheorem*{cor*}{Corollary}
\theoremstyle{remark}

\newtheorem{const}[thm]{Construction}
\newtheorem{rem}[thm]{Remark}

\newtheorem{notat}[thm]{Notation}

\title{Quillen metric for singular families of Riemann surfaces with cusps and compact perturbation theorem.} 

\author
{Siarhei Finski
}

\date{}

\usepackage[%
    left=1in,%
    right=1in,%
    top=1.1in,%
    bottom=0.8in,%
    paperheight=11in,%
    paperwidth=8.5in%
]{geometry}

\newcommand{\imun} {\sqrt{-1}}
\newcommand{\vol}{v}

\newcommand{\reg}{{\textbf{r}}}
\newcommand{\comp}{\mathbb{C}}
\newcommand{\real}{\mathbb{R}}

\newcommand{\nat}{\mathbb{N}}
\newcommand{\integ}{\mathbb{Z}}

\newcommand{\dd}{\mathbb{D}}

\newcommand{\tinyint}{\begingroup\textstyle\int\endgroup}

\newcommand{\res}{{\rm{Res}}}

\newcommand{\ccal}{\mathscr{C}}

\newcommand{\dbar}{ \overline{\partial} }

\newcommand{\laplcomp}{\Box}

\newcommand{\rk}[1]{{\rm{rk}} ( #1 )}

\renewcommand{\Re}{\operatorname{Re}}
\renewcommand{\Im}{\operatorname{Im}}
\newcommand{\scal}[2]{\big< #1, #2 \big>}

\newcommand{\modul}{\mathscr{M}}
\newcommand{\modulcomp}{\overline{\mathscr{M}}}

\newcommand{\univcurv}{\mathscr{C}}
\newcommand{\univcurvcomp}{\overline{\mathscr{C}}}

\newcommand{\td}{{\rm{Td}}}
\newcommand{\ch}{{\rm{ch}}}

\newenvironment{sciabstract}{}

\begin{document} 
\maketitle

\begin{sciabstract}
  \textbf{Abstract.} We study the behavior of the Quillen metric for the family of Riemann surfaces with cusps when the additional cusps are created by degeneration.
  \par More precisely, in our previous paper, we've seen that the renormalization of the Quillen metric associated with a family of Riemann surfaces with cusps extends continuously over the locus of singular curves. The main result of this article shows that, modulo some explicit universal constant, this continuous extension coincides with the Quillen metric of the normalization of singular curves. This result shows that the Quillen metric is compatible with the adjunction of cusps.
  When this theorem is applied directly to the moduli space of curves, we obtain the compatibility of the Quillen metric with clutching morphisms in the moduli space of pointed stable curves.
  \par As one application, we obtain the compatibility between our definition of the analytic torsion and the definition of Takhtajan-Zograf using lengths of closed geodesics.
  \par As a consequence of the proof of the main theorem, we get an explicit relation in terms of Bott-Chern forms between the Quillen metric associated with a cusped metric and the Quillen metric associated with a metric on the compactified Riemann surface. This refines relative compact perturbation theorem we obtained before by pinning down the universal constant.

\end{sciabstract}

\pagestyle{fancy}
\lhead{}
\chead{Quillen metric for a singular family of Riemann surfaces with cusps.}
\rhead{\thepage}
\cfoot{}


\newcommand{\Addresses}{{
  \bigskip
  \footnotesize
  \noindent \textsc{Siarhei Finski, UFR de Mathématiques, Case 7012, Université Paris Diderot-Paris 7, France.}\par\nopagebreak
  \noindent  \textit{E-mail }: \texttt{siarhei.finski@imj-prg.fr}.
}}

\tableofcontents

\section{Introduction}\label{sect_intro}
	We study the behavior of the Quillen metric for the family of Riemann surfaces with cusps when the additional cusps are created by degeneration.
	\par Let $X$ and $S$ be complex manifolds, and let $\pi : X \to S$ be a proper holomorphic map. The construction of Grothendick-Knudsen-Mumford \cite{Knudsen1976} (cf. also \cite[\S 3]{BGS3}) associates for every holomorphic vector bundle $\xi$ over $X$ the “determinant of the direct image of $\xi$" - the holomorphic line bundle over $S$, which we denote  (cf. (\ref{eq_det_r_star}))
	\begin{equation}
		\lambda(j^* \xi)^{-1} := \det (R^{\bullet} \pi_* \xi).
	\end{equation}
	\par Let's fix a holomorphic, proper, surjective map $\pi: X \to S$ of complex manifolds, such that for every $t \in S$, the space $X_t := \pi^{-1}(t)$ is a complex curve whose singularities are at most ordinary double points (in the terminology of \cite{BisBost}, \cite{FinII2}, a f.s.o.).
		In other words, we are considering a family of Riemann surfaces with double-point singularities.
		\par 
		We denote by $\Sigma_{X/S} \subset X$ the submanifold of singular points of the fibers (see Corollary \ref{cor_sigma}).
		We denote by $\Delta = \pi_*(\Sigma_{X/S})$ the divisor formed by the locus of the singular fibers $\pi$.
		\par
		In this article we only consider $\pi$ for which the associated divisor $\Delta$ has normal crossings.
		\par 
		Let $\sigma_1, \ldots, \sigma_m : S \to X \setminus \Sigma_{X/S}$ be disjoint holomorphic sections of $\pi$.
		We denote by $D_{X/S}$ the divisor on $X$, given by 
		\begin{equation}\label{defn_dxs}
			D_{X/S} = \Im(\sigma_1) + \cdots + \Im(\sigma_m).			
		\end{equation}
		\par Let the norm $\norm{\cdot}_{X/S}^{\omega}$ on the canonical line bundle $\omega_{X/S}$ (see (\ref{eq_rel_can_defn})) over $X \setminus (\pi^{-1}(|\Delta|) \cup |D_{X/S}|)$ be such that its restriction over each nonsingular fiber $X_t := \pi^{-1}(t)$, $t \in S \setminus |\Delta|$ of $\pi$ induces the Kähler metric $g^{TX_t}$ on $X_t \setminus \{ \sigma_1(t), \ldots, \sigma_m(t) \}$ such that the triple $(X_t, \{ \sigma_1(t), \ldots, \sigma_m(t) \}, g^{TX_t})$ is a surface with cusps in the sense of \cite{FinII1}, \cite{MullerCusp}, \cite{FreixasARR} (see Section \ref{sect_recall_relantors}).
		In other words, we suppose that the norm $\norm{\cdot}_{X/S}^{\omega}$ induces complete Kähler metric over $X_t \setminus \{ \sigma_1(t), \ldots, \sigma_m(t) \}$ so that the scalar curvature is equal to $-1$ away from a compact subset on $X_t$.
		Thus, the sections $\sigma_1, \ldots, \sigma_m$ are essentially parameterizing the positions of cusps in the fibers. 
		\begin{const}\label{const_norm_div}
		For a complex manifold $Y$ and a divisor $D_0 \subset Y$, let $\norm{\cdot}^{\rm{div}}_{D_0}$ be the singular norm on $\mathscr{O}_Y(D_0)$, defined by
		\begin{equation}\label{defn_norm_D}
			\| s_{D_0} \|^{\rm{div}}_{D_0}(x) = 1, \quad \text{for any } x \in Y \setminus D_0,
		\end{equation}
		where $s_{D_0}$ is the canonical section of the divisor $D_0$ with ${\rm{div}} (s_{D_0}) = D_0$. 
		\par We endow the \textit{twisted canonical line bundle} 
		\begin{equation}\label{defn_tw_can}
			\omega_{X/S}(D) := \omega_{X/S} \otimes \mathscr{O}_X(D_{X/S})
		\end{equation}
		with the canonical norm $\, \norm{\cdot}_{X/S}$ over $X \setminus (\pi^{-1}(|\Delta|) \cup |D_{X/S}|)$, induced by $\norm{\cdot}^{\omega}_{X/S}$ and $\norm{\cdot}^{\rm{div}}_{D_{X/S}}$.
	\end{const}
	One should think of the norm $\norm{\cdot}_{X/S}$ as some norm which has logarithmic singularities in the neighborhood of  $|D_{X/S}|$.
	Let $(\xi, h^{\xi})$ be a holomorphic Hermitian vector bundle over $X$. Let $h^{\det \xi}$ be the induced Hermitian metric on $\det \xi := \Lambda^{\max} \xi$.
	In \cite[Definition 2.16]{FinII1}, we defined the analytic torsion $T(g^{TX_t}, h^{\xi} \otimes \norm{\cdot}_{M}^{2n})$ of the fiber $(X_t, g^{TX_t})$ associated with a singular Hermitian vector bundle $(\xi \otimes \omega_{X/S}(D)^n, h^{\xi} \otimes \norm{\cdot}_{M}^{2n})$ through the regularized trace of the heat kernel, obtained by subtraction of a universal contribution coming from cusp.
	See Section \ref{sect_recall_relantors} for a more detailed discussion on this.
	\par 
	In \cite[\S 2.1]{FinII1}, we've also seen that for $n \leq 0$, the $L^2$-scalar product\footnote{Our normalization differs from the one used in \cite{FinII1}, \cite{FinII2} by a factor $2 \pi$. The reason for such a normalization is to make things compatible with \cite{DeligDet}, \cite{GilSoul92},  \cite{BisLeb91}, \cite{FreixasARR}, in particular, to make Serre duality an isometry, see (\ref{eq_serre_isom}). }
	\begin{equation}\label{defn_L_2}
		\scal{\alpha}{\alpha'}_{L^2} := \frac{1}{2 \pi} \int_{X_t} \scal{\alpha(x)}{\alpha'(x)}_h d \vol_{X_t}(x), \quad \alpha, \alpha' \in \ccal^{\infty}(X_t, \xi \otimes \omega_{X/S}(D)^n),
	\end{equation}
	where $\langle \cdot, \cdot \rangle_h$ is the pointwise scalar product, $d \vol_{X_t}$ is the induced Riemannian volume form on $(X_t, g^{TX_t})$, induces the natural $L^2$-norm on the determinant line bundle $\lambda(j^*(\xi \otimes \omega_{X/S}(D)^n))$, $n \leq 0$ over $S \setminus |\Delta|$.
	\par 
	In \cite{FinII1}, \cite{FinII2} (cf. Definition \ref{defn_quil_norm}), we've defined the \textit{Quillen norm} $\norm{\cdot}_{Q} (g^{TX_t}, h^{\xi} \otimes \, \norm{\cdot}_{X/S}^{2n})$ on the determinant line bundle $\lambda(j^*(\xi \otimes \omega_{X/S}(D)^n))$, $n \leq 0$ over $S \setminus |\Delta|$ as the product of the analytic torsion of the fiber $T(g^{TX_t}, h^{\xi} \otimes \norm{\cdot}_{M}^{2n})$ from \cite{FinII1}, and the $L^2$-norm of the fiber, see Section \ref{sect_recall_relantors}. 
	As we explained in \cite{FinII2}, this definition gives a non-compact $1$-dimensional version of the definition of the Quillen norm of Bismut-Gillet-Soulé \cite{BGS3} and generalizes the definition of Quillen \cite{Quillen}, which was given for $n, m = 0$ and  $\pi$, $(\xi, h^{\xi})$ trivial.
	\par
	Let's denote by $\norm{\cdot}^W_{X/S}$ the Wolpert norm on $\otimes_{i = 1}^{m} \sigma_i^{*}(\omega_{X / S})$ induced by $\norm{\cdot}_{X/S}^{\omega}$ (see Definition \ref{defn_wolpert_norm}).
	This norm tracks the change of Poincaré-compatible coordinates near the cusp.
	The necessary definitions for the following passage are given in Definitions \ref{defn_loglog_gr}, \ref{defn_crit_preloglog}, \ref{defn_logloggwth}.
		\begin{align}\label{suppos_s3}
		\begin{split}
			& \qquad \textit{We suppose  that   the Hermitian  norm  $\norm{\cdot}_{X/S}$ on $\omega_{X/S}(D)$  extends  continuously  over}
			\\
			&\textit{$X \setminus (\Sigma_{X/S} \cup |D_{X/S}|)$, has log-log growth with singularities along $\Sigma_{X/S} \cup |D_{X/S}|$, is good} 
			\\
			&\textit{in the sense of Mumford on $X$ with singularities along $\pi^{-1}(\Delta) \cup D_{X/S}$, and the coupling}
			\\
			&\textit{of \, \, $c_1(\omega_{X/S}(D), \norm{\cdot}_{X/S}^{2})$  \, \, with  \, \, two   \, \, continuous  \, \, vertical  \, \, vector  \, \, fields  \, \, over }
			\\
			&\textit{$X \setminus ( \Sigma_{X/S} \cup |D_{X/S}| )$ is continuous over $X \setminus ( \Sigma_{X/S} \cup |D_{X/S}| )$.}
		\end{split}
	\end{align}
	Then in \cite[Theorem C3]{FinII2} (cf. Theorem \ref{thm_cont}) we proved that under assumption (\ref{suppos_s3}), the norm 
		\begin{equation}\label{quil_wol_norm}
			\norm{\cdot}_{\mathscr{L}_n} := \big( \norm{\cdot}_Q (g^{TX_t}, h^{\xi} \otimes \, \norm{\cdot}_{X/S}^{2n}) \big)^{12} \otimes \big( \norm{\cdot}^W_{X/S}\big)^{-\rk{\xi}} \otimes \big( \norm{\cdot}^{\rm{div}}_{\Delta} \big)^{\rk{\xi}} \otimes (\otimes_{i=1}^{m} \sigma_i^* h^{\det \xi})^3
		\end{equation}				
		on the line bundle
		\begin{equation}\label{det_wol_prod}
			\mathscr{L}_n :=  \lambda \big( j^* (\xi  \otimes \omega_{X/S}(D)^n) \big)^{12} \otimes (\otimes_{i = 1}^{m} \sigma_i^{*}\omega_{X/S} )^{-\rk{\xi}} \otimes \mathscr{O}_S(\Delta)^{\rk{\xi}} \otimes (\otimes_{i=1}^{m} \sigma_i^* \det \xi)^6
		\end{equation}
		extends continuously over $S$. 
		The main goal of this article is to give the precise value of this continuous extension, and give a geometric interpretation of it as the Quillen norm of the normalization of a singular fiber.
		\par More precisely, as $\Delta$ has normal crossings, for any $t \in S$, by shrinking the base $S$, we may always suppose that for some $l \in \nat$, the divisor $\Delta$ decomposes in the neighborhood of $t$ as
		\begin{equation}\label{eq_div_delta_dec}
			\Delta = k \cdot \Delta_0 + k_1 \cdot \Delta_1 + \cdots + k_l \cdot \Delta_l,
		\end{equation}
		where $\Delta_i$, $i = 0, \ldots, l$ are divisors induced by the submanifolds $|\Delta_i|$ and $k, k_j \in \nat^*$, $j = 1,\ldots, l$.
		Let  $\Delta_j^{0} :=\Delta_j \cap \Delta_0$ be the induced divisor on $S' := |\Delta_0|$, and let $\Delta'$ be the divisor on $S'$ given by 
		\begin{equation}\label{eq_defn_delta_prime}
			\Delta' := k_1 \cdot \Delta_1^{0} + \cdots + k_l \cdot \Delta_l^{0}.
		\end{equation}
		Let $\iota: S' \to S$ be the obvious inclusion.
		We denote $Z := \pi^{-1}(S')$, $Z_t := \pi^{-1}(t)$, $t \in S'$, and let $\rho : Y \to Z$ be the normalization of $Z$. 
		We denote by $\pi' : Y \to S'$ the family of surfaces, induced by the following commutative square
	\begin{equation}\label{eq_sh_exct_seq}
		\begin{CD}
			Y @> \rho >> X 
			\\
			@VV \pi' V @VV \pi V
			\\
			S' @> \iota >> S
		\end{CD}
	\end{equation}
	The restriction of the holomorphic sections $\sigma_1, \ldots, \sigma_m$ on $S'$ induce the holomorphic sections, which we denote by $\sigma'_1, \ldots, \sigma'_m : S' \to Y$. See Figure 1 for an example.
	\begin{figure}[h]\label{fig_graft}
		\centering
		\includegraphics[width=\textwidth]{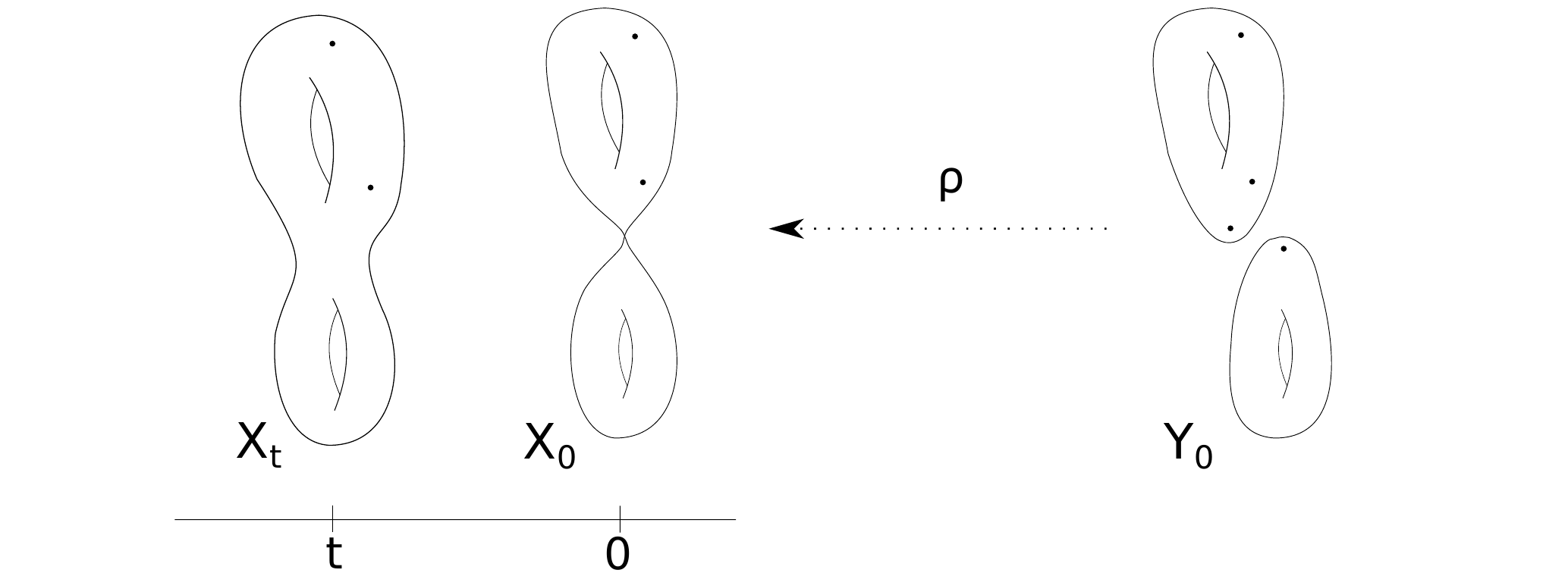}	
		\caption{A degenerating familiy. Our goal is to relate a norm on the line bundle at the singular fiber with a norm on the line bundle on its normalization. Points represent the elements in $D_{X/S}|_{X_t}$, $D_{X/S}|_{X_0}$ and $D_{Y/S'}|_{Y_0}$. Notice the added marked points on the normalization. 
		}
	\end{figure}
	\par Let $\Sigma_{Z/S'}$ be the locus of points, which get normalized in $\rho$. The manifold $\Sigma_{Z/S'}$ is a union of some connected components of $\Sigma_{X/S}$, thus, it has codimension $2$ in $X$ (see Corollary \ref{cor_sigma}). Let
	\begin{equation}
		\kappa : \Sigma_{Z/S'} \hookrightarrow X
	\end{equation}
	the obvious inclusion.
	Then the restriction of $\pi'$ to $\rho^{-1}(\kappa(\Sigma_{Z/S'}))$ is the covering of degree $2k$, see (\ref{eq_div_delta_dec}). 
	By shrinking the base, we may suppose that it is a trivial cover, so there are holomoprhic sections $\sigma'_{m+1}, \ldots, \sigma'_{m + 2k} : S' \to Y$ such that $\rho^{-1}(\Sigma_{Z/S'}) = \cup_{i = 1}^{2k} \Im(\sigma'_{m+i})$ and $\rho \circ \sigma'_{m+2i - 1} = \rho \circ \sigma'_{m+2i}$, $i = 1, \ldots, k$.
	We define the divisor $D_{Y/S'}$ over $Y$ by
	\begin{equation}
		D_{Y/S'} : = \Im(\sigma_1') + \cdots  + \Im(\sigma'_{m + 2k}).
	\end{equation}
	We also define the \textit{twisted canonical line bundle} of $\pi'$ as follows
	\begin{equation}
		\omega_{Y/S'}(D) := \omega_{Y/S'} \otimes \mathscr{O}_{Y}(D_{Y/S'}).
	\end{equation}
	 Then, classically (cf. Section \ref{sect_recall_relantors}), we have the canonical isomorphism
	 \begin{equation}\label{isom_pull_back_twisted}
	 	\rho^* (\omega_{X/S}(D)) \simeq \omega_{Y/S'}(D).
	 \end{equation}
	 Under assumptions (\ref{suppos_s3}), the isomorphism (\ref{isom_pull_back_twisted}) induces the Hermitian norm $\norm{\cdot}_{Y/S'}$ on $\omega_{Y/S'}(D)$ over $Y \setminus |D_{Y/S'}|$ by 
	 \begin{equation}
	 	\norm{\cdot}_{Y/S'} := \rho^*(\, \norm{\cdot}_{X/S}).
	 \end{equation}
	 Let $\norm{\cdot}_{Y/S'}^{\omega}$ be the norm on $\omega_{Y/S'}$, which is induced by $\norm{\cdot}_{Y/S'}$ as in Construction \ref{const_norm_div}. 
	 Then $\norm{\cdot}_{Y/S'}$ and $\norm{\cdot}_{Y/S'}^{\omega}$ are Hermitian norms over $Y \setminus ( (\pi')^{-1}(|\Delta'|) \cup | D_{Y/S'} | )$.
	 \par 
	 \noindent
	\begin{align}\label{suppos_restr}
		\begin{split}
			& \qquad \textit{We suppose that the norm $\norm{\cdot}_{Y/S'}^{\omega}$ over $Y \setminus (\pi^{-1}(|\Delta'|) \cup |D_{Y/S'}|)$ is such that its re-}
			\\
			&\textit{striction over each nonsingular fiber $Y_t := \pi^{-1}(t)$, $t \in S' \setminus |\Delta'|$ of $\pi'$ induces the Kähler}
			\\
			&\textit{metric $g^{TY_t}$, for which the triple $(Y_t, \{ \sigma_1'(t), \ldots, \sigma_{m+2k}'(t) \}, g^{TY_t})$ is a surface with cusps.}
		\end{split}
	\end{align}
	\par Essentially the assumption (\ref{suppos_restr}) says that the cusps on the normalization of the singular fibers are produced either by the extension of the existing cusps or by degeneration.
	\par We denote by $\norm{\cdot}_{Y/S'}^{W}$ the Wolpert norm on $\otimes_{i = 1}^{m + 2k} (\sigma'_i)^* \omega_{Y/S'}$, induced by $\norm{\cdot}_{Y/S'}^{\omega}$ (which is well-defined by the assumption \ref{suppos_restr}).
	Now, similarly to (\ref{quil_wol_norm}), (\ref{det_wol_prod}), we define the norm 
		\begin{multline}\label{quil_wol_norm_rest}
			\norm{\cdot}_{\mathscr{L}'_n} := \big( \norm{\cdot}_Q (g^{TY_t}, \rho^*(h^{\xi}) \otimes \, \norm{\cdot}_{Y/S'}^{2n}) \big)^{12} \otimes \big( \norm{\cdot}^W_{Y/S'} \big)^{-\rk{\xi}} \\
			\otimes \big( \norm{\cdot}^{\rm{div}}_{\Delta'} \big)^{\rk{\xi}}
			\otimes \big( \otimes_{i=1}^{m+2k} (\sigma_i' \circ \rho)^* h^{\det \xi} \big)^3
		\end{multline}				
		on the line bundle
		\begin{multline}\label{det_wol_prod_rest}
			\mathscr{L}'_n :=  \lambda \big(j^* ( \rho^*(\xi)  \otimes \omega_{Y/S'}(D)^n) \big)^{12} \otimes 
			(\otimes_{i = 1}^{m + 2k} (\sigma'_i)^* \omega_{Y/S'} )^{-\rk{\xi}} 
			\\
			\otimes \mathscr{O}_{S'}(\Delta')^{\rk{\xi}}
			\otimes \big(\otimes_{i=1}^{m+2k} ( \sigma_i' \circ \rho)^* \det \xi \big)^6.
		\end{multline}
		\par 
		We will now relate the restriction of $\mathscr{L}_n$ to $S'$ with $\mathscr{L}'_n$.
		For this, we denote by $N_{\Sigma_{Z/S'}/X}$ (resp. $N_{S'/S}$) the normal vector bundle  of $\Sigma_{Z/S'}$ in $X$ (resp. of $S'$ in $S$).
		Since the fibers of $X$ have only double-point singularities, the projection $\pi$ induces the canonical isomorphism (see (\ref{eq_pr_sing}), cf. also \cite[(2.9)]{BisDegQuil})
		\begin{equation}\label{isom_normal}
			d \pi^2 : \wedge^2 (N_{\Sigma_{Z/S'} / X}) \otimes ( \det \rho_* ( \mathscr{O}_{\rho^{-1} \Sigma_{Z/S'}})) \to \kappa^* \pi^* N_{S'/S}.
		\end{equation}
		Also, for the relative tangent bundle $TY/S'$ of $\pi'$ and for any $i = 1, \ldots, k$, the normalization map $\rho$ induces the canonical isomorphism
		\begin{equation}\label{isom_pt_rho}
			(\sigma'_{m + 2i - 1})^* (TY/S') \otimes (\sigma'_{m + 2i})^* (TY/S')  \to \wedge^2 (N_{\Sigma_{Z/S'} / X}).
		\end{equation}
		We denote by $\omega_S$ and $\omega_{S'}$ the canonical line bundles over $S$ and $S'$.
		By combining the duals of the isomorphisms (\ref{isom_normal}), (\ref{isom_pt_rho}), for $i = 1, \ldots, k$, we get the canonical isomorphism
		\begin{equation}\label{isom_pi2}
			(\omega_S \otimes \omega_{S'}^{-1})|_{S'} \to (\sigma_{m + 2i - 1}')^* (\omega_{Y/S'}) \otimes (\sigma_{m + 2i}')^* (\omega_{Y/S'}).
		\end{equation}
		The following isomorphism is given by Poincaré residue morphism (cf. \cite[p. 147]{GrifHar})
		\begin{equation}\label{isom_conj}
			(\omega_S^{k} \otimes \mathscr{O}_S(k \Delta_0))|_{S'} \to \omega_{S'}^{k}.
		\end{equation}
		By combining the isomoprhism (\ref{isom_pi2}), applied for each $i = 1, \ldots, k$, the isomorphism (\ref{isom_conj}) and by multiplying by $(\otimes_{i = 1}^{m} \sigma_i^{*}\omega_{X/S} )^{-1} \otimes \mathscr{O}_S(\sum k_i \Delta_i)$, we get the canonical isomorphism
		\begin{equation}\label{isom_rest_w_div}
			\Big( \big(\otimes_{i = 1}^{m} \sigma_i^{*}\omega_{X/S} \big)^{-1} \otimes \mathscr{O}_S(\Delta) \Big) \big|_{S'}
			\to
			\big(\otimes_{i = 1}^{m + 2k} (\sigma'_i)^* \omega_{Y/S'} \big)^{-1} \otimes \mathscr{O}_{S'}(\Delta').
		\end{equation}
		\par For $t \in S'$, we have the following exact sequence of sheaves (cf. \cite[(5.53)]{BisDegQuil})
		\begin{multline}\label{eq_sh_ex_seq}
			0 
			\rightarrow 
			\mathscr{O}_{Z_t} \big( j^*(\xi \otimes \omega_{X/S}(D)^n) \big) 
			\rightarrow 
			\rho_* \mathscr{O}_{Y_t} \big( j^*(\rho^*(\xi) \otimes \omega_{Y/S'}(D)^n) \big) 
			\\
			\rightarrow 
			\mathscr{O}_{\Sigma_{Z/S'}} \big( \kappa^* \xi \otimes \det( \rho_* \mathscr{O}_{\rho^{-1} \Sigma_{Z/S'}}) \big)
			\rightarrow 
			0,
		\end{multline}
		where the first map is induced by the pull-back and (\ref{isom_pull_back_twisted}), and the second map is the difference of residue morphism at $\rho^{-1}(\Sigma_{Z/S'})$.
		The short exact sequence (\ref{eq_sh_ex_seq}) of the line bundles over $S'$ induces the canonical isomorphism (cf. \cite[(5.55)]{BisDegQuil})
		\begin{multline}\label{isom_det_restr}
			\lambda \big( j^* (\xi  \otimes \omega_{X/S}(D)^n) \big)|_{S'}
			\to 
			\lambda \big(j^* ( \rho^*(\xi)  \otimes \omega_{Y/S'}(D)^n) \big)
			\\
			\otimes
			\det \big( \pi_*(\kappa^*(\xi)) \big)
			\otimes
			\det \big((\pi \circ \rho)_* \mathscr{O}_{\rho^{-1} \Sigma_{Z/S'}} \big)^{\rk{\xi}}.
		\end{multline}
		We note that the square of $\det ((\pi \circ \rho)_* \mathscr{O}_{\rho^{-1} \Sigma_{Z/S'}})$ is canonically trivialized, so from now on, we don't mention those powers explicitly.
		Trivially, we have an isomorphism
		\begin{equation}\label{isom_kappa_sigma}
			\det \big( \pi_*(\kappa^*(\xi)) \big)^2 \to 
			\big( \otimes_{i = 1}^{2k} (\sigma_{m + i}' \circ \rho)^* \det \xi \big)
			\otimes 
			\big( \det \pi_* \mathscr{O}_{\Sigma_{Z/S'}} \big)^{2 \cdot \rk{\xi}}.
		\end{equation}
		The composition of the isomorphisms (\ref{isom_rest_w_div}), (\ref{isom_det_restr}) and (\ref{isom_kappa_sigma}) induce the canonical isomorphism
		\begin{equation}\label{isom_main}
			\mathscr{L}_n|_{S'} \to \mathscr{L}'_n 
			\otimes 
			\big( \det \pi_* \mathscr{O}_{\Sigma_{Z/S'}} \big)^{12 \cdot \rk{\xi}},
		\end{equation}
		which is the protagonist of this paper.
		\par For $k \in \nat^*$, we define
		\begin{equation}\label{defn_C_k}
			C_0 = -6\log(\pi), \qquad
			C_k = -6(1+k) \log(2) - 6(1 + 2k) \log(\pi) - 6\log((2k)!).
		\end{equation}
		\par 
		Now we can state the main result of this article, which describes the continuous extension of the norm (\ref{quil_wol_norm}) in terms of the same objects that we've used in the definition of (\ref{quil_wol_norm}).
		\begin{thm}[Restriction theorem]\label{thm_imm_thm}
			Let $\pi: X \to S$ be a proper, holomorphic, surjective map of complex manifolds, such that for every $t \in S$, the space $X_t := \pi^{-1}(t)$ is a complex curve whose singularities are at most ordinary double points. 
			We suppose that the divisor of singular curves $\Delta$ decomposes as in (\ref{eq_div_delta_dec}). We use the notation as in (\ref{eq_defn_delta_prime}) for $k \in \nat$ and $S'$.
			\par 
			Let $\sigma_1, \ldots, \sigma_m : S \to X$ be disjoint holomorphic sections of $\pi$, which do not pass through singular points of the fibers.
			We denote by $D_{X/S}$ the divisor (\ref{defn_dxs}).
			\par Let $\norm{\cdot}_{X/S}^{\omega}$ be a Hermitian norm on the canonical line bundle $\omega_{X/S}$ (see (\ref{eq_rel_can_defn})) over $X \setminus (\pi^{-1}(|\Delta|) \cup |D_{X/S}|)$ such that its restriction over each $X_t := \pi^{-1}(t)$, $t \in S \setminus |\Delta|$ induces the Kähler metric $g^{TX_t}$ on $X_t \setminus \{ \sigma_1(t), \ldots, \sigma_m(t) \}$ such that the triple $(X_t, \{ \sigma_1(t), \ldots, \sigma_m(t) \}, g^{TX_t})$ is a surface with cusps in the sense of \cite{FinII1}, \cite{MullerCusp}, \cite{FreixasARR} (see Section \ref{sect_recall_relantors}).  
			\par 	Let $(\xi, h^{\xi})$ be a holomorphic Hermitian vector bundle over $X$.
			We define $\norm{\cdot}_{\mathscr{L}_n}$, $\mathscr{L}_n$ as in (\ref{quil_wol_norm}) and (\ref{det_wol_prod}).
			Let let family of complex curves $\pi' : Y \to S'$ be constructed as in (\ref{eq_sh_exct_seq}). 
			We suppose that assumptions (\ref{suppos_s3}) and (\ref{suppos_restr}) hold.
			We define $\norm{\cdot}_{\mathscr{L}'_n}$, $\mathscr{L}'_n$ as in (\ref{quil_wol_norm_rest}) and (\ref{det_wol_prod_rest}).
			\par Then the norm $\norm{\cdot}_{\mathscr{L}_n}$ extends continuously over $S$ and under the isomorphism (\ref{isom_main}), we have
			\begin{equation}
					\norm{\cdot}_{\mathscr{L}_n}|_{S'} = \exp(k \cdot \rk{\xi} \cdot C_{-n}) \cdot \norm{\cdot}_{\mathscr{L}'_n}.
			\end{equation}
		\end{thm}
		\begin{rem}
			We note that a similar theorem was proved by Bismut in \cite[Theorems 0.2, 0.3]{BisDegQuil} (cf Theorem \ref{thm_bismut_rest}). Let's comment on the differences between Theorem \ref{thm_imm_thm} and \cite[Theorems 0.2, 0.3]{BisDegQuil}. 
			\par First of all, Theorem \ref{thm_imm_thm} is for codimension $1$ family, where fibers are endowed with metric with cusps singularities, and the result \cite[Theorems 0.2]{BisDegQuil} works for families of any codimension but with smooth metric.
			Also, the way we endow singular fibers with the metric is crucially different. 
			In our case even when the general fiber has no cusps, the metric on the normalization of the singular fiber acquires at least two cusps. This is different from \cite[Theorems 0.2, 0.3]{BisDegQuil}, where author induce smooth metric on the normalization of the singular fiber. In particular, Theorem \ref{thm_imm_thm} doesn't follow directly from \cite[Theorems 0.2, 0.3]{BisDegQuil} and anomaly formula.
		\end{rem}
		\par Now, let's describe our second result. We fix a compact Riemann surface $\overline{M}$ and a set of points $D_M \subset \overline{M}$, $\#D_M = m$, $m < + \infty$. We denote $M := \overline{M} \setminus D_M$. Suppose that a pointed Riemann surface $(\overline{M}, D_M)$ is stable, i.e. the genus $g(\overline{M})$ of $\overline{M}$ satisfies 
	\begin{equation}\label{cond_stable}
		2 g(\overline{M}) - 2 + m > 0, 
	\end{equation}
	 then, by the uniformization theorem (cf. \cite[Chapter IV]{FarKra}, \cite[Lemma 6.2]{Auvr}), there is the \textit{canonical hyperbolic metric} $g^{TM}_{\rm{hyp}}$ of constant scalar curvature $-1$ on $M$ with cusps at $D_M$. 
	 We denote by $\norm{\cdot}_{M}^{\rm{hyp}}$ the norm induced by $g^{TM}_{\rm{hyp}}$ on $\omega_M(D)$ over $M$. 
	 Then, as we explain in \cite[\S 2.1]{FinII1}, the triple $(\overline{M}, D_M, g^{TM}_{\rm{hyp}})$ is a surface with cusps (see Section \ref{sect_recall_relantors}), thus, the analytic torsion $T(g^{TM}_{\rm{hyp}}, (\, \norm{\cdot}_{M}^{\rm{hyp}})^{2n})$ is well-defined in this case.
	 \par Alternatively, we denote by $Z_{(\overline{M}, D_M)}(s), s \in \comp$ the Selberg zeta-function, which is given for $\Re(s) > 1$ by the absolutely converging product:
	\begin{equation}\label{defn_sel_zeta}
		Z_{(\overline{M}, D_M)}(s) = \prod_{\gamma} \prod_{k=0}^{\infty} (1 - e^{-(s + k)l(\gamma)})^{2},
	\end{equation}
	where $\gamma$ runs over the set of all simple non-oriented closed geodesics on ($M, g^{TM}_{\rm{hyp}})$, and $l(\gamma)$ is the length of $\gamma$.
	The function $Z_{(\overline{M}, D_M)}(s)$ admits a meromorphic extension to the whole complex $s$-plane with a simple zero at $s = 1$ (see for example \cite[(5.3)]{PhongHook}).
	\par 
	Let $\zeta(s) := \sum_{k=1}^{\infty} k^{-s}$ be the Riemann zeta function. For $k \in \nat^*$, we put
	\begin{equation}\label{defn_c_k_coeff}
	\begin{aligned}
		& \textstyle c_0 = 4 \zeta'(-1) - \frac{1}{2} + \log(2 \pi),
		\\
		&
		\textstyle c_k = \sum_{l = 0}^{k - 1} (2k - 2l - 1) \big( \log (2k + 2kl - l^2 - l) - \log (2) \big) + 
		(\frac{1}{3} + k + k^2 ) \log(2) 
		\\		
		& \qquad \qquad \quad
		\textstyle + (2k + 1)\log (2 \pi) + 4 \zeta'(-1) 
		- 2(k + \frac{1}{2})^2 
		- 4 \sum_{l = 1}^{k - 1} \log(l!) - 2 \log (k!).
	\end{aligned}
	\end{equation}
	For $k \in \nat$, we denote by $B_k : \nat^2 \to \real$, $E : \nat^2 \to \real$ the following functions
	\begin{equation}
	\begin{aligned}
		& B_k(g, m) =  \exp \Big(  (2 - 2g( \overline{M}) - m)\frac{c_k}{2} \Big),
		\\
		& E(g, m) = \exp \Big( (g( \overline{M}) + 2 - m) \frac{\log(2)}{3} \Big).
	\end{aligned}
	\end{equation}
	In particular, we see that for any $k \in \nat$, $(g, m) \in \nat^2$, we have
	\begin{equation}\label{eq_b_k_dep_g_m_11}
	\begin{aligned}
		& B_k(g + m, 0) = B_k(g, m) \cdot B_k(1, 1)^m,
		\\
		& E(g + m, 0) = E(g, m) \cdot E(1, 1)^m.
	\end{aligned}
	\end{equation}
	\par Then, in case for hyperbolic surfaces and $(\xi, h^{\xi})$ trivial, for $l \in \integ$, $l < 0$, Takhtajan-Zograf in \cite[(6)]{TakZog} proposed\footnote{The constant in front of Selberg zeta function didn't appear in \cite{TakZog}, as their result is independent of it.} the analogue of the analytic torsion defined via Selbrerg zeta function as
	\begin{equation}\label{eqn_sel_norm}
	\begin{aligned}
		&
		T_{TZ}(g^{TM}_{\rm{hyp}}, 1) 
		= 
		E(g(\overline{M}), m) \cdot B_0 (g(\overline{M}), m) \cdot Z_{(\overline{M}, D_M)}'(1),
		\\
		&
		T_{TZ}(g^{TM}_{\rm{hyp}}, (\, \norm{\cdot}_{M}^{\rm{hyp}})^{2l})
		=
		B_{-l} (g(\overline{M}), m) \cdot Z_{(\overline{M}, D_M)}(-l + 1).
	\end{aligned}
	\end{equation}
	\begin{thm}[Compatibility theorem]\label{thm_compat}
		For any surface with cusps $(\overline{M}, D_M, g^{TM}_{\rm{hyp}})$, for which $g^{TM}_{\rm{hyp}}$ has constant scalar curvature $-1$, the following identity holds
		\begin{equation}\label{eq_thm_compat}
			T(g^{TM}_{\rm{hyp}}, (\, \norm{\cdot}_{M}^{\rm{hyp}})^{2n}) 
			= 
			T_{TZ}(g^{TM}_{\rm{hyp}}, (\, \norm{\cdot}_{M}^{\rm{hyp}})^{2n}).
		\end{equation}
		Thus, our definition of the analytic torsion is compatible with the definition of Takhtajan-Zograf.
	\end{thm}
	 \begin{rem}\label{rem_thm_compat}
	 	For $m=0$, i.e. when surfaces have no cusps, Theorem \ref{thm_compat} was shown by Phong-D'Hoker \cite[(7.30)]{PhongHook}, \cite[(3.6)]{PhongHook2nd} (see also \cite{sarnakDet}, \cite[(50)]{BolStei} and \cite[(9)]{Oshima}). Our proof is based on their result. We note that Albin-Rochon in \cite{AlbRoch} proved (\ref{eq_thm_compat}) up to a universal constant (see also \cite[(2.43)]{FinII2}). 
	 	Our approach to (\ref{eq_thm_compat}) is  based on degenerating families, which is different from the one of Albin-Rochon.
	\end{rem}
	Now let's describe the applications of Theorems \ref{thm_imm_thm}, \ref{thm_compat} in the study of the moduli space $\modul_{g, m}$ of $m$-pointed Riemann surfaces of genus $g \in \nat$, $2g - 2 + m > 0$. We denote by $\modulcomp_{g, m}$ the \textit{Deligne-Mumford compactification} of $\modul_{g, m}$, by $\partial \modul_{g, m} := \modulcomp_{g, m} \setminus \modul_{g, m}$ the \textit{compactifying divisor}, by $\univcurv_{g, m}$ and $\univcurvcomp_{g, m}$ the universal curves over $\modul_{g, m}$ and $\modulcomp_{g, m}$ respectively. 
	We denote by $\Pi : \univcurvcomp_{g, m} \to \modulcomp_{g, m}$ the \textit{universal projection}. We denote by $D_{g, m}$ the divisor on $\univcurvcomp_{g, m}$, formed by $m$ fixed points. We denote by $\omega_{g, m}$ the relative canonical line bundle of $\Pi$, by $\otimes_{i = 1}^{m} \sigma_i^{*} \omega_{g, m}$ the determinant of the restriction of $\omega_{g, m}$ to the divisor $D_{g, m}$, and by $\omega_{g, m}(D)$ the twisted relative canonical line bundle,
	\begin{equation}\label{def_rel_can_modul}
		\omega_{g, m}(D) := \omega_{g, m} \otimes \mathscr{O}_{\univcurvcomp_{g, m}}(D_{g, m}).
	\end{equation}
	\par By the uniformization theorem (cf. \cite[Chapter IV]{FarKra}, \cite[Lemma 6.2]{Auvr}, \cite{AuvrMaArx}), we endow $\omega_{g, m}(D)$ with the Hermitian norm $\norm{\cdot}_{g, m}^{\rm{hyp}}$, such that its restriction over each fiber induces the canonical hyperbolic metric of constant scalar curvature $-1$ by Construction \ref{const_norm_div}. This endows the determinant line bundle $\lambda(j^*(\omega_{g, m}(D)^n))$, $n \leq 0$, (\ref{defn_det_line}), which is usually called the \textit{Hodge line bundle}, with the induced Quillen metric $\norm{\cdot}^{Q, n}_{g, m}$.
	We endow the line bundle $\otimes_{i = 1}^{m} \sigma_i^{*} \omega_{g, m}$ with the associated Wolpert norm $\norm{\cdot}^W_{g, m}$, see Wolpert \cite[Definition 1]{Wol07} (cf. \cite[Definition 1.5]{FinII1} or Definition \ref{defn_wolpert_norm}).
	\par We recall that we proved in {\cite[Corollary 1.11]{FinII2}} (cf. Theorem \ref{thm_cont}) that the norm 
		\begin{equation}\label{eq_renorm_hodge_norm}
			\norm{\cdot}^{H, n}_{g, m} :=
			(\, \norm{\cdot}^{Q, n}_{g, m})^{12} 
			\otimes
			(\, \norm{\cdot}^W_{g, m})^{-1} 
			\otimes 
			\, \norm{\cdot}_{\partial \modul_{g,m}}^{\rm{div}}
		\end{equation}
		on the line bundle
		\begin{equation}\label{eq_renorm_hodge}
			\lambda_{g, m}^{H, n} :=
			\lambda(j^*(\omega_{g, m}(D)^n))^{12} 
			\otimes 
			(\otimes \sigma_i^{*} \omega_{g, m})^{-1} 
			\otimes 
			\mathscr{O}_{\modulcomp_{g,m}}(\partial \modul_{g,m})
		\end{equation}
		extends continuously over $\modulcomp_{g, m}$. Also, the norm $\norm{\cdot}^{H, n}_{g, m}$ is smooth over $\modul_{g, m}$, as it can be seen, for example, from the curvature theorem (cf. \cite[Theorem D]{FinII2}) and the smoothness of the Weil-Petersson metric.
	\par 
	For the definition of the clutching morphisms
	\begin{equation}\label{eq_clutch_morph}
	\begin{aligned}
		& \alpha_{ij}: \modulcomp_{g-1, m+2} \to \modulcomp_{g, m}, \\
		& \beta_{(g_1, m_1), (g_2, m_2)}^{P}: \modulcomp_{g_1, m_1 + 1} \times \modulcomp_{g_2, m_2 + 1} \to \modulcomp_{g, m},
	\end{aligned}
	\end{equation}
	where $ i < j$, $i, j = 1, \ldots, m+2$; $m_1, m_2 \in \nat$, $g_1, g_2 \in \nat$, $m_1 + m_2 = m$,  $g_1 + g_2 = g$, $2g_1 + m_1 - 2 > 0$, $2g_2 + m_2 - 2 > 0$ and $P \in \{I, J \subset \{1, 2, \ldots, m\} : I \cap J = \emptyset, I \cup J =\{1, 2, \ldots, m\},  |I| = m_1, |J| = m_2 \}$, see Knudsen \cite{Knud2}.
	We recall that the compactifying divisor $\partial \modul_{g, m}$ can be described in terms of (\ref{eq_clutch_morph}) by (cf.  \cite[p.262]{Arb_v2})
	\begin{equation}
		\big| \partial \modul_{g, m} \big| = \Big( \cup \Im(\alpha_{ij}) \Big) \cup \Big( \cup \Im \big(\beta_{(g_1, m_1), (g_2, m_2)}^{P} \big) \Big).
	\end{equation}
	From now on and till the end of this article, for brevity, we drop the subscripts from $\alpha, \beta$.
	\par After an application of adjunction formula, which asserts the canonical triviality of the line bundle $\Pi_* ( \omega_{g,m}(D)|_{|D_{g, m}|} )$, the isomorphism (\ref{isom_main}) specifies in this case to the isomoprhisms
	\begin{align}\label{eq_alpha_isom}
		& \alpha^* \lambda_{g, m}^{H, n} \simeq \lambda_{g-1, m_2}^{H, n}, \\
		\label{eq_beta_isom}
		& \beta^* \lambda_{g, m}^{H, n} \simeq \lambda_{g_1, m_1 + 1}^{H, n} \boxtimes \lambda_{g_2, m_2 + 1}^{H, n},
	\end{align}
	which also remarkably respect the natural $\integ$-structure of the line bundles (\ref{eq_renorm_hodge}), coming from the arithmetic structure of the $\modulcomp_{g, m}$, as it was proved by Knudsen in \cite[Theorem 4.2]{Knud3} (cf. \cite{FreixasARR}).
	\begin{thm}[Restriction theorem on $\modulcomp_{g, m}$]\label{thm_imm_thm_modul}
		a) The isomorphism (\ref{eq_alpha_isom}) is an isometry if the left-hand side is endowed with $\norm{\cdot}^{H, n}_{g, m}$, and the right-hand side with $\exp(m \cdot C_{-n}) \cdot \norm{\cdot}^{H, n}_{g-1, m+2}$.
		\par b) Similarly, the isomorphism (\ref{eq_beta_isom}) is an isometry if the left-hand side is endowed with $\norm{\cdot}^{H, n}_{g, m}$, and the right-hand side is endowed with the norm $\exp(m \cdot C_{-n}) \cdot (\, \norm{\cdot}^{H, n}_{g_1, m_1 + 1} \boxtimes \norm{\cdot}^{H, n}_{g_2, m_2 + 1})$.
	\end{thm}
	\begin{rem}
		For a special family of curves from Section \ref{sect_hyp_surf}, which is less general than Theorem \ref{thm_imm_thm_modul}b), Freixas proved Theorem \ref{thm_imm_thm_modul} for $n = 0$ in \cite[Corollary 5.8]{FreixasARR} and extended it for $n \leq 0$ in \cite[Theorem 5.3]{FreixasARR}.
		\par Instead of the usual definition Quillen norm, Freixas used its version, defined as a product of (\ref{eqn_sel_norm}) and the $L^2$-norm. By Theorem \ref{thm_compat}, our results are compatible.
		Note that to calculate the constant $C_{-n}$ from Theorem \ref{thm_imm_thm} we use those results of Freixas.
	\end{rem}
	\par Now let's state our final result which describes an explicit relation between the Quillen metric associated with a cusped metric and the Quillen metric associated with a metric on the compactified Riemann surface. This theorem should be regarded as a refinement of the relative compact perturbation theorem, \cite[Theorem A]{FinII1}.
	\begin{sloppypar}
		To state it precisely, let's define the \textit{regularized integral} on a surface with cusps.
		Let $(\overline{M}, D_M, g^{TM})$ be a surface with cusps. Let $\alpha \in \ccal^{\infty}(M, \wedge^2 T\overline{M})$. We suppose that for any $P_i \in D_M$, there are coordinates $z_i$ around $P_i \in D_M$, such that for some $\epsilon > 0$ small enough, there are $C \in \comp$, $l \in \nat$ such that the following estimate holds
		\begin{equation}
			\alpha|_{\{|z_i| < \epsilon\}} = - \frac{C \cdot \imun \cdot dz_i d\overline{z}_i}{|z_i|^{2} |\log |z_i||} + O \bigg(  \frac{\log |\log|z_i||^{l}  dz_i d\overline{z}_i}{|z_i \log |z_i||^2}  \bigg).
		\end{equation}
		We define $\int^{\textbf{r}}_{M} \alpha \in \comp$ by the following limit
		\begin{equation}\label{eq_defn_reg_integ}
			\int^{\textbf{r}}_{M} \alpha 
			= 
			\lim_{\epsilon \to 0} \bigg( \int_{M \setminus (\cup \{|z_i| < \epsilon \})} \alpha + 4 \pi \cdot C \cdot (\# D_M) \cdot \log |\log \epsilon| \bigg).
		\end{equation}	 
		In other words, $\int^{\textbf{r}}_{M} \alpha$ is the finite part of $\int_{M \setminus (\cup \{|z_i| < \epsilon \})} \alpha$, as $\epsilon \to 0$.
		It is an easy verification that $\int^{\textbf{r}}_{M} \alpha$ doesn't depend on the choice of the coordinates $z_i$.
	\end{sloppypar}
	\par 
	Let's recall that by \cite[Theorem 1.27]{BGS1}, the Bott-Chern classes of a vector bundle $\xi$ with Hermitian metrics $h^{\xi}_{1}$,  $h^{\xi}_{2}$ are natural differential forms, defined modulo $\Im(\partial) + \Im(\dbar)$, so that
		\begin{equation}\label{eq_der_tilde_tdch}
		\begin{aligned}
			& \frac{\partial \dbar}{2 \pi \imun} \widetilde{\td} (\xi, h^{\xi}_{1}, h^{\xi}_{2}) 
			&& = 
			\td (\xi, h^{\xi}_{1}) - 
			\td (\xi, h^{\xi}_{2}), \\
			& \frac{\partial \dbar}{2 \pi \imun} \widetilde{\ch} (\xi, h^{\xi}_{1}, h^{\xi}_{2}) 
			&& = 
			\ch (\xi, h^{\xi}_{1})  -
			\ch (\xi, h^{\xi}_{2}),
		\end{aligned}
		\end{equation}
		where $\td$, $\ch$ are Todd and Chern forms.
		By \cite[Theorem 1.27]{BGS1}, we have the following identities
		\begin{equation}
			\widetilde{\ch}(\xi, h^{\xi}_{1},  h^{\xi}_{2})^{[0]} = 2 \widetilde{\td}(\xi, h^{\xi}_{1},  h^{\xi}_{2})^{[0]} = \ln \big( \det( h_1^{\xi} /h_2^{\xi}) \big). \label{ch_bc_0}
		\end{equation}
		If, moreover, $\xi := L$ is a line bundle, we have
		\begin{equation}
			\widetilde{\ch}(L, h^L_{1}, h^L_{2})^{[2]} = 6 \widetilde{\td}(L, h^L_{1},  h^L_{2})^{[2]} 
			 =  \ln ( h^L_{1}/h^L_{2} ) \Big(c_1(L, h^L_{1})  + c_1(L, h^L_{2}) \Big) / 2, \label{ch_bc_2}
		\end{equation}
		where $c_1$ is the first Chern form.
		In what follows, when we write a Bott-Chern class, one should interpret it as a \textit{differential form}, given by (\ref{ch_bc_0}), (\ref{ch_bc_2}).
		\par	
		Now, for $k \in \nat$, we define
		\begin{equation}\label{eq_defn_e_k}
			E_k = 4 \zeta'(-1) - \log(2 \pi) + \frac{1 - C_{-n}}{6}.
		\end{equation}
	\begin{thm}[Compact perturbation theorem]\label{thm_full_comp_pert}
		Let $(\overline{M}, D_M, g^{TM})$ be a surface with cusps. 
		We denote by $\norm{\cdot}_M$ the induced metric on $\omega_M(D)$ over $M$ as in Construction \ref{const_norm_div}. 
		We denote by $\norm{\cdot}^W$ the Wolpert norm on $\otimes_{P \in D_M} \omega_{\overline{M}}|_{P}$ induced by $g^{TM}$.
		\par 
		Let $g^{T\overline{M}}$ be a Kähler metric over $\overline{M}$, and let $\norm{\cdot}_{\overline{M}}$ be some Hermitian metric on  $\omega_M(D)$ over $\overline{M}$.
		We denote by $\norm{\cdot}^{D_M}_{\overline{M}}$ the metric on $\otimes_{P \in D_M} \omega_{\overline{M}}|_{P}$ induced by $g^{T\overline{M}}$.
		Let $\xi$ be a holomorphic vector bundle over $\overline{M}$, and let $h^{\xi}$ and $h^{\xi}_{0}$ be two Hermitian metrics on $\xi$ over $\overline{M}$.
		\begin{equation}\label{eq_comp_pert}
			\begin{aligned}
				2 \ln & \Big(  
				\norm{\cdot}_{Q}  \big(g^{TM},  h^{\xi} \otimes \, \norm{\cdot}_{M}^{2n} \big) 
				\big/				 
				 \norm{\cdot}_{Q} \big(g^{T\overline{M}}, h^{\xi}_0 \otimes \, \norm{\cdot}_{\overline{M}}^{2n} \big) 
				 \Big) 
				\\
				&  = 
		 			\int_{M}^{\textbf{r}} 
		 			\Big[ 
	 					\widetilde{\td} \big(\omega_{\overline{M}}^{-1}, g^{T\overline{M}}, g^{TM} \big) \ch \big(\xi, h^{\xi}_0 \big)  \ch \big(\omega_M(D)^n, \norm{\cdot}_{\overline{M}}^{2n} \big)  \\
						&  \phantom{= \int_{M} 
		 			\Big[ } +	
		 			\td \big(\omega_{M}^{-1}, g^{TM} \big) \widetilde{\ch} \big(\xi, h^{\xi}_0, h^{\xi} \big)  \ch \big(\omega_M(D)^n, \norm{\cdot}_{\overline{M}}^{2n} \big)  \\
						&  \phantom{= \int_{M} 
		 			\Big[ } +			 		 					
			 			\td \big(\omega_{M}^{-1}, g^{TM} \big) \ch \big(\xi, h^{\xi} \big) \widetilde{\ch} \big(\omega_M(D)^n, \norm{\cdot}_{\overline{M}}^{2n}, \, \norm{\cdot}_{M}^{2n} \big) 									 				
			 		\Big]^{[2]} \\
			 		& \phantom{ = }  + \frac{\rk{\xi}}{6} \ln \Big( \norm{\cdot}^W / \norm{\cdot}^{D_M}_{\overline{M}} \Big)	
					-
			 		\frac{1}{2} \sum_{P \in D_M} \ln \Big(\det (h^{\xi}_0 / h^{\xi})|_{P} \Big)
			 		\\
			 		& \phantom{ = }  + 
			 		\Big( \#(D_M) \cdot \rk{\xi} \cdot E_{-n} \Big).
			 \end{aligned}
		\end{equation}
	\end{thm}
	\par Finally, we note that Theorem \ref{thm_imm_thm} suggest that the renormalization 
	\begin{equation}
		T^{{\rm{ren}}}(g^{TX_t}, h^{\xi} \otimes \, \norm{\cdot}_{X/S}^{2n}) := \exp(m \cdot \rk{\xi} C_{-n} / 12) \cdot T(g^{TX_t}, h^{\xi} \otimes \, \norm{\cdot}_{X/S}^{2n})
	\end{equation}	
	is more natural from the point of view of restriction theorem. For $(\xi, h^{\xi})$ trivial and $(M, g^{TM})$ stable hyperbolic surface, this coincides with the normalization of Freixas in \cite[Definition 2.2]{FreixasARR} and \cite[Definition 4.2]{FreixARRgen}.
	\par Let's describe the structure of this paper. In Section 2, we recall the definition of the Quillen norm on the family of Riemann surfaces with cusps, the definition of Wolpert norm, and some results from \cite{FinII1}, \cite{FinII2}, which study those norms.
	Then we recall an analogue of Theorem \ref{thm_imm_thm} due to Freixas about convergence of the Quillen norm for a special family of hyperbolic Riemann surfaces, where the Quillen norm is defined using (\ref{eqn_sel_norm}).
	In Section 3 we extend a result of Bismut \cite[Theorem 0.3]{BisDegQuil} to families endowed with non-Kähler metric and give a proof of Theorems \ref{thm_imm_thm}, \ref{thm_compat}, \ref{thm_imm_thm_modul}, \ref{thm_full_comp_pert}.
	\par \textbf{Notation.} For a complex manifold $X$, we denote by $\Omega_X$ the sheaf of holomorphic sections of the vector bundle $T^{*(1,0)}X$, and by $\omega_X$ the canonical line bundle $\det (T^{*(1,0)}X)$ of $X$. For a divisor $D$ in $X$, we denote by $s_{D}$ the canonical meromorphic section of $\mathscr{O}_X(D)$.
	\par	
	For $\epsilon > 0$, we define
	\begin{equation}
		D(\epsilon) = \{ u \in \comp : |u| < \epsilon \},  \quad D^*(\epsilon) = \{ u \in \comp : 0 < |u| < \epsilon \}.
	\end{equation}
	\par {\bf{Acknowledgments.}} This work is a part of our PhD thesis, which was done at Université Paris Diderot. We would like to express our deep gratitude to our PhD advisor Xiaonan Ma for his teaching, overall guidance, constant support and invaluable comments on the preliminary version of this article.

	\section{Families of nodal curves and hyperbolic metric on them}
	In this section we recall the relevant notations.
	More precisely, in Section 2.1, we recall the notion of the Quillen norm from \cite{FinII1}, \cite{BGS1}, \cite{BGS2}, \cite{BGS3}, the basic notions for families of Riemann surfaces with cusps from \cite{BisBost}, \cite{FinII2} and relevant results from \cite{FinII1}, \cite{FinII2}. 
	In Section 2.2, we recall several notions of singularities of Hermitian metrics on holomorphic line bundles and a useful regularity result for push-forward of differential form in a family of curves with double-point singularities, which we proved in \cite{FinII2}.

\subsection{Determinant line bundles, Serre duality and Quillen norms}\label{sect_recall_relantors}
	Let $\overline{M}$ be a compact Riemann surface, and let $D_M = \{ P_1^{M}, \ldots, P_m^{M} \}$ be a finite set of distinct points in $\overline{M}$. Let $g^{TM}$ be a Kähler metric on the punctured Riemann surface $M = \overline{M} \setminus D_M$. 
	\par For $\epsilon \in ]0,1[$, let $z_i^{M} : \overline{M} \supset V_i^M(\epsilon) \to D(\epsilon) = \{ z \in \comp : |z| \leq \epsilon \}$, $i = 1, \ldots, m$, be a local holomorphic coordinate around $P_i^{M}$, and 
	\begin{equation}\label{defn_v_i}
		V_i^{M}(\epsilon) := \{x \in M :  |z_i^{M}(x)| < \epsilon \}.
	\end{equation}	 
	We say that $g^{TM}$ is \textit{Poincaré-compatible} with coordinates $z_1^{M}, \ldots, z_m^{M}$ if for any $i = 1, \ldots, m$,  there is $\epsilon > 0$ such that $g^{TM}|_{V_i^{M}(\epsilon)}$ is induced by the Hermitian form
	\begin{equation}\label{reqr_poincare}
		\frac{\imun dz_i^{M} d\overline{z}_i^{M}}{ \big| z_i^{M}  \log |z_i^{M}| \big|^2}.
	\end{equation}
	We say that $g^{TM}$ is a \textit{metric with cusps} if it is Poincaré-compatible with some holomorphic coordinates of $D_M$.
	A triple $(\overline{M}, D_M, g^{TM})$ of a Riemann surface $\overline{M}$, a set of punctures $D_M$ and a metric with cusps $g^{TM}$ is called a \textit{surface with cusps} (cf. \cite{MullerCusp}). 
	\par From now on, we fix a surface with cusps $(\overline{M}, D_M, g^{TM})$ and a Hermitian vector bundle $(\xi, h^{\xi})$ over it.  
	We denote by $\omega_{M} := T^{*(1,0)}\overline{M}$ the \textit{canonical line bundle} over $\overline{M}$. 
	We denote by $\norm{\cdot}_{M}^{\omega}$ the norm on $\omega_{\overline{M}}$ induced by $g^{TM}$ over $M$ by the natural identification $TM \ni X \mapsto \frac{1}{2} (X - JX) \in T^{(1,0)}M$, where $J$ is the complex structure of $M$. 
	Let $\mathscr{O}_{\overline{M}}(D_M)$ be the line bundle associated with the divisor $D_M$. 	
	The \textit{twisted canonical line bundle} is defined as 
	\begin{equation}
			\omega_M(D) :=  \omega_{\overline{M}} \otimes  \mathscr{O}_{\overline{M}}(D_M).
	\end{equation}		
	The metric $g^{TM}$ endows by Construction \ref{const_norm_div} the line bundle $\omega_M(D)$ with the induced Hermitian metric $\norm{\cdot}_{M}$ over $M$.
	\par We recall here briefly the definition of the \textit{analytic torsion} $T(g^{TM}, h^{\xi} \otimes \norm{\cdot}_M^{2n})$ for $m \in \nat$ from \cite[Definition 2.17]{FinII1}.
	\par Assume first $m=0$, then the analytic torsion was defined by Ray-Singer \cite[Definition 1.2]{Ray73} as the regularized determinant of the Kodaira Laplacian $\laplcomp^{\xi \otimes \omega_M(D)^n}$ associated with $(M, g^{TM})$ and $(\xi \otimes \omega_M(D)^n, h^{\xi} \otimes \norm{\cdot}_M^{2n})$. 
	More precisely, let $\lambda_i, i \in \nat$ be the non-zero eigenvalues of $\laplcomp^{\xi \otimes \omega_M(D)^n}$. By Weyl's law, the associated zeta-function 
	\begin{equation}\label{defn_zeta_comp}
		\zeta_M(s) := \sum \lambda_i^{-s},
	\end{equation}
	 is defined for $s \in \comp$, $\Re (s) > 1$ and it is holomorphic in this region. Moreover, we have
	 \begin{equation}
	 		\zeta_{M}(s) = \frac{1}{\Gamma(s)} \int_{0}^{+ \infty}
			{\rm{Tr}}  \big[ \exp^{\perp}(-t  \laplcomp^{\xi \otimes \omega_M(D)^n}) \big] t^{s} \frac{dt}{t},
	 \end{equation}
	 where $\exp^{\perp}(-t  \laplcomp^{\xi \otimes \omega_M(D)^n}) $ is the spectral projection onto the eigenspace corresponding to non-zero eigenvalues.
	 Also, as it can be seen by the small-time expansion of the heat kernel and the usual properties of the Mellin transform, $\zeta_{M}(s)$ extends meromorphically to the entire $s$-plane. This extension is holomorphic at $0$, and the \textit{analytic torsion} is defined by
	\begin{equation}\label{defn_an_t_st}
		T(g^{TM}, h^{\xi} \otimes \norm{\cdot}_M^{2n}) := \exp(- \zeta'_M(0)).
	\end{equation}
	By (\ref{defn_zeta_comp}) and (\ref{defn_an_t_st}), we may interpret the analytic torsion as
	\begin{equation}\label{an_tors_intepr}
		T(g^{TM}, h^{\xi} \otimes \norm{\cdot}_M^{2n}) := \prod_{i=0}^{\infty} \lambda_i.
	\end{equation}
	\par 
	Now, assume $m > 0$. Then $M$ is non-compact, and the heat operator associated to $\laplcomp^{\xi \otimes \omega_M(D)^n}$ is no longer of trace class. Also the spectrum of $\laplcomp^{\xi \otimes \omega_M(D)^n}$ is not discrete in general. Thus, neither the definition (\ref{defn_an_t_st}), nor the interpretation (\ref{an_tors_intepr}) are applicable.
	\begin{sloppypar}
	 In \cite[Definition 2.10]{FinII1}, for $n \leq 0$, we defined the \textit{regularized heat trace}
	 	${\rm{Tr}}^{\reg} [ \exp^{\perp} ( -t  \laplcomp^{\xi \otimes \omega_M(D)^n} ]$
		as a “difference" of the heat trace of $\laplcomp^{\xi \otimes \omega_M(D)^n}$ and the heat trace of the Kodaira Laplacian  $\laplcomp^{\omega_P(D)^n}$ corresponding to the $3$-punctured projective plane $P := \overline{P} \setminus \{0, 1, \infty\}$, $\overline{P} := \mathbb{C}P^1$, endowed with the hyperbolic metric $g^{TP}$ of constant scalar curvature $-1$ and the induced metric $\norm{\cdot}_P$ on $\omega_P(D) := \omega_{\overline{P}} \otimes \mathscr{O}_{\overline{P}}(0 + 1 + \infty)$. 
		Then in \cite[Definition 2.16]{FinII1}, we defined the \textit{regularized spectral zeta function} $\zeta_{M}(s)$ for $s \in \comp$, $\Re(s) > 1$ by
		\begin{equation}\label{eq_defn_rel_zeta}
			\zeta_{M}(s) = \frac{1}{\Gamma(s)} \int_{0}^{+ \infty}
			{\rm{Tr}}^{\reg}  \big[ \exp^{\perp}(-t  \laplcomp^{E_M^{\xi, n}}) \big] t^{s} \frac{dt}{t}.
		\end{equation} 
	And we concluded in \cite[p. 17]{FinII1}, similarly to the case $m = 0$, the function $\zeta_{M}(s)$ extends meromorphically to $\comp$ and $0$ is a holomorphic point.
	Then in \cite[Definition 2.17]{FinII1}, we defined the regularized analytic torsion as 
	\begin{equation}\label{eq_defn_rel_tor}
			T (g^{TM}, h^{\xi} \otimes \norm{\cdot}_{M}^{2n}) := \exp(- \zeta_{M}'(0)/2) \cdot T_{TZ}(g^{TP}, \norm{\cdot}_{P}^{2n})^{m \cdot \rk{\xi} / 3}.
	\end{equation} 
	In other words, we defined the analytic torsion by subtracting the universal contribution of the cusp from the heat trace and by normalizing it in sch a way that it coincides with the analytic torsion of Takhtajan-Zograf for $\mathbb{C}P^1 \setminus \{ 0, 1, \infty \}$, endowed with the complete metric of constant scalar curvature $-1$.
	\end{sloppypar}
	\par Then for $n \leq 0$, in \cite[\S 2.1]{FinII1}, we explain how to endow the complex line 
	\begin{multline}\label{defn_det_line}
		\big(\det H^{\bullet}(\overline{M}, \xi \otimes \omega_M(D)^n) \big)^{-1} \\
		:= \big( \Lambda^{\max} H^{0}(\overline{M}, \xi \otimes \omega_M(D)^n) \big)^{-1} \otimes  \Lambda^{\max} H^{1}(\overline{M}, \xi \otimes \omega_M(D)^n),
	\end{multline}
	with the $L^2$-norm $\norm{\cdot}_{L^2}(g^{TM}, h^{\xi} \otimes \norm{\cdot}_M^{2n})$ induced by the $L^2$-scalar product (\ref{defn_L_2}).
	In the compact case it coincides with the $L^2$-norm induced on the harmonic forms.
	\par 
	The Quillen norm on the complex line is defined by
	\begin{equation}\label{defn_quil}
		 \norm{\cdot}_{Q}(g^{TM}, h^{\xi} \otimes \, \norm{\cdot}_{M}^{2n}) 
		=
		 T(g^{TM}, h^{\xi} \otimes \, \norm{\cdot}_{M}^{2n})^{1/2}
		 \cdot 
		 \norm{\cdot}_{L^2}(g^{TM}, h^{\xi} \otimes \, \norm{\cdot}_{M}^{2n}).
	\end{equation}
	To motivate, when $m = 0$, this coincides with the usual definition of the Quillen norm from Quillen \cite{Quillen}, Bismut-Gillet-Soulé \cite[(1.64)]{BGS1} and \cite[Definition 1.5]{BGS3}.
	\par Following \cite{FinII1}, we say that a (smooth) metric $g^{TM}_{\rm{f}}$ over $\overline{M}$ is a \textit{flattening} of $g^{TM}$ if there is $\nu > 0$ such that $g^{TM}$ is induced by (\ref{reqr_poincare}) over $V_i^{M}(\nu)$, and 
	\begin{equation}\label{fl_exterior}
		g^{TM}_{\rm{f}}|_{M \setminus (\cup_i V_i^{M}(\nu))} = g^{TM}|_{M \setminus (\cup_i V_i^{M}(\nu))}.
	\end{equation}
	Similarly, we defined a flattening $\norm{\cdot}_{M}^{\rm{f}}$ of the norm $\norm{\cdot}_M$.
	\begin{thm}[{\cite[Theorem A]{FinII1}}]\label{thm_comp_appr}
		Let $g^{TM}_{\rm{f}}$, $\norm{\cdot}_M^{\rm{f}}$ be flattenings of $g^{TM}$, $\norm{\cdot}_M$. Then
		\begin{multline}\label{eqn_of_quil_norms_no_flat}
		2 \rk{\xi}^{-1} \log \Big( 
			\norm{\cdot}_{Q} \big(g^{TM}, h^{\xi} \otimes \, \norm{\cdot}_{M}^{2n}\big) 
			\big / 
			\norm{\cdot}_{Q} \big(g^{TM}_{\rm{f}}, h^{\xi} \otimes (\, \norm{\cdot}_{M}^{\rm{f}})^{2n}\big)
			\Big)
			\\	
			-
			\rk{\xi}^{-1} \int_M c_1(\xi, h^{\xi}) \Big(2n \log (\, \norm{\cdot}_{M}^{\rm{f}}/ \norm{\cdot}_{M}) +  \log (g^{TM}_{\rm{f}} / g^{TM}) \Big)
	\end{multline}
	depends only on the integer $n \in \integ$, $n \leq 0$, the functions $( g^{TM}_{\rm{f}} / g^{TM} )|_{V_i^{M}(1)} \circ (z_i^{M})^{-1} : \dd^* \to \real $ and $(\, \norm{\cdot}_{M}^{\rm{f}}/ \norm{\cdot}_{M} )|_{V_i^{M}(1)}  \circ (z_i^{M})^{-1} : \dd^* \to \real$, for $i = 1, \ldots, m$.
	\end{thm}
	Now let's recall another natural norm associated with a surface with cusps
	\begin{defn}[{\cite[Definition 1.5]{FinII1}}]\label{defn_wolpert_norm}
		For a surface with cusps $(\overline{M}, D_M, g^{TM})$, the \textit{Wolpert norms} $\norm{\cdot}^{W, i}_{M}$ on the complex lines $\omega_{\overline{M}}|_{P_i^{M}}$, $i=1, \ldots, m$, are defined by $\| dz_i^{M} \|^{W, i}_{M} = 1$. 
		They induce the Wolpert norm $\norm{\cdot}^W_{M}$ on the complex line $\otimes_{i = 1}^{m} \omega_{\overline{M}}|_{P_i^{M}}$. 
	\end{defn}
	\begin{thm}[{\cite[Theorem B]{FinII1}}]\label{thm_anomaly_cusp}
		Let $\phi : M \to \real$ be a smooth function such that for the metric
		\begin{equation}\label{anomaly_rel_metrics}
			g^{TM}_{0} = e^{2 \phi} g^{TM},
		\end{equation}
		the triple $(\overline{M}, D_M, g^{TM}_{0})$ is a surface with cusps.
		We denote by $\norm{\cdot}_M, \norm{\cdot}_{M}^{0}$ the norms induced by $g^{TM}, g^{TM}_{0}$ on $\omega_M(D)$, and by $\norm{\cdot}^W_{M}$, $\norm{\cdot}^{W,0}_{M}$ the associated Wolpert norms. 
		Let $ h^{\xi}_0$ be a Hermitian metric on $\xi$ over $\overline{M}$.
		Then the right-hand side of the following equation is finite, and
		\begin{equation}\label{eq_anomaly_cusp}
			\begin{aligned}
				2 \log & \Big(  
				\norm{\cdot}_{Q}  \big(g^{TM}_{0},  h^{\xi}_{0} \otimes (\, \norm{\cdot}_{M}^{0})^{2n} \big) 
				\big/				 
				 \norm{\cdot}_{Q} \big(g^{TM}, h^{\xi} \otimes \, \norm{\cdot}_{M}^{2n} \big) 
				 \Big) 
				\\
				&  = 
		 			  \int_{M} 
		 			\Big[ 
	 					\widetilde{\td} \big(\omega_M(D)^{-1}, \, \norm{\cdot}^{-2}_{M}, (\, \norm{\cdot}^{0}_{M})^{-2} \big) \ch \big(\xi, h^{\xi} \big)  \ch \big(\omega_M(D)^n, \norm{\cdot}_M^{2n} \big)  \\
						&  \phantom{= \int_{M} 
		 			\Big[ } +	
		 			\td \big(\omega_M(D)^{-1}, (\, \norm{\cdot}^{0}_{M})^{-2} \big) \widetilde{\ch} \big(\xi, h^{\xi}, h^{\xi}_{0} \big)  \ch \big(\omega_M(D)^n, \norm{\cdot}_M^{2n} \big)  \\
						&  \phantom{= \int_{M} 
		 			\Big[ } +			 		 					
			 			\td \big(\omega_M(D)^{-1}, (\, \norm{\cdot}^{0}_{M})^{-2} \big) \ch \big(\xi, h^{\xi}_{0} \big) \widetilde{\ch} \big(\omega_M(D)^n, \norm{\cdot}_{M}^{2n}, (\, \norm{\cdot}_{M}^{0})^{2n} \big) 									 				
			 		\Big]^{[2]} \\
			 		& \phantom{ = }  - \frac{\rk{\xi}}{6} \log \Big( \norm{\cdot}^W_{M} / \norm{\cdot}^{W,0}_{M} \Big)	+ \frac{1}{2} \sum \log \Big(\det (h^{\xi} / h^{\xi}_{0})|_{P_i^{M}} \Big).
			 \end{aligned}
		\end{equation}
	\end{thm}
	Now let's pass tot the study of “singular Riemann surfaces".
	By a \textit{curve} in this article we mean (cf. \cite[Definition of nodal curve on p. 79]{Arb_v2}) an analytic space such that every one of its points is either smooth or is locally complex-analytically isomorphic to a neighborhood of the origin in $\{(z_0, z_1) \in \comp^2 :  z_0z_1 = 0 \}$.
	\par 
	Let $C$ be a curve with singularities $\Sigma_C \subset C$. Let $\rho : N \to C$ be the normalization of $C$. 
	For brevity, we denote the \textit{twisted relative canonical bundle} on $N$ by 
 \begin{equation}
 	 \omega_{N}(D) := \omega_{N} \otimes \mathscr{O}_{N}(\rho^{-1} \Sigma_{C}).
 \end{equation}
 We recall that for a point $p \in \rho^{-1}(\Sigma_C)$, and a local holomorphic coordinate $z$ of $p$, the Poincaré residue morphism $\res_p: (\omega_N \otimes \mathscr{O}_{N}(\rho^{-1} \Sigma_{C}))|_p \to \comp$, is defined by 
 \begin{equation}
 	\res_p \Big( \frac{dz \otimes s_{\rho^{-1} \Sigma_{C}}}{z} \Big) = 1,
 \end{equation}
 where $s_{\rho^{-1} \Sigma_{C}}$ is the canonical section of the divisor line bundle $\mathscr{O}_{N}(\rho^{-1} \Sigma_{C})$.
 \par  Now, recall that the canonical sheaf $\omega_{C}$ of $C$ is defined (cf. \cite[p.91]{Arb_v2}) as an invertible subsheaf
\begin{equation}
	\omega_{C} \subset \rho_*(\omega_{N}(D)),
\end{equation}
defined by the following prescription. 
A section $\upsilon$ of $\rho_*(\omega_{N}(D))$, viewed as a section of $\omega_{N}(D)$, is a section of $\omega_{C}$ if and only if for any $x, y \in N$, $x \neq y$ such that $\rho(x) = \rho(y)$, we have
\begin{equation}\label{eq_res_cond}
	\res_x(\upsilon) + \res_y(\upsilon) = 0.
\end{equation}
We denote by
$
	\ccal^{\infty}_{\res} (N, \omega_{N}(D)^n)
$
the set of smooth sections $\upsilon$ of $\omega_{N}(D)^n$ over $N$ such that for any $x, y \in N$, $x \neq y$, $\rho(x) = \rho(y)$, we have
\begin{equation}\label{eq_res_cond2}
	\res_x(\upsilon) = (-1)^n \res_y(\upsilon).
\end{equation}
\par By definition, we have the short exact sequence of sheaves
		\begin{equation}\label{eq_sh_ex_seq_wc}
			0 
			\rightarrow 
			\omega_{C}
			\rightarrow 
			\rho_*(\omega_{N}(D))
			\xrightarrow{\res}
			\oplus_{p \in \Sigma_C} \mathscr{O}_p
			\rightarrow 
			0,
		\end{equation}
		where the last isomorphism is given by the map
		\begin{equation}
			\upsilon \mapsto \oplus_{p \in \Sigma_C} 1_p \cdot \big( \res_{x_p} (\upsilon) + \res_{y_p} (\upsilon) \big),
		\end{equation}
		where $x_p, y_p \in N$ are distinct points satisfying $\rho(x_p) = \rho(y_p) = p$.
		Then (\ref{eq_sh_ex_seq_wc}) and the resolution of the sheaf $\omega_{N}(D)$ by the sheaves of germs of holomorphic forms with values in $\omega_{N}(D)^n$ induce, for any $n \in \integ$, the natural isomorphisms
 \begin{equation}\label{isom_cohomology_interp}
 \begin{aligned}
 	 	& \rho^* : H^0 \big(C, \omega_{C}^n) 
 	 	\to 
 	 	\ker (\dbar|_{\ccal^{\infty}_{\res} (N, \omega_{N}(D)^n)}),
 	 	\\
 	 	& \rho^* : H^1 \big(C, \omega_{C}^n) 
 	 	\to 
 	 	\ccal^{\infty} (N, \overline{\omega}_{N} \otimes \omega_{N}(D)^n)
 		/ 
 		\Im(\dbar|_{\ccal^{\infty}_{\res} (N, \omega_{N}(D)^n)}).
 \end{aligned}
 \end{equation}	
 \par Serre duality (cf. \cite[p. 90 - 91]{Arb_v2}) is the canonical isomorphism
 \begin{equation}\label{isom_serre_dual}
 	H^1 (C, \omega_{C}^n) 
 	\to 
 	(H^0 (C, \omega_{C}^{1-n}))^*, 
 \end{equation}
 given by the following pairing: for $\upsilon \in H^1 \big(C, \omega_{C}^n)$ and $\alpha \in H^0 (C, \omega_{C}^{1-n})$, by (\ref{isom_cohomology_interp}), we define
 \begin{equation}\label{eq_serre_dual}
 	(\upsilon, \alpha) = \frac{- \imun}{2 \pi} \int_{N} \upsilon \wedge \alpha.
 \end{equation}
 The integration (\ref{eq_serre_dual}) is well-defined since only the poles of first order appear under the integral. By (\ref{eq_res_cond2}), Stokes and Residue theorems, (\ref{eq_serre_dual}) defines a pairing of $H^1 (C, \omega_{C}^n)$ with $H^0 (C, \omega_{C}^{1-n})$.
 \par Now, when $X$ is non-singular and $\omega_{X}$ is endowed with a Hermitian norm $\norm{\cdot}_{X}$, by (\ref{defn_L_2}), the left-hand side and the right-hand side of (\ref{isom_serre_dual}) are endowed with the induced $L^2$-norm $\norm{\cdot}_{L^2}$. 
 By using the description of Serre duality through the Hodge star operator (cf. \cite[p. 310]{DemCompl}), we observe that for any $\upsilon \in H^1 \big(C, \omega_{C}^n)$, we have
 \begin{equation}
 	\sup 
 	\Big\{ 
 		\big| \tinyint_{N} \upsilon \wedge \beta \big|^2  : \beta \in H^0 \big(C, \omega_{C}^{1-n}) \setminus \{0\}, \quad \norm{\beta}_{L^2} = 1
 	\Big\}
 	=
 	2 \pi \norm{\upsilon}_{L^2}^{2}.
 \end{equation}
 Thus, under the isomorphism (\ref{isom_serre_dual}), we have the \textit{isometry}
 \begin{equation}\label{eq_serre_isom}
 	(H^1(C, \omega_C^{n}), \norm{\cdot}_{L^2}) = (H^0(C, \omega_C^{1-n})^*, \norm{\cdot}_{L^2}^{-1}).
 \end{equation}
	\par 
	Now let's pass to the study of curves in \textit{families}.
	We fix a holomorphic, proper, surjective map $\pi: X \to S$ of complex manifolds, such that for every $t \in S$, the space $X_t := \pi^{-1}(t)$ is a curve (in the terminology of \cite{BisBost}, \cite{FinII2}, a f.s.o.).
\begin{prop}[{\cite[Proposition 3.1]{BisBost}}]\label{prop_coord}
	For every $x \in X$, there are local holomorphic coordinates $(z_0, \ldots, z_q)$ of $x \in X$ and $(w_1, \ldots, w_q)$ of $\pi(x) \in S$, such that $\pi$ is locally defined either by one of the following identities
	\begin{align}
		& w_i = z_i,  &&\text{for} \quad i=1, \ldots, q,  \label{eq_pr_nonsing}
		\\
		& w_1 = z_0 z_1; \quad w_i = z_i,  &&\text{for} \quad i=2, \ldots, q. \label{eq_pr_sing}
	\end{align}
\end{prop} 
\begin{cor}[{\cite[\S 3(a)]{BisBost}}]\label{cor_sigma}
	Let $\Sigma_{X/S} \subset X$ be the locus of double points of the fibers of $\pi$. Then:
	\\ \hspace*{0.5cm} 	a) $\Sigma_{X/S}$ is a submanifold of $X$ of codimension $2$;
	\\ \hspace*{0.5cm} 	b) the map $\pi|_{\Sigma_{X/S}} : \Sigma_{X/S} \to S$ is a closed immersion;
	\\	 \hspace*{0.5cm}	c) the map $\pi|_{X \setminus \Sigma_{X/S}} : X \setminus \Sigma_{X/S} \to S$ is a submersion.
	\\
	In particular, the direct image $\Delta = \pi_*(\Sigma_{X/S})$ is a divisor in $S$.
\end{cor}
\begin{notat}
	We use the notation $\Delta$, $\Sigma_{X/S}$ for the divisor and the submanifold from Corrolary \ref{cor_sigma}.
\end{notat}
\par Let's recall the construction of the \textit{relative canonical} line bundle $\omega_{X/S}$ of a f.s.o. $\pi : X \to S$.  Define the sheaf $\Omega_{X/S}$ by the exact sequence:
\begin{equation}\label{ex_seq_1}
\pi^* \Omega_S \to \Omega_X \to \Omega_{X/S} \to 0.
\end{equation}
By Corollary \ref{cor_sigma}, the exact sequence (\ref{ex_seq_1}) becomes exact to the left when restricted to $X \setminus \Sigma_{X/S}$:
\begin{equation}\label{ex_seq_2}
0 \to \pi^* \Omega_S|_{X \setminus \Sigma_{X/S}} \to \Omega_X|_{X \setminus \Sigma_{X/S}} \to \Omega_{X/S}|_{X \setminus \Sigma_{X/S}} \to 0.
\end{equation}
By taking determinants of (\ref{ex_seq_2}), we deduce the isomorphism
\begin{equation}
	 \Omega_{X/S}|_{X \setminus \Sigma_{X/S}} = (\omega_X \otimes \pi^* \omega_S^{-1})|_{X \setminus \Sigma_{X/S}}.
\end{equation}
We define
\begin{equation}\label{eq_rel_can_defn}
	\omega_{X/S} := \omega_X \otimes \pi^* \omega_S^{-1}.
\end{equation}
Then $\omega_{X/S}$ is the unique extension of the line bundle $\Omega_{X/S}|_{X \setminus \Sigma_{X/S}}$ over $X$. This line bundle is called the \textit{relative canonical} line bundle of $\pi : X \to S$.
\par Let $x \in \Sigma_{X/S}$. Take local coordinates $(z_0, \ldots , z_q)$ on an open neighborhood $V$ of $x \in X$ and local coordinates $(w_1, \ldots, w_q)$ of $\pi(x) \in S$, as in (\ref{eq_pr_sing}).
Then the manifold $\Sigma_{X/S} \cap V$ is given by
\begin{equation}
	\{ z_0 = 0 \text{ and }  z_1 = 0\}.
\end{equation}
Consider the sections $dz_0/z_0$ and $dz_1/z_1$ of $\Pi_X$, defined over the sets $\{ z_0 \neq 0 \}$ and $\{ z_1 \neq 0 \}$ respectively. The images of $dz_0/z_0$ and $-dz_1/z_1$ in $\omega_{X / S}$ coincide over $\{z_0z_1 \neq 0\}$, since
\begin{equation}\label{eq_dz_0_dz_1_ident}
	\frac{dz_0}{z_0} + \frac{dz_1}{z_1} = \pi^* \frac{dw_1}{w_1}.
\end{equation}
Thus, they define a nowhere vanishing section $\sigma$ of $\omega_{X / S}$ over $V \setminus \Sigma_{X/S}$. Since $\Sigma_{X/S}$ is of codimension $2$, the section $\sigma$ extends to a nowhere vanishing section over $V$ of the line bundle $\omega_{X/S}$.
\par Now, let $s_0 := \pi(x) \in \Delta$, $x \in \Sigma_{X/S}$, and let $\rho : Y_{s_0} \to X_{s_0}$ be the normalization of $X_{s_0}$ at $x$. Then by the discussion above, there is the canonical isomorphism 
\begin{equation}\label{eq_can_pull_back_norm}
	\rho^* \omega_{X/S} = \omega_{Y_{s_0}} \otimes \mathscr{O}_{Y_{s_0}}(\rho^{-1}(x)),
\end{equation}
 which induces the isomorphism (\ref{isom_pull_back_twisted}).
 The identity (\ref{eq_dz_0_dz_1_ident}) implies that for the natural inclusion $i_{s_0} : X_{s_0} \to X$, the pull-back $(i_{s_0})^* \omega_{X/S}$ is canonically isomorphic to $\omega_{X_{s_0}}$.
\par Now let's fix disjoint sections $\sigma_1, \ldots, \sigma_m : S \to X \setminus \Sigma_{X/S}$ and a Hermitian metric $\norm{\cdot}^{\omega}_{X/S}$ on $\omega_{X / S}$ over $\pi^{-1}(S \setminus |\Delta|) \setminus ( \cup_i \Im(\sigma_i) )$, such that for any $t \in S \setminus |\Delta|$, the restriction of $\, \norm{\cdot}^{\omega}_{X/S}$ over $\pi^{-1}(t) \setminus (\cup_i \sigma_i(t))$ induces the Kähler metric $g^{TX_t}$ over $X_t \setminus (\cup_i \sigma_i(t))$ such that the associated triple $(X_t, \{ \sigma_1(t), \ldots, \sigma_m(t) \}, g^{TX_t})$ becomes a surface with cusps. As a short-cut, we call $(\pi; \sigma_1, \ldots, \sigma_m; \norm{\cdot}^{\omega}_{X/S})$ a f.s.c. (family of surfaces with cusps).
\par Now, let $(\xi, h^{\xi})$ be a Hermitian vector bundle over $X$. For $t \in S$, we denote
	\begin{multline}\label{eq_det_r_star}
		\det (R^{\bullet} \pi_* (\xi  \otimes \omega_{X/S}(D)^n))_t := \det H^{0}(X_t, \xi \otimes \omega_{X/S}(D)^n) \\
		\otimes (\det H^{1}(X_t, \xi  \otimes \omega_{X/S}(D)^n))^{-1}.
	\end{multline}
	By Grothendick-Knudsen-Mumford \cite{Knudsen1976} (cf. \cite[Proposition 4.1]{BisBost}), the family of complex lines $(\det (R^{\bullet} \pi_* (\xi  \otimes \omega_{X/S}(D)^n))_t)_{t \in S}$ is endowed with a natural structure of holomorphic line bundle $\det (R^{\bullet} \pi_* (\xi  \otimes \omega_{X/S}(D)^n))$ over $S$. We denote 
	\begin{equation}
		\lambda(j^*(\xi  \otimes \omega_{X/S}(D)^n)) := \big( \det (R^{\bullet} \pi_* (\xi  \otimes \omega_{X/S}(D)^n)) \big)^{-1}.
	\end{equation}
	\par Following \cite{FinII2}, the pointwise Quillen norms induce the Quillen norm $\norm{\cdot}_{Q} \big( g^{TX_t}, h^{\xi} \otimes \, \norm{\cdot}_{X/S}^{2n} \big)$ on the line bundle $\lambda(j^*(\xi  \otimes \omega_{X/S}(D)^n))$. 
	Similarly, the pointwise Wolpert norms glue into the Wolpert norm $\norm{\cdot}^W_{X/S}$ on $\otimes_{i = 1}^{m} \sigma_i^{*} \omega_{X/S}$.
	\begin{sloppypar}
	\begin{defn}\label{defn_quil_norm}
		The Quillen norm on the line bundle $\lambda(j^*(\xi  \otimes \omega_{X/S}(D)^n))$, $n \leq 0$ over $S \setminus |\Delta|$ is defined for $t \in S \setminus |\Delta|$ by 
		\begin{equation}\label{eq_defn_quil_norm}
			\norm{\cdot}_{Q} \big( g^{TX_t}, h^{\xi} \otimes \, \norm{\cdot}_{X/S}^{2n} \big)
			:= 
			 T \big( g^{TX_t}, h^{\xi} \otimes \, \norm{\cdot}_{X/S}^{2n} \big)^{1/2} \cdot
			\norm{\cdot}_{L^2} \big( g^{TX_t}, h^{\xi} \otimes \norm{\cdot}_{X/S}^{2n} \big).
		\end{equation}
	\end{defn}
	\end{sloppypar}
	Of course, one has to motivate Definition \ref{defn_quil_norm}. One of the motivations might look as follows: one can easily see that the jumps in the cohomology of the fibers, which might occur when one moves over the base space, would cause singularities in the $L^2$ norm.
	However, a special case of continuity theorem says
	\begin{thm}[{\cite[Theorem C3]{FinII2}}]\label{thm_cont}
		Suppose that assumption (\ref{suppos_s3}) holds. Then the Hermitian norm (\ref{quil_wol_norm}) on the line bundle (\ref{det_wol_prod}) extends continuously over $S$.
	\end{thm}

\subsection{Singular Hermitian vector bundles}\label{sect_sing_metr}
	In this section we recall several notions of singularities for Hermitian vector bundles.
	\par
	We work with a complex manifold $Y$ of dimension $q+1$, a \textit{normal crossing divisor} $D_0 \subset Y$ and a submanifold $F \subset Y$.
	\begin{defn}\label{defn_adapt_chart}
		A triple $(U; z_0, \ldots, z_q; l)$ of an open set $U \subset Y$, coordinates $z_0, \ldots, z_q : U \to \comp$ and $l \in \nat$ is called an \textit{adapted chart} for $D_0$ (resp. $F$) at $x \in D_0$ (resp. $x \in F$) if $U = \{ (z_0, \ldots, z_q) \in \comp^{q+1} : |z_i| < 1, \text{ for all } i = 0, \ldots, q \}$ and $D_0 \cap U$ (resp. $F \cap U$) is defined by $\{z_0 \cdots z_l = 0\}$ (resp. $\{z_0 = 0, \ldots, z_l = 0\}$).
	\end{defn}
	\begin{notat}
		Let $(U; z_0, \ldots, z_q; l)$ be an \textit{adapted chart} for $D_0$. We denote
		\begin{equation}
			d \zeta_k = 
			\begin{cases}
				\hfill dz_k/(z_k \log |z_k|^2), & \text{if} \quad 0 \leq k \leq l, \\
				\hfill dz_k, & \text{if} \quad l+1 \leq k \leq q.
			\end{cases}
		\end{equation}
	\end{notat}
	\begin{defn}\label{defn_loglog_gr}
		\par a) \cite[Definition 2.1]{FreixTh} A function $f : Y \setminus F \to \comp$ has \textit{log-log growth on $Y$, with singularities along $F$} if for any $x \in Y$, for some adapted chart $(U; z_0, \ldots, z_q; l)$ of $F$ at $x$, and for some $C > 0$, $p \in \nat$, we have
		\begin{equation}
			\textstyle |f(z_0, \ldots, z_q)| \leq C \Big( \log \big| \log \big( \max_{k = 0}^{l} \{ |z_k| \} \big) \big| \Big)^p  + C.
		\end{equation}
		\par b) \cite[p. 240]{MumHirz} A differential form over $Y \setminus D_0$ has \textit{Poincaré growth on $Y$, with singularities along $D_0$}, if it can be expressed as a linear combination of monomials constructed using $d \zeta_k, \overline{d \zeta_k}$, $k = 0, \ldots, q$ with coefficients $f \in \ccal^{\infty}(Y \setminus D_0) \cap L^{\infty}(Y \setminus D_0)$.
		\par c) \cite[Definition 2.14]{FreixTh} A smooth function $f: Y \setminus D_0 \to \comp$ is \textit{P-singular, with singularities along $D_0$}, if $\partial f$, $\dbar f$, $\partial \dbar f$ have Poincaré growth on $Y$, with singularities along $D_0$. 
	\end{defn}
	\begin{defn}[{\cite[p. 242]{MumHirz}}]\label{defn_crit_preloglog}
		Let $L$ be a holomorphic line bundle over $Y$ and let $h^L$ be a smooth Hermitian metric on $L$ over $Y \setminus D_0$.
		Then $h^L$ is good with singularities along $D_0$ if for every local holomorphic frame $\upsilon$ of $L$ over $U \subset Y$, the function $\log h^L(\upsilon, \upsilon)$ is P-singular, with singularities along $D_0$.
	\end{defn}
	\begin{rem}
		The original definition of Mumford differs from the one presented here. Their equivalence is proved in {\cite[Proposition 3.2]{FreixTh}}.
	\end{rem}
	\begin{defn}\label{defn_logloggwth}
		Let $L$ be a holomorphic line bundle over $Y$ and let $h^L$ be a smooth Hermitian metric on $L$ over $Y \setminus F$.
		Then $h^L$ has log-log growth with singularities along $F$ if for every local holomorphic frame $\upsilon$ of $L$ over $U \subset Y$, the function $\log h^L(\upsilon, \upsilon)$ has log-log growth with singularities along $F$.
	\end{defn}
	We fix a holomorphic, proper, surjective map $\pi: X \to S$ of complex manifolds, such that for every $t \in S$, the space $X_t := \pi^{-1}(t)$ is a curve. Suppose that the divisor of singular curves $\Delta$ has normal crossings. Let $D$ be a divisor on $X$ such that $\pi|_{D} : D \to S$ is a local isomorphism.
	\begin{prop}\label{prop_int_loglog}
	Let $\alpha$ be a differential $(1,1)$-form over $X \setminus (\Sigma_{X/S} \cup |D|)$, such that it has Poincaré growth on $X \setminus |D| \cup |\pi^{-1}(\Delta)|$ with singularities along $D \cup \pi^{-1}(\Delta)$, and the coupling of $\alpha$ with continuous vertical vector fields over $X \setminus (\Sigma_{X/S} \cup |D|)$ is continuous.
	Let $f : X \setminus (\Sigma_{X/S} \cup |D|) \to \real$ be a continuous function, with log-log growth along $\Sigma_{X/S} \cup |D|$.
	\par Then for the normalization $\rho : Y_t \to X_t$ of $X_t$, $t \in |\Delta|$, the form $\rho^*(f \alpha)$ is integrable over $Y_{t}$.
	Moreover, the function $\pi_{*}[f \alpha]$ extends continuously over $S$, and the value of this extension is
	\begin{equation}\label{eq_prop_cont_ext}
		\pi_{*}[f \alpha](t) = \tinyint_{Y_{t}} \rho^*(f \alpha).
	\end{equation}
\end{prop}
\begin{proof}
	The first part of the statement was proved in {\cite[Proposition 3.1c)]{FinII2}}. The second statement follows directly from the proof of {\cite[Proposition 3.1c)]{FinII2}}.
\end{proof}

\subsection{Families of hyperbolic surfaces and Quillen metric}\label{sect_hyp_surf}
	In this section we recall the results of Freixas from \cite{FreixasARR} and \cite{FreixARRgen}, which describe how the Quillen metric, defined using the version of analytic torsion due to Takhtajan-Zograf (see (\ref{eqn_sel_norm})) on stable surfaces endowed with constant scalar curvature $-1$, behaves in a degenerating family.
	The family he considers is a special case of plumbing family construction, which is originally due to Wolpert \cite[p. 434]{Wolp90}.
	\par 
	Let's describe this construction.
	We fix a Riemann surface $\overline{M}$ with $m$ fixed points $D_M = \{P^M_{1}, \ldots, P^M_{m} \} \subset \overline{M}$ and a Riemann surface $\overline{T}$ of genre $1$ with one fixed point $D_T = \{ P^T \} \subset \overline{T}$.
	Take $m$ copies $(\overline{T}_i, P^T_{i})$ of $(\overline{T}, P^T)$.
	Let $g \in \nat$ be the genus of $\overline{M}$. 
	Clutching morphisms $\beta$ (see (\ref{eq_clutch_morph})), applied to the pairs of points $\{ P^M_{1}, P^T_{1} \}, \ldots, \{ P^M_{m}, P^T_{m} \}$, realizes the pointed surface $(M, m \cdot T) := (\overline{M}, D_M) \cup (\overline{T}_1, P^T_{1}) \cup \cdots \cup (\overline{T}_m, P^T_{m})$ as a point in a compactifying divisor $\partial \modul_{g+m, 0}$ of $\modulcomp_{g+m, 0}$. The \textit{plumbing family} associated with  $(M, m \cdot T)$ is a family of pointed curves representing a transversal direction to $\partial \modul_{g+m, 0}$ in $\modulcomp_{g+m, 0}$.
	More precisely, we consider
	\begin{enumerate}
		\item Neighborhood $U_i$ of $P^M_{i} \in \overline{M}$, $i = 1, \ldots, m$ biholomorphic to an open disc and a holomorphic coordinate mappings $F_i : U_i \to \comp$ with $F_i(P^M_{i}) = 0$;
		\item Similarly, a neighborhood $V$ of $P^T \in \overline{T}$, and a holomorphic coordinate mapping $G : V \to \comp$ satisfying $G(P^T) = 0$;
		\item A small complex parameter $t \in \comp$.
	\end{enumerate}
	We suppose that $U_i$ are pairwise disjoint. 
	Let $c > 0$ be such that $D(c) \subset \comp$ is contained in $\Im (F_i)$, $i = 1, \ldots, m$ and $\Im (G)$. 
	We take $m$ copies $G_1, \ldots, G_m$ of $G$, and regard them as functions acting on $\overline{T}_1, \ldots, \overline{T}_m$ respectively.
	Let $|t| < c^2$.
	For $d \in D(c)$, we note
	\begin{equation}\label{defn_rd}
		R^{d, *} = \big( \overline{M} \setminus (\cup_{i = 1}^{m} \{ |F_i| \leq |d| \}) \big) \cup \big( \overline{T}_1 \setminus \{ |G_1| \leq  |d| \} \big) \cup \cdots \cup \big( \overline{T}_m \setminus \{ |G_m| \leq  |d| \} \big).
	\end{equation}
	\begin{sloppypar}
		Consider the equivalence relation on points of $R^{t/c, *}$ generated by: 
		\begin{equation}
			\text{$p \sim q$ if $|t|/c \leq |F_i(p)| \leq c $, $|t|/c \leq |G_i(q)| \leq c$, $F_i(p)G_i(q) = t$.}
		\end{equation}
		 Form the identification space $X_t = R^{t/c, *} / \sim$. 
		 The curve $X_t$, $t \in D(c^2)$, is called the \textit{plumbing construction} for $(M, m \cdot T)$ associated with the \textit{plumbing data} $(\cup_i U_i, V, \cup_i F_i, G, t)$. Trivially, we see that a set $X := \cup_{t \in D(c^2)} X_t$ can be endowed with a structure of a complex manifold, for which $\pi : X \to S := D(c^2)$ is a proper holomorphic map of codimension 1.
		 This construction is also called the \textit{Bers trick} for $\overline{M}$ (cf. \cite{BersTrick}, \cite[Construction 4.3.2]{FreixTh}).
	\end{sloppypar}
			 \begin{figure}[h]
			\centering
			\includegraphics[width=0.8\textwidth]{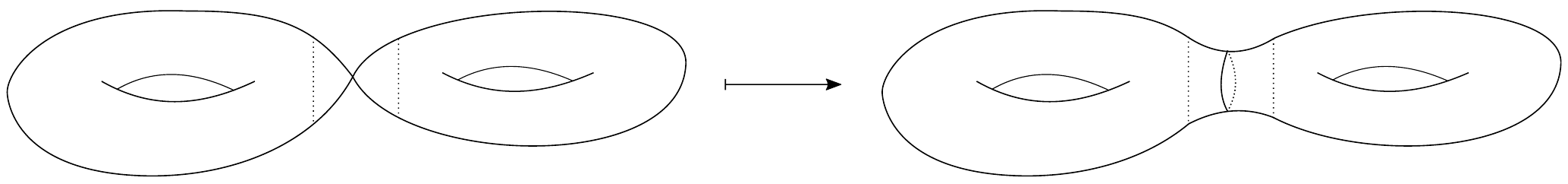}	
			\caption{Plumbing family. The regions away from the dashed lines are isomorphic.
			}
			\label{fig_graft}
		\end{figure}
	 Now, suppose that the pointed surface $(\overline{M}, D_M)$ is stable, i.e. it satisfies (\ref{cond_stable}), and let $g^{TM}_{\rm{hyp}}$ be the canonical complete hyperbolic metric of constant scalar curvature $-1$ on $\overline{M} \setminus D_{M}$ with cusps at $D_M$. Then one can take the functions $F_i$, $i = 1, \ldots, m$, from the plumbing construction to be Poincaré-compatible coordinates $z^M_{i}$ of $P^M_{i}$ (see (\ref{reqr_poincare})). 
	 Similarly, we make the choice for $G_i = z^T$. 
	 The plumbing family associated to this plumbing data is called the \textit{canonical plumbing family}. 
	 \par From now on, we fix the canonical plumbing family $\pi: X \to S := D(c^2)$. Then the divisor of singular curves is given by $\Delta = m \cdot \{0\}$.
	 We denote by 
	 \begin{equation}
	 	\rho : Y_0 := (\overline{M} \cup \overline{T}_1 \cup \cdots \cup \overline{T}_m) \to X_0
	 \end{equation}
	 the normalization of the singular fiber.
	 We denote by 
	 \begin{equation}
	 	\Sigma_{X/S} = \{Q_1, \ldots, Q_m \}, \quad Q_i = \rho(P_i^{M}), 
	 \end{equation}
	 the set of singular points in $X_0$.
	  Let $g^{TX_t}_{\rm{hyp}}$, $t \neq 0$ be the canonical hyperbolic metric of constant scalar curvature $-1$ on $X_t \setminus D_{X_t}$ with cusps at $D_{X_t}$. 
	 We denote by $Z_{X_t}(s)$ the Selberg zeta-function associated with $X_t$, given by the formula (\ref{defn_sel_zeta}). 
	 Let $\norm{\cdot}_{X/S}^{\rm{hyp}}$ be the Hermitian norm on $\omega_{X/S}(D)$ over $X \setminus (\pi^{-1}(|\Delta|) \cup |D_{X/S}|)$, induced from $g^{TX_t}_{\rm{hyp}}$ by Construction \ref{const_norm_div}.
	 \par 
	 We consider the determinant line bundle $\lambda(j^*(\omega_{X/S}(D)^n))$, $n \leq 0$, and we endow it over $S \setminus \Delta$ with the Takhtajan-Zograf version of Quillen norm (cf. \cite[\S 6]{FreixTh}), given by (compare with (\ref{eq_defn_quil_norm}))
	\begin{equation}\label{eq_tz_defn_quil_norm}
		\norm{\cdot}_{Q}^{TZ} \big( g^{TX_t}_{\rm{hyp}}, (\, \norm{\cdot}_{X/S}^{\rm{hyp}})^{2n} \big)
		:= 
		 T_{TZ} \big( g^{TX_t}_{\rm{hyp}}, (\, \norm{\cdot}_{X/S}^{\rm{hyp}})^{2n} \big)^{1/2} \cdot
		\norm{\cdot}_{L^2} \big( g^{TX_t}_{\rm{hyp}}, (\, \norm{\cdot}_{X/S}^{\rm{hyp}})^{2n} \big).
	\end{equation}
	We construct the norm (compare with (\ref{quil_wol_norm}))
	\begin{equation}\label{quil_tz_wol_norm}
		\norm{\cdot}_{\mathscr{L}_n}^{TZ} := \big( \norm{\cdot}_{Q}^{TZ} ( g^{TX_t}_{\rm{hyp}}, (\, \norm{\cdot}_{X/S}^{\rm{hyp}})^{2n} ) \big)^{12} 
		\otimes 
		\norm{\cdot}^{\rm{div}}_{\Delta}
	\end{equation}				
	on the line bundle (compare with (\ref{det_wol_prod}))
	\begin{equation}\label{det_tz_wol_prod}
		\mathscr{L}_n^{TZ} :=  
		\lambda \big( j^* (\omega_{X/S}^n) \big)^{12} 
		\otimes 
		\mathscr{O}_S(\Delta).
	\end{equation}
	\par We denote by $\norm{\cdot}_M^{\rm{hyp}}$, $\norm{\cdot}_T^{\rm{hyp}}$ the norms on $\omega_M(D)$, $\omega_T(D)$ induced by the canonical hyperbolic metrics $g^{TM}_{\rm{hyp}}$, $g^{TT}_{\rm{hyp}}$ of constant scalar curvature $-1$ on $\overline{M} \setminus D_{M}$, $\overline{T} \setminus D_{T}$ with cusps at $D_M$ and $D_T$ respectively.
	We denote by $\norm{\cdot}^{W, \rm{hyp}}_{M}$, $\norm{\cdot}^{W, \rm{hyp}}_{T}$ the associated Wolpert norms on the complex lines $\det (\omega_{\overline{M}}|_{D_M})$ and $\det (\omega_{\overline{T}}|_{D_T})$.
	Now, we define the norm (compare with (\ref{quil_wol_norm_rest}))
		\begin{multline}\label{quil_tz_wol_norm_rest}
			\norm{\cdot}_{\mathscr{L}'_n}^{TZ} := 
			\Big( 
			\norm{\cdot}_Q^{TZ} (g^{TM}_{\rm{hyp}}, (\, \norm{\cdot}_{M}^{\rm{hyp}})^{2n}) 
			\otimes 
			\big( \norm{\cdot}_Q^{TZ} (g^{TT}_{\rm{hyp}}, (\, \norm{\cdot}_{T}^{\rm{hyp}})^{2n}) \big)^{m} 
			\Big)^{12} 
			\\
			\otimes \big( \norm{\cdot}^{W, \rm{hyp}}_{M} \otimes (\,  \norm{\cdot}^{W, \rm{hyp}}_{T})^{m} \big)^{-1}
		\end{multline}				
		on the complex line (compare with (\ref{det_wol_prod_rest}))
		\begin{equation}\label{det_tz_wol_prod_rest}
			\mathscr{L}^{TZ}_n{}' :=  
			\Big( 
			\lambda \big(\omega_{M}(D)^n \big)
			\otimes 
			\lambda \big(\omega_{T}(D)^n \big)^{m} 
			\Big)^{12} 
			\otimes 
			\Big(
			\det (\omega_{\overline{M}}|_{D_M}) 
			\otimes 
			( \det (\omega_{\overline{T}}|_{D_T}))^{m} 
			\Big)^{-1}.
		\end{equation}
		Then the isomorphism (\ref{isom_main}) gives in our case the canonical isomorphism
		\begin{equation}\label{isom_tz_main}
			\mathscr{L}_n^{TZ}|_{\Delta} \to \mathscr{L}^{TZ}_n{}' 
			\otimes 
			\big(
			\otimes_{i = 1}^{m}			
				\mathscr{O}_{Q_{i}}
			\big)^{12}.
		\end{equation}
		We recall that $C_k$, $k \in \nat$ were defined in (\ref{defn_C_k}).
		The main theorem of this section is
		\begin{thm}[Freixas, {\cite[Corollary 5.8]{FreixasARR}, \cite[Theorem 5.3]{FreixasARR}}]\label{thm_tz_imm_thm}
			The norm $\norm{\cdot}_{\mathscr{L}_n}^{TZ}$ extends continuously over $S$, and, under the isomorphism (\ref{isom_tz_main}), the following identity holds
			\begin{equation}
					\norm{\cdot}_{\mathscr{L}_n}^{TZ}|_{\Delta} = \exp(m \cdot C_{-n}) \cdot \norm{\cdot}_{\mathscr{L}_n{}'}^{TZ}.
			\end{equation}
		\end{thm}
	 	\begin{rem}
	 		Theorem \ref{thm_tz_imm_thm} corresponds exactly to Theorem \ref{thm_imm_thm} for a special choice of a family of curves, a special choice of the metric and different choice of the definition of the analytic torsion.
	 	\end{rem}
	\subsection{Model grafting and pinching expansion}\label{sect_gen_gft}
	The goal of this section is to recall model grafting construction and the pinching expansion of the hyperbolic metric due to Wolpert \cite{Wolp90}. For simplicity, we state his results only for the plumbing family considered in Section \ref{sect_hyp_surf}.
	We conserve the notation from Section \ref{sect_hyp_surf}.
	\par To be compatible with further notation, we denote
	\begin{equation}
		z^i_{0} := z^M_{i}, \qquad z^i_{1} :=  z^T_{i}.
	\end{equation}
	By the definition of plumbing family, the coordinates $(z^i_{0}, z^i_{1})$ serve as local holomorphic charts in the neighborhood $Q_i \in X$.
	We denote
	\begin{equation}
		U(Q_i, \epsilon) = \Big\{ x \in X :  |z^i_{0}(x)| < \epsilon,   |z^i_{1}(x)| < \epsilon \Big\}.
	\end{equation}
	Again, by the definition of plumbing family, in $t$-coordinates on $S$, we have
	\begin{equation}\label{eq_proj_gen_gft}
		\pi(z^i_{0}, z^i_{1}) = z^i_{0} z^i_{1}.
	\end{equation}
	\par The canonical hyperbolic metric on $M$ (resp. $T$) with cusps at $D_M$ (resp. $D_T$) induces a metric $g^{TR^{\epsilon}}_{{\rm{hyp}}}$ on $R^{\epsilon}$ (see (\ref{defn_rd}) for the definition of $R^{\epsilon}$). 
	Let $\epsilon$ be so small, so that $g^{TR^{\epsilon}}_{{\rm{hyp}}}$ is induced by 
	\begin{equation}
	\begin{aligned}
		& \frac{\imun d z^M_{j} d \overline{z}^M_{j} }{| z^M_{j} \log |z^M_{j}| |^2},  
      			&& \text{ over } \big\{ |z^M_{j}| < 2 \epsilon \big\}, \\
      	& \frac{\imun d z^T_{j} d \overline{z}^T_{j} }{| z^T_{j} \log |z^T_{j}| |^2},  
      			&& \text{ over } \big\{ |z^T_{j}| < 2 \epsilon \big\}. \\
	\end{aligned}
	\end{equation}
	\par We choose $c = \epsilon^2$ in the plumbing construction from Section \ref{sect_hyp_surf}.
	 Now, since the manifold $X \setminus (\cup_{i = 1}^{k} U(Q_i, \epsilon))$ is naturally isomorphic to the product $R^{\epsilon} \times D(\epsilon^2)$, the metric $g^{TR^{\epsilon}}_{{\rm{hyp}}}$ induces the Kähler metric $g^{TX_t}$ on $X_t \setminus (\cup_{i = 1}^{k} U(Q_i, \epsilon))$.
	\par 
	The \textit{model grafted metric} is built from the metric $g^{TX_t}$ and the \textit{hyperbolic metric on a cylinder}, see (\ref{defn_g_cyl1}).
	It is meant to be a simplified model of degeneration of metric, appearing in the degeneration of hyperbolic surfaces.
	More precisely, let $\nu : X \to [0,1]$ be smooth function, satisfying
	\begin{equation}\label{eq_defn_nu}
		\nu(x) = 
		\begin{cases} 
      		\hfill 0, & \text{ for } x \in X \setminus (\cup_{i = 1}^{k} U(Q_i, 2\epsilon)), \\
      		\hfill 1,  & \text{ for } x \in \cup_{i = 1}^{k} U(Q_i, \epsilon). 
 		\end{cases}
	\end{equation}
	For $t \in D(\epsilon^2)$, we denote by $g^{\text{Cyl}}_{i, t}$ the metric over the set
	\begin{equation}
		\big\{(z^i_{0}, z^i_{1}) \in X_t : |t| / (2 \epsilon) < |z^i_{0}| < 2 \epsilon \big\},
	\end{equation}		
	induced by the Kähler form
	\begin{equation}\label{defn_g_cyl1}
		\bigg( 
			\frac{\pi}{|z^i_{0}| \log |t|}
			\bigg(
				\sin 
				\frac{\pi \log |z^i_{0}|}{\log |t|}
			\bigg)^{-1}		
		\bigg)^2
		\imun
		d z^i_{0} d \overline{z}^i_{0}.
	\end{equation}
	We remark that due to the fact that over $X_t$, we have $z^i_{0} z^i_{1} = t$, the expression (\ref{defn_g_cyl1}) is symmetric with respect to the change of variables $z^i_{0} \leftrightarrow z^i_{1}$.
	\begin{sloppypar}
		 Following Wolpert \cite{Wolp90}, we define the \textit{model grafted metric} $g^{TX_t}_{\text{gft}}$ as follows: over $X_t \setminus (\cup_{i = 1}^{m} U(Q_i, 2 \epsilon))$, $g^{TX_t}_{\text{gft}}$ coincides with $g^{TX_t}$, and over $U(Q_i, 2 \epsilon)$, it is given by
	\begin{equation}\label{defn_gen_gft}
			g^{TX_t}_{\text{gft}} = 
      		\big(
      			g^{\text{Cyl}}_{i,t}
      		 \big)^{\nu}
      		 (g^{TX_t})^{1-\nu}.
	\end{equation}
	This metric has the following nice properties
	\begin{prop}\label{prop_gft_good}
		The metric $\norm{\cdot}_{X/S}^{\text{gft}}$ induced by $g^{TX_t}_{\text{gft}}$ over $X \setminus \pi^{-1}(|\Delta|)$ extends continuously over  $X \setminus \Sigma_{X/S}$. Moreover, it is good in the sense of Mumford on $X \setminus \pi^{-1}(|\Delta|)$ with singularities along $\pi^{-1}(\Delta)$ and has log-log growth with singularities along $\Sigma_{X/S}$.
	\end{prop}
	\begin{proof}
		This is an easy verification, see for example Wolpert \cite[Lemma 1.5]{Wolp90}.
	\end{proof}
	\begin{figure}[h]
		\centering
		\includegraphics[width=0.7\textwidth]{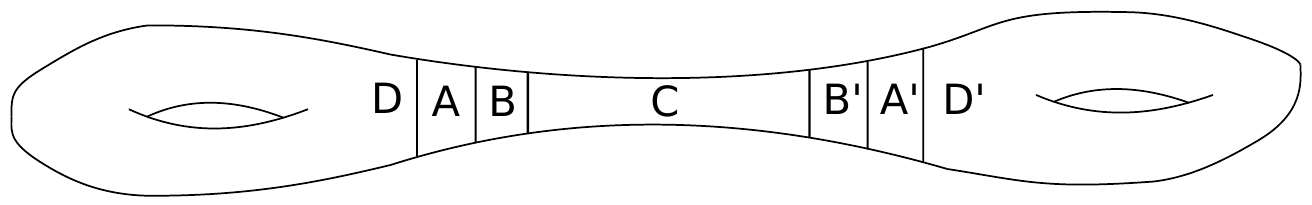}	
		\caption{The model grafting. Over the regions $D \cup A$, $D' \cup A'$, the metric $g^{TX_t}_{{\rm{gft}}}$ is isometric to $g^{TX_t}$.  
		Over the regions $A$, $A'$ it is Poincaré-compatible (see (\ref{reqr_poincare})) with coordinates $z^i_{0}$, $z^i_{1}$.
		Over the regions $B$, $B'$ the metric $g^{TX_t}_{{\rm{gft}}}$ is a geometric interpolation between $g^{TX_t}$ and (\ref{defn_g_cyl1}). 
		Over the region $C$, the metric $g^{TX_t}_{{\rm{gft}}}$ is given by (\ref{defn_g_cyl1}).
		}
		\label{fig_graft}
	\end{figure}
	\end{sloppypar}
	\par 
	Now let's cite a result about the comparison of the hyperbolic metric near degeneration and the grafted metric.
	Let's denote by $g^{TX_t}_{{\rm{hyp}}}$ the hyperbolic metric of constant scalar curvature $-1$ on $X_t$.
	As usual, we denote by $\norm{\cdot}^{\omega, {\rm{hyp}}}_{X/S}$ the induced Hermitian norm on $\omega_{X/S}$.
	The pinching expansion describes the behavior of $g^{TX_t}_{{\rm{hyp}}}$ near the singular fiber, it won't be used explicitly in this article, however, it is very important to understand the motivation behind the the metric $g^{TX_t}_{\varkappa}$ from Section \ref{sect_final_sect}.
	\begin{thm}[The pinching expansion, {Wolpert \cite[Expansion 4.2]{Wolp90}}]\label{thm_wolp_exp}
		For $t \in S \setminus |\Delta|$, we have
		\begin{equation}\label{compar_g_gft}
			 g^{TX_t}_{{\rm{hyp}}} = g^{TX_t}_{\rm{gft}} \Big(1 +  O \big( |\log |t||^{-2} \big) \Big),
		\end{equation}
		where the $O$-term is for the $\ccal^{k}$-norm over $X_t$ for any $k \in \nat$ with respect to $g^{TX_t}_{{\rm{hyp}}}$.
	\end{thm}
	\begin{rem}
		We note that in \cite[Theorem 4.3.1]{FreixTh}, Freixas proved Theorem \ref{thm_wolp_exp} for degenerating families of hyperbolic surfaces with cusps.
	\end{rem}

\section{The behavior of Quillen metric near singular fibers}
	The goal of this section is to prove Theorems \ref{thm_imm_thm}, \ref{thm_compat}, \ref{thm_imm_thm_modul}, \ref{thm_full_comp_pert}, which are the main statements of this article. 
	It is organized as follows. 
	In Section 3.1 we use Theorems \ref{thm_comp_appr}, \ref{thm_anomaly_cusp}, \ref{thm_cont}, \ref{thm_tz_imm_thm} to prove Theorems \ref{thm_imm_thm}, \ref{thm_compat}, \ref{thm_imm_thm_modul} modulo a certain universality statement.
	In Section 3.2 we slightly generalize the result of Bismut \cite[Theorem 0.3]{BisDegQuil} about the behavior of the Quillen norm in a smooth Kähler family of degenerating compact Riemann surfaces by dropping out the Kähler assumption. 
	Finally, in Section 3.3 by using this result, we prove the universality statement, which we used in Section 3.1.
	From the proof of this statement we will also deduce Theorem \ref{thm_full_comp_pert}.

		\subsection{Quillen metric on the singular locus, proof of Theorems \ref{thm_imm_thm}, \ref{thm_compat}, \ref{thm_imm_thm_modul}}\label{sect_main_pfs}
		In this section we will prove Theorems \ref{thm_imm_thm}, \ref{thm_compat}, \ref{thm_imm_thm_modul} modulo a certain universality statement, which will be established in Section \ref{sect_final_sect}. To state it, we conserve the notation from the statement of Theorem \ref{thm_imm_thm}.
		\begin{thm}\label{thm_imm_thm_upto_const}
			For any $n \leq 0$, there exists a universal constant $A_{-n} \in \real$ such that Theorem \ref{thm_imm_thm} holds for any family of curves $\pi : X \to S$ without hyperbolic cusps (i.e. $m = 0$) with the constant $A_{-n}$ in place of $C_{-n}$.
		\end{thm}		
		Now let's see how Theorem \ref{thm_imm_thm_upto_const} can be used to prove Theorems \ref{thm_imm_thm}, \ref{thm_compat}.
		\begin{proof}[Proof of Theorems \ref{thm_imm_thm}, \ref{thm_compat}]
			The proof consists of 3 steps, and we start with a short résumé of them.
			In \textit{Steps 1, 2}, we prove that in fact it is enough to establish Theorem \ref{thm_imm_thm} for $m = 0$.
			More precisely, in \textit{Step 1} we see that by Theorem \ref{thm_anomaly_cusp}, we can trivialize the Poincaré-compatible coordinates associated to $g^{TX_t}$ (thus, the associated Wolpert norm). 
			And in \textit{Step 2}, by Theorem \ref{thm_comp_appr}, we will see that one can delete the cusps from the metric obtained in \textit{Step 1}, so that the statements of Theorem \ref{thm_imm_thm} for the two metrics are equivalent.
			\par 
			In the first two steps the reduction is done by modifying norms $\norm{\cdot}_{X/S}^{\omega}$, $\norm{\cdot}_{X/S}$ only in the neighborhood of $|D_{X/S}|$ and our technique is the same as in the proofs of \cite[Theorems C, D]{FinII2}.
			\par Finally, in \textit{Step 3}, by applying Theorem \ref{thm_tz_imm_thm} two times and using the fact that for the 3-punctured hyperbolic sphere, endowed with constant scalar curvature $-1$ complete metric, our analytic torsion coincides with the version of the analytic torsion (\ref{eqn_sel_norm}) of Takhtajan-Zograf (see \cite[Remark 2.18a)]{FinII1}), we will see that the constant $A_{-n}$ from Theorem \ref{thm_imm_thm_upto_const} actually coincides with $C_{-n}$ from Theorem \ref{thm_imm_thm}. 
			This with the first two steps finishes the proof of Theorem \ref{thm_imm_thm}.
			\par We also note that by Theorem \ref{thm_anomaly_cusp}, we may suppose that near the singular locus, the Hermitian vector bundle $(\xi, h^{\xi})$ is trivial over a small neighborhood of $|D_{X/S}| \cup \Sigma_{X/S}$.
			\par Finally, we note that by Theorem \ref{thm_cont}, it is enough to prove Theorem \ref{thm_imm_thm} only in the case $\dim S = 1$, $S = D(1)$, $|\Delta| = \{0\}$. From now on we make those assumptions, and we conserve the relevant notations from Section \ref{sect_bismut_proof}.
			\par \textbf{Step 1.} 
			Let $V_{i, c}$, $i=1, \ldots, m$, $c > 0$ (resp. $U$) be a neighborhood of $\sigma_i(t_0)$ (resp. $t_0$) such that for some local coordinates $(z_0, \ldots, z_q)$ of $\sigma_i(t_0)$ and $(w_1, \ldots, w_q)$ of $t_0 \in S$, satisfying (\ref{eq_pr_nonsing}), we have $V_{i, c} = \{ x \in \pi^{-1}(U) : |z_0(x)| < c \}$ and $\{z_0(x) = 0\} = \{ \sigma_i(t): t \in U \}$.
			For simplicity, we note $V_i := V_{i, 1}$.
			Let 
			$\nu_0 : \real_+ \to [0,1]$ be a smooth function satisfying
			\begin{equation}\label{defn_nu}
				\nu_0(u) = 
				\begin{cases}
					\hfill 0, & \text{if} \quad u < 1/2, \\
					\hfill 1, & \text{if} \quad u > 3/4. 
				\end{cases}
			\end{equation}
			\par 
			We denote by $\norm{\cdot}^{\omega, 0}_{X/S}$ the norm on $\omega_{X/S}$ over $X \setminus (\pi^{-1}(|\Delta|) \cup |D_{X/S}|)$ such that $\norm{\cdot}^{\omega, 0}_{X/S}$ coincides with $\norm{\cdot}^{\omega}_{X/S}$ away from $\cup_{i = 1}^{m} V_i$, and over $(\cup_{i = 1}^{m} V_i) \setminus ( \pi^{-1}(|\Delta|) \cup |D_{X/S}| )$, we have
			\begin{equation}\label{defn_om_0}
				\norm{dz_0}^{\omega, 0}_{X/S} = \big| z_0 \log |z_0| \big|^{1 - \nu_0(|z_0|)} \cdot \big( \norm{dz_0}^{\omega}_{X/S} \big)^{\nu_0(|z_0|)}.
			\end{equation}
			Let $\norm{\cdot}^{0}_{X/S}$ be the norm on $\omega_{X/S}(D)$ induced from $\norm{\cdot}_{X/S}^{\omega, 0}$ as in Construction \ref{const_norm_div}, and let $g^{TX_t}_{0}$, $t \in S$ be the induced Kähler metric $X_t \setminus D_{X/S}$ with cusps at $D_{X/S} \cap X_t$. 
			We denote by $\norm{\cdot}^{W, 0, i}_{X/S}$, $i = 1, \ldots, m$, the Wolpert norms (see Definition \ref{defn_wolpert_norm}) on $ \sigma_i^{*} \omega_{X/S}$ induced by $g^{TX_t}_{0}$, and by $\norm{\cdot}^{W, 0}_{X/S}$ the induced Wolpert norm on $\otimes_{i = 1}^{m} \sigma_i^{*} \omega_{X/S}$.
			Then by Construction \ref{const_norm_div} and (\ref{defn_om_0}), we see that if $\norm{\cdot}_{X/S}$ satisfies assumptions (\ref{suppos_s3}) and (\ref{suppos_restr}), then $\norm{\cdot}_{X/S}^{0}$ satisfies assumptions (\ref{suppos_s3}) and (\ref{suppos_restr}) as well.
			In fact, this property along with the fact that $\norm{\cdot}^{\omega, 0}_{X/S}$ doesn't vary in the horizontal direction around the cusps are the only facts we need from the construction (\ref{defn_om_0}).
			\par We denote by $g^{TY_0}_0$ the Kähler metric on $Y_0 \setminus D_{Y_0}$, constructed from $\norm{\cdot}_{Y_0}^{0} := \rho^* ( \, \norm{\cdot}^{0}_{X/S})$ as in Construction \ref{const_norm_div}. We denote by $\norm{\cdot}^{W, 0}_{Y_0}$ the Wolpert norm on $\otimes_{i = 1}^{m + 2k} (\sigma_i')^{*} \omega_{Y_0}$ induced by $g^{TY_0}_{0}$.
			\par As we supposed that $(\xi, h^{\xi})$ is trivial in a neighborhood of $|D_{X/S}|$, and the metrics $\norm{\cdot}_{X/S}$, $\norm{\cdot}_{X/S}^{0}$ differ only in the neighborhood of $|D_{X/S}|$, by Theorem \ref{thm_anomaly_cusp}, applied pointwise for the line bundle $\lambda(j^*(\xi \otimes \omega_{X/S}(D)^n))^{12} \otimes (\otimes_{i = 1}^{m} \sigma_i^{*} \omega_{X/S})^{- \rk{\xi}}$, for any $t \in S \setminus |\Delta|$, we see that we have
			\begin{equation}\label{eq_thmc_aux_anomal1}
			\begin{aligned}
				&\frac{1}{6} \log \Big(  
					\norm{\cdot}_{Q}  \big(g^{TX_t}_{0},  h^{\xi} \otimes (\, \norm{\cdot}_{X/S}^{0})^{2n}\big)^{12} \otimes 
					\big( \norm{\cdot}_{X/S}^{W, 0} \big)^{-\rk{\xi}}
				\Big)
				\\
				& \qquad \qquad \qquad \qquad  \qquad
				- 
				\frac{1}{6} \log \Big(  
					\norm{\cdot}_{Q}  \big(g^{TX_t},  h^{\xi} \otimes (\, \norm{\cdot}_{X/S})^{2n}\big)^{12} \otimes 
					\big( \norm{\cdot}_{X/S}^{W} \big)^{-\rk{\xi}}
				\Big)
						\\
						&  = 
				 			\rk{\xi} \cdot \int_{X_t}
				 			\Big( 
			 					\widetilde{\td} \big(\omega_{X/S}(D)^{-1}, \, \norm{\cdot}^{-2}_{X/S}, (\, \norm{\cdot}^{0}_{X/S})^{-2} \big)  \ch \big(\omega_{X/S}(D)^n, \norm{\cdot}_{X/S}^{2n} \big) \\
								&  \phantom{= \int_{M} 
				 			\Big[ } 
				 			\qquad \qquad
				 			+			 		 					
					 			\td \big(\omega_{X/S}(D)^{-1}, (\, \norm{\cdot}^{0}_{X/S})^{-2} \big) \widetilde{\ch} \big(\omega_{X/S}(D)^n, \norm{\cdot}_{X/S}^{2n}, (\, \norm{\cdot}_{X/S}^{0})^{2n} \big) 									 				
					 		\Big).
				\end{aligned}
			\end{equation}
			We note that the conformal factor corresponding to the change of the metric from $\norm{\cdot}^{\omega}_{X/S}$ to $\norm{\cdot}^{\omega, 0}_{X/S}$ is non-trivial in the neighborhood of the cusp. Thus, we use Theorem \ref{thm_anomaly_cusp} with the conformal factor which doesn't have compact support in the punctured surface.
			\par 
			By applying Theorem \ref{thm_anomaly_cusp}, where instead of flattening, we choose a “partial flattening" applied for the $m$ points, coming from $\sigma_1', \ldots, \sigma_m'$, we get
			\begin{equation}\label{eq_thmc_aux_anomal1y0}
			\begin{aligned}
				&\frac{1}{6} \log \Big(  
					\norm{\cdot}_{Q}  \big(g^{TY_0}_{0},  \rho^* (h^{\xi}) \otimes (\, \norm{\cdot}_{Y_0}^{0})^{2n}\big)^{12} \otimes 
					\big( \norm{\cdot}_{Y_0}^{W, 0} \big)^{-\rk{\xi}}
				\Big)
				\\
				& \qquad \qquad \qquad \qquad  \qquad
				- 
				\frac{1}{6} \log \Big(  
					\norm{\cdot}_{Q}  \big(g^{TY_0}, \rho^* (h^{\xi}) \otimes (\, \norm{\cdot}_{Y_0})^{2n}\big)^{12} \otimes 
					\big( \norm{\cdot}_{Y_0}^{W} \big)^{-\rk{\xi}}
				\Big)
						\\
						&  = 
				 			\rk{\xi} \cdot \int_{Y_0}
				 			\Big( 
			 					\widetilde{\td} \big(\omega_{Y_0}(D)^{-1}, \, \norm{\cdot}^{-2}_{Y_0}, (\, \norm{\cdot}^{0}_{Y_0})^{-2} \big) \ch \big(\omega_{Y_0}(D)^n, \norm{\cdot}_{Y_0}^{2n} \big) \\
								&  \phantom{= \int_{M} 
				 			\Big[ } 
				 			\qquad \qquad \qquad \quad +			 		 					
					 			\td \big(\omega_{Y_0}(D)^{-1}, (\, \norm{\cdot}^{0}_{Y_0})^{-2} \big) \widetilde{\ch} \big(\omega_{Y_0}(D)^n, \norm{\cdot}_{Y_0}^{2n}, (\, \norm{\cdot}_{Y_0}^{0})^{2n} \big) 									 				
					 		\Big).
				\end{aligned}
			\end{equation}
			By Proposition \ref{prop_int_loglog}, we see that the right-hand-side of (\ref{eq_thmc_aux_anomal1}) extends continuously over $S$,
			moreover, as $t \to 0$, by (\ref{eq_prop_cont_ext}), the right-hand side of (\ref{eq_thmc_aux_anomal1}) tends to the right-hand side of (\ref{eq_thmc_aux_anomal1y0}). Thus, we see that it is enough to prove Theorem \ref{thm_imm_thm} for the metrics $\norm{\cdot}_{X/S}^{0}$, $\norm{\cdot}_{X/S}^{\omega, 0}$, $\norm{\cdot}_{X/S}^{W, 0}$ instead of $\norm{\cdot}_{X/S}$, $\norm{\cdot}_{X/S}^{\omega}$, $\norm{\cdot}_{X/S}^{W}$. We also note, that by (\ref{defn_om_0}), for $i = 1,\ldots, m$, the following identity holds
			\begin{equation}\label{eq_rel_wolp_normstep2}
				\big \| dz_0|_{\sigma_i(t)} \big \|^{W, 0, i}_{X / S} =  \big \| dz_0|_{\sigma_i'(0)} \big \|^{W, 0, i}_{Y / S'} = 1.
			\end{equation}
			\par 
			\textbf{Step 2.}
			We denote by $V_i' = V_{i, 1/2} \subset V_i$, $i = 1, \ldots, m$. Let $\norm{\cdot}^{\omega, {\rm{cmp}}}_{X/S}$ be the Hermitian norm on $\omega_{X/S}$ over $X \setminus \pi^{-1}(|\Delta|)$ such that $\norm{\cdot}^{\omega, {\rm{cmp}}}_{X/S}$ coincides with $\norm{\cdot}^{\omega, 0}_{X/S}$ away from $\cup_{i = 1}^{m} V_i'$, and for $\nu_0 : \real \to [0,1]$ from (\ref{defn_nu}), it is given over $V_i'$ by
			\begin{equation}\label{defn_om_cmp}
				\norm{dz_0}^{\omega, {\rm{cmp}}}_{X/S} = |z_0 \log |z_0||^{\nu_0(2|z_0|)}.
			\end{equation}
			\par
			We denote by $g^{TX_t}_{{\rm{cmp}}}$ the induced Kähler metric on  $X_t$.
			By (\ref{defn_om_cmp}), we see that if $\norm{\cdot}_{X/S}^{\omega, 0}$ satisfies assumptions (\ref{suppos_s3}) and (\ref{suppos_restr}), then $\norm{\cdot}_{X/S}^{\omega, \rm{cmp}}$ satisfies assumptions (\ref{suppos_s3}) and (\ref{suppos_restr}) as well, but for $D_{X/S} = \emptyset$, i.e. without the cusps.
			In fact, this property along with the fact that $\norm{\cdot}_{X/S}^{\omega, {\rm{cmp}}}$ doesn't vary in the horizontal direction around the cusps are the only facts we need from the construction (\ref{defn_om_cmp}).
			\par 
			We denote by $g^{TY_0}_{{\rm{cmp}}}$ the Kähler metric over $Y_0 \setminus \rho^{-1}(\Sigma_{X/S})$ induced from $\norm{\cdot}_{X/S}^{\omega, \rm{cmp}}$ as in Section 1 for $D_{X/S} = \emptyset$.
			We denote by $\norm{\cdot}^{{\rm{cmp}}}_{X/S}$ the norm on $\omega_{X/S}(D)$ over $X \setminus \pi^{-1}(|\Delta|)$, such that $\norm{\cdot}^{{\rm{cmp}}}_{X/S}$ coincides with $\norm{\cdot}^{0}_{X/S}$ away from $\cup_{i = 1}^{m} V_i'$, and over $V_i'$ we have
			\begin{equation}\label{defn_om_cmp2}
				\big\lVert dz_0 \otimes s_{D_{X/S}}/z_0 \big\rVert^{{\rm{cmp}}}_{X/S} = |\log |z_0||^{\nu_0(2|z_0|)}.
			\end{equation}
			We denote by $\norm{\cdot}^{{\rm{cmp}}}_{Y_0} := \rho^{*} (\, \norm{\cdot}^{{\rm{cmp}}}_{X/S})$ the induced Hermitian norm on $\omega_{Y_0}(D)$ over $Y_0 \setminus \rho^{-1}(\Sigma_{X/S})$. 
			Now, since in $g^{TX_t}_{0}$, the Poincaré-compatible coordinates of the cusps are trivialized, by Theorem \ref{thm_comp_appr}, we see that for $t \in S \setminus |\Delta|$, the following holds
			\begin{multline}\label{eq_thmc_aux_anomal22}
			2 \log \Big(  
				\norm{\cdot}_{Q} \big(g^{TX_t}_{\rm{cmp}},  h^{\xi} \otimes  (\, \norm{\cdot}_{X/S}^{{\rm{cmp}}})^{2n}\big) 
					\big/				 
					 \norm{\cdot}_{Q} \big(g^{TX_t}_{0}, h^{\xi} \otimes  (\, \norm{\cdot}_{X/S}^{0})^{2n}\big) 
					 \Big) 
			\\
			= 2 \log \Big(  
			\norm{\cdot}_{Q} \big(g^{TY_0}_{\rm{cmp}},  \rho^* (h^{\xi}) \otimes  (\, \norm{\cdot}_{Y_0}^{{\rm{cmp}}})^{2n}\big) 
				\big/				 
				 \norm{\cdot}_{Q} \big(g^{TY_0}_{0}, \rho^*  (h^{\xi}) \otimes (\, \norm{\cdot}_{Y_0}^{0})^{2n}\big) 
				 \Big),
			\end{multline}
			where we didn't mention the last term of (\ref{eqn_of_quil_norms_no_flat}) since $(\xi, h^{\xi})$ is trivial in the neighborhood of $|D_{X/S}|$, and the norms $\norm{\cdot}_{X/S}^{{\rm{cmp}}}$, $\norm{\cdot}_{X/S}^{0}$ differ only in the neighborhood of $|D_{X/S}|$.
			We denote by $\norm{\cdot}^{{\rm{cmp}}}_{D_{X/S}}$ the norm on $\mathscr{O}_X(D_{X/S})$, given by $\norm{\cdot}^{{\rm{cmp}}}_{X/S} / \norm{\cdot}^{\omega, {\rm{cmp}}}_{X/S}$. Then the norm $\norm{\cdot}^{{\rm{cmp}}}_{D_{X/S}}$ is trivial away from $\cup_{i = 1}^{m} V_i$ and smooth over $X$.
			\par By (\ref{isom_rest_w_div}), (\ref{eq_rel_wolp_normstep2}) and (\ref{eq_thmc_aux_anomal22}), we see that it is enough to prove Theorem \ref{thm_imm_thm} for the Hermitian vector bundles $(\xi \otimes \mathscr{O}_X(D_{X/S})^{n}, h^{\xi} \otimes (\, \norm{\cdot}^{{\rm{cmp}}}_{D_{X/S}})^{2n})$, $(\omega_{X/S}, \norm{\cdot}_{X/S}^{\omega, \rm{cmp}})$ and $D_{X/S} = \emptyset$, instead of $(\xi, h^{\xi})$, $(\omega_{X/S}, \norm{\cdot}_{X/S}^{\omega, 0})$ and $D_{X/S}$, given by (\ref{defn_dxs}). 
			But as the Hermitian norm $\norm{\cdot}^{{\rm{cmp}}}_{D_{X/S}}$ is smooth over $X$, such a statement would follow from Theorem \ref{thm_imm_thm} for $m = 0$. Thus, we conclude that to prove Theorem \ref{thm_imm_thm} in its full generality, it is enough to prove it only for $m = 0$. From now on, we suppose $m = 0$.
			\par 	\textbf{Step 3.}
			The goal of this step is to show that the constant $A_{-n}$ from Theorem \ref{thm_imm_thm_upto_const} actually coincides with the constant $C_{-n}$ from Theorem \ref{thm_imm_thm}. This would finish the proof of Theorem \ref{thm_imm_thm} by Steps 1,2. During the proof we will also establish Theorem \ref{thm_compat}.
			\par We consider a stable pointed Riemann surface $(\overline{M}, D_M)$ and the associated canonical plumbing family $\pi : X \to S$ with the canonical hyperbolic norm $\norm{\cdot}_{X/S}^{\rm{hyp}}$ on $\omega_{X/S}$ from Section \ref{sect_hyp_surf}.
			\par Then, in the notations of Section \ref{sect_hyp_surf}, by a theorem of Phong-d'Hooker (cf. Remark \ref{rem_thm_compat}), the following identity of norms over $S \setminus |\Delta|$ holds
			\begin{equation}\label{eq_pdh_restatement}
				\norm{\cdot}_Q (g^{TX_t}_{{\rm{hyp}}}, (\,\norm{\cdot}_{X/S}^{{\rm{hyp}}})^{2n}) 
				= 
				\norm{\cdot}_Q^{TZ} (g^{TX_t}_{{\rm{hyp}}}, (\,\norm{\cdot}_{X/S}^{{\rm{hyp}}})^{2n}).
			\end{equation}
			\par We apply this construction for $(\overline{M}, D_M) := (\overline{T}, D_T)$, where $(\overline{T}, D_T)$ is a 1-pointed torus, considered in Section \ref{sect_hyp_surf}. Then by Theorems \ref{thm_tz_imm_thm}, \ref{thm_imm_thm_upto_const} and (\ref{eq_pdh_restatement}), we get
			\begin{equation}\label{eq_torus_eq_mu_tz}
				\exp(A_{-n} /  2) \cdot \norm{\cdot}_Q (g^{TT}_{{\rm{hyp}}}, (\,\norm{\cdot}_{T}^{{\rm{hyp}}})^{2n}) 
				= 
				\exp(C_{-n} /  2) \cdot \norm{\cdot}_Q^{TZ} (g^{TT}_{{\rm{hyp}}}, (\,\norm{\cdot}_{T}^{{\rm{hyp}}})^{2n}).
			\end{equation} 
			By applying (\ref{eq_pdh_restatement}) again, but now for any $(\overline{M}, D_{\overline{M}})$, by Theorems \ref{thm_tz_imm_thm}, \ref{thm_imm_thm_upto_const} and (\ref{eq_torus_eq_mu_tz}), we see that for any $(\overline{M}, D_M)$, $m := \# D_M$, we have
			\begin{equation}\label{eq_final_pf_imm1}
				\exp(m \cdot A_{-n} /  2) \cdot \norm{\cdot}_Q (g^{TM}_{{\rm{hyp}}}, (\,\norm{\cdot}_{M}^{{\rm{hyp}}})^{2n}) 
				= 
				\exp(m \cdot C_{-n} /  2) \cdot  \norm{\cdot}_Q^{TZ} (g^{TM}_{{\rm{hyp}}}, (\,\norm{\cdot}_{M}^{{\rm{hyp}}})^{2n}).
			\end{equation} 
			However, by \cite[Remark 2.18a)]{FinII1}, our definition of the analytic torsion coincides with the definition of Takhtajan-Zograf for the 3-punctured hyperbolic sphere $P := \mathbb{P} \setminus \{0, 1, \infty \}$, i.e.
			\begin{equation}\label{eq_final_pf_imm2}
				\norm{\cdot}_Q (g^{TP}_{{\rm{hyp}}}, (\,\norm{\cdot}_{P}^{{\rm{hyp}}})^{2n}) 
				= 
				\norm{\cdot}_Q^{TZ} (g^{TP}_{{\rm{hyp}}}, (\,\norm{\cdot}_{P}^{{\rm{hyp}}})^{2n}).
			\end{equation} 
			By combining (\ref{eq_final_pf_imm1}) and (\ref{eq_final_pf_imm2}), we get $A_{-n} = C_{-n}$, which finishes the proof of Theorem \ref{thm_imm_thm} for $m = 0$. By this and Steps 1,2, we see that Theorem \ref{thm_imm_thm} holds for any $m \in \nat$. 
			From this, by (\ref{defn_quil}), (\ref{quil_tz_wol_norm}) and (\ref{eq_final_pf_imm1}), we deduce Theorem \ref{thm_compat}.
		\end{proof}
		\begin{proof}[Proof of Theorem \ref{thm_imm_thm_modul}.]
			By \cite[Proposition 5.6]{FinII2}, the norm $\norm{\cdot}_{g, m}^{{\rm{hyp}}}$ satisfies assumptions (\ref{suppos_s3}) and (\ref{suppos_restr}).
			Thus, Theorem \ref{thm_imm_thm_modul} is a direct consequence of Theorem \ref{thm_imm_thm} and \cite[Proposition 5.6]{FinII2}.
			The fact that the underlying spaces are orbifolds doesn't pose any problem, as our methods are local, and thus, can be applied on an orbifold chart.
		\end{proof}

	\subsection{Family of Riemann surfaces with smooth metric and Quillen metric}\label{sect_bismut_proof}
		In this section we describe a generalization of the result of Bismut \cite[Theorem 0.3]{BisDegQuil} for non-necessarily Kähler manifolds. This theorem describes the behavior of the Quillen norm in a family of degenerating Riemann surfaces endowed with \textit{compact} Riemann  metric, which comes from the metric on the total space of the family. It will be used in Section 4.3, but it is also of some independent interest.
		\par Let's fix a holomorphic, proper, surjective map $\pi: X \to S$ of complex manifolds, such that for every $t \in S$, the space $X_t := \pi^{-1}(t)$ is a curve (see Section \ref{sect_recall_relantors}). 
		Let $(\xi, h^{\xi})$ be a Hermitian vector bundle over $X$. 
		Let $g^{TX}$ be a Riemannian metric over $X$, which is compatible with the complex structure of $X$. 
		By $h^{TX}$ we note the Hermitian metric on $T^{(1, 0)}X$ induced by $g^{TX}$ by the natural identification $TX \ni Y \mapsto \frac{1}{2} (Y - JY) \in T^{(1,0)}X$, where $J$ is the complex structure of $X$.
		We denote by $g^{TX_t}$ the restriction of the metric $g^{TX}$ on $X_t$, $t \in S \setminus |\Delta|$. 
		Since $g^{TX}$ is compatible with the complex structure, the metric $g^{TX_t}$ is Kähler on $X_t$. 
		We denote by $\norm{\cdot}_Q(g^{TX_t}, h^{\xi})$ the Quillen norm on the line bundle $\lambda(j^* \xi)$  over $S \setminus |\Delta|$ (see (\ref{defn_quil})).
		\par For simplicity, assume that $\dim S = 1$, $S = D(1)$ and $|\Delta| = \{ 0 \}$.  
		We write $\Sigma_{X/S} = \{Q_1, \ldots, Q_k\}$. 
		Let $\rho : Y_0 \to X_0$ be the normalization of $X_0$. 
		We denote 
		\begin{equation}\label{eq_defn_pi}
			\rho^{-1}(\Sigma_{X/S}) = \{P_1, \ldots,   P_{2k} \},
		\end{equation}
		where $P_i$ are enumerated in such a way that $\rho(P_{2j- 1}) = \rho(P_{2j}) = Q_j$ for $j = 1, \ldots, k$.
		We denote by $g^{TY_0} := \rho^*(g^{TX})$ the induced Riemannian metric on $Y_0$ and by $\norm{\cdot}_{Y_0}^{\omega}$ the induced Hermitian norm on $\omega_{Y_0}$. 
		Since $g^{TX}$ is compatible with the complex structure, we see that $g^{TY_0}$ is Kähler on $Y_0$. 
		We denote by $\norm{\cdot}_Q(g^{TY_0}, \rho^* (h^{\xi}))$ the induced Quillen norm on the complex line $\lambda(\rho^* \xi)$.
		\par 
		Let $\norm{\cdot}_{\Sigma_{X/S}/X}^{i}$ be the Hermitian norm induced by the natural isomorphism (\ref{isom_pt_rho}) on the complex lines $\omega_{Y_0}|_{P_{2i - 1}} \otimes \omega_{Y_0}|_{P_{2i}}$, $i = 1, \ldots, k$.
		More explicitly, let local holomorphic coordinates $z^{i}_{0}, z^{i}_{1}$ around $Q_i \in X$ and $t$ around $0 \in S$ be as in (\ref{eq_proj_gen_gft}).
		We denote 
		\begin{equation}\label{eq_defn_ai_bi_ci}
			a_i = h^{TX} \Big( \frac{\partial}{\partial z^{i}_{0}}, \frac{\partial}{\partial z^{i}_{0}} \Big), \quad 
			b_i = h^{TX} \Big( \frac{\partial}{\partial z^{i}_{0}}, \frac{\partial}{\partial z^{i}_{1}} \Big), \quad 
			c_i = h^{TX} \Big( \frac{\partial}{\partial z^{i}_{1}}, \frac{\partial}{\partial z^{i}_{1}} \Big).
		\end{equation}
		Then, by definition, we have
		\begin{equation}\label{eq_norm_det_ai_defn}
			\big \|
				d z^{i}_{0} \otimes d z^{i}_{1} 
			\big \|_{\Sigma_{X/S}/X}^{i}
			=
			\big( a_i c_i - |b_i|^2 \big)^{-1/2}(Q_i).
		\end{equation}
		We denote by $\norm{\cdot}_{\Sigma_{X/S}/X}$ the induced norm on the complex line $\otimes_{i = 1}^{k} ( \omega_{Y_0}|_{P_{2i - 1}} \otimes \omega_{Y_0}|_{P_{2i}} )$.
		\par Over $S$, we introduce the holomorphic line bundle 
		\begin{equation}
			\mathscr{L} 
			:= 
			\lambda(j^* \xi)^{12} 
			\otimes 
			\mathscr{O}_{S}(\Delta)^{2 \cdot \rk{\xi}}.
		\end{equation}
		We endow it with a norm
		\begin{equation}
			\norm{\cdot}_{\mathscr{L}} := \norm{\cdot}_Q(g^{TX_t}, h^{\xi})^{12} \otimes (\, \norm{\cdot}_{\Delta}^{{\rm{div}}})^{2 \cdot \rk{\xi}}.
		\end{equation}
		We bring the attention of the reader to the fact that the power of the divisor line bundle $\mathscr{O}_{S}(\Delta)$ in $\mathscr{L}$  is different from the construction (\ref{det_wol_prod}) (cf.  (\ref{quil_wol_norm})).
		This is due to the fact that the geometric setting in this section is different from Section \ref{sect_intro}, as here, for example, the assumption (\ref{suppos_restr}) is not satisfied for the metric induced by $g^{TX_t}$, thus, the Wolpert norm is not well-defined and the analogue of the norm (\ref{quil_wol_norm_rest}), doesn't make any sense.
		\par 
		We introduce the complex line
		\begin{equation}
			\mathscr{L}' 
			:= 
			\lambda(\rho^* \xi)^{12} 
			\otimes 
			(\otimes_{i = 1}^{2k} \det \rho^* (\xi)|_{P_i})^{6} 
			\otimes 
			(\otimes_{i = 1}^{k}(\omega_{Y_0}|_{P_{2i - 1}} \otimes \omega_{Y_0}|_{P_{2i}}))^{- 2 \cdot \rk{\xi}}.
		\end{equation}
		We denote by $\norm{\cdot}_{\mathscr{L}'}$ the norm on the complex line $\mathscr{L}'$ which is induced by $\norm{\cdot}_Q(g^{TY_0}, \rho^* (h^{\xi}))$, $h^{\xi}$ and $\norm{\cdot}_{\Sigma_{X/S}/X}$.
		Analogically to (\ref{isom_main}), one has the following canonical isomorphism
		\begin{equation}\label{isom_compact_case_bismut}
			\mathscr{L}|_{\Delta} 
			\to 
			\mathscr{L}' 
			\otimes 
			(\otimes_{i = 1}^{k} \mathscr{O}_{Q_i})^{12 \cdot \rk{\xi}}.
		\end{equation}
		Now we can state the main result of this section.
		\begin{thm}\label{thm_bismut_rest}
			The norm $\norm{\cdot}_{\mathscr{L}}^{{\rm{cmp}}}$ extends continuously over $S$. Moreover, under the isomorphism (\ref{isom_compact_case_bismut}), the following identity holds
			\begin{equation}\label{eq_bismut_rest}
				\norm{\cdot}_{\mathscr{L}}^{{\rm{cmp}}}|_{\Delta} = \exp \Big(\rk{\xi} \cdot k \cdot \big(24 \zeta'(-1) - 6 \log(2 \pi) \big) \Big) \cdot \norm{\cdot}_{\mathscr{L}'}^{{\rm{cmp}}}.
			\end{equation}
		\end{thm}
		\begin{proof}
			First of all, let's assume that $g^{TX}$ is Kähler. Then we argue that Theorem \ref{thm_bismut_rest} is just a restatement of \cite[Theorem 0.3]{BisDegQuil} due to Bismut.
			\par To see this, let's fix a holomorphic coordinate $t$ on $S$ such that $|\Delta| = \{ t = 0 \}$. We denote by $\norm{\cdot}_{\Delta}$ the Hermitian norm on $\mathscr{O}_S(\Delta)$, characterized by
			\begin{equation}\label{eq_defn_norm_delta}
				\| s_{\Delta} / t^k \|_{\Delta} = 1.
			\end{equation}
			As ${\rm{div}} (s_{\Delta}) = k \{ 0 \}$, we deduce that $\norm{\cdot}_{\Delta}$ is smooth over $S$.
			By the definition of the singular norm $\norm{\cdot}_{\Delta}^{{\rm{div}}}$ from (\ref{defn_norm_D}), by (\ref{eq_defn_norm_delta}), we have
			\begin{equation}\label{eq_rel_smooth_non_smooth}
				\norm{\cdot}_{\Delta}^{{\rm{div}}} = |t|^{-k} \cdot \norm{\cdot}_{\Delta}.
			\end{equation}
			By (\ref{eq_proj_gen_gft}), the isomorphism (\ref{isom_rest_w_div}) specifies in our case to
			\begin{equation}\label{eq_isom_poinc_bismut_case}
			\begin{aligned}
				&
				\mathscr{O}_S(\Delta)|_{|\Delta|} \to 
				\big(
				\otimes_{i = 1}^{k}(\omega_{Y_0}|_{P_{2i - 1}} \otimes \omega_{Y_0}|_{P_{2i}}) 
				\big)^{- 1} 
				,
				\\
				&
				\Big( \frac{s_{\Delta}}{t^k} \Big)|_{t = 0} \mapsto \big(\otimes_{i = 1}^{k} (dz^i_{0}|_{P_{2i - 1}} \otimes dz^i_{1}|_{P_{2i}} ) \big)^{-1}.
			\end{aligned}
			\end{equation}
			We denote by $\| d\pi^2 \|$ the norm of the isomorphism (\ref{eq_isom_poinc_bismut_case}). By (\ref{eq_defn_norm_delta}), we have
			\begin{equation}\label{eq_defn_dpi2}
				\norm{d\pi^2} 
				:= 
				\Big(
				\big \| 
					\otimes_{i = 1}^{k} (dz^i_{0}|_{P_{2i - 1}} \otimes dz^i_{1}|_{P_{2i}} ) 
				\big \|_{\Sigma_{X/S}/X}
				\Big)^{-1}.
			\end{equation}
			We note that due to our normalization of the $L^2$-norm, (\ref{defn_L_2}), the difference between our definition of the Quillen norm, and the one from \cite{BGS2}, \cite{BGS3}, \cite{BisDegQuil}, which we denote by $\norm{\cdot}_Q^{BGS}$, is 
			\begin{equation}
				\norm{\cdot}_Q(g^{TX_t}, h^{\xi}) = \exp \big( \log(2 \pi) \cdot \chi(X_t, \xi|_{X_t}) / 2 \big) \cdot \norm{\cdot}_Q^{BGS}(g^{TX_t}, h^{\xi}),
			\end{equation}
			where $\chi(X_t, \xi|_{X_t})$ is the Euler characteristic, given by
			\begin{equation}
				\chi(X_t, \xi|_{X_t}) = \dim H^0(X_t, \xi|_{X_t}) -  \dim H^1(X_t, \xi|_{X_t}).
			\end{equation}
			By Riemann-Roch theorem, the value $\chi(X_t, \xi|_{X_t})$ is topological, and thus, by Ehresmann theorem, it is constant over $S \setminus |\Delta|$.
			\par 
			We denote by $\norm{\cdot}_Q^{\xi}(g^{TY_0}, \rho^* (h^{\xi}))$ the norm on the complex line $\lambda(j^*\xi) \otimes (\otimes_{i = 1}^{2k} \det \xi|_{P_i})^{6}$ induced by $\norm{\cdot}_Q(g^{TY_0}, \rho^* (h^{\xi}))$  and $h^{\xi}$.
			Similarly, due to our normalization of the $L^2$-norm, (\ref{defn_L_2}), the difference between our definition of the norm $\norm{\cdot}_Q^{\xi}(g^{TY_0}, \rho^* (h^{\xi}))$, and the one from \cite{BGS2}, \cite{BGS3}, \cite{BisDegQuil}, which we denote by $\norm{\cdot}_Q^{\xi, BGS}(g^{TY_0}, \rho^* (h^{\xi}))$, is 
			\begin{equation}
				\norm{\cdot}_Q^{\xi}(g^{TY_0}, \rho^* (h^{\xi})) = \exp \big( \log(2 \pi) \cdot \chi(Y_0, \rho^*(\xi)|_{Y_0}) / 2 \big) \cdot \norm{\cdot}_Q^{\xi, BGS}(g^{TY_0}, \rho^* (h^{\xi})).
			\end{equation}
			\par We fix a smooth frame $\upsilon$ of $\lambda(j^*\xi)$ over $S$. Then by \cite[Theorem 0.3, (0.5), the fact that the genus $E$ is additive]{BisDegQuil}, under the isomorphisms (\ref{isom_det_restr}), (\ref{isom_kappa_sigma}), the following identity holds
			\begin{multline}\label{eq_lim_bis_gen_1111}
				\lim_{t \to 0} \Big( \log \big( \| \upsilon(t) \|_Q^{BGS}(g^{TX_t}, h^{\xi}) \big) - \frac{\rk{\xi}}{6} \log \big( \norm{s_{\Delta}(t)}_{\Delta} \big) \Big)
				\\
				=
				\log \big( \| \upsilon(0) \|_Q^{\xi, BGS}(g^{TY_0}, \rho^* (h^{\xi})) \big)
				+
				\frac{\rk{\xi}}{6} \log \norm{d\pi^2}  + 2 \zeta'(-1) \cdot k \cdot \rk{\xi}.
			\end{multline}
			Now, by (\ref{eq_sh_ex_seq}) and the induced long exact sequence, we deduce
			\begin{equation}
				\chi(X_t, \xi|_{X_t}) = \chi(Y_0, \rho^*(\xi)) - k \cdot \rk{\xi}.
			\end{equation}
			However, by (\ref{eq_rel_smooth_non_smooth}), (\ref{eq_defn_dpi2}) and (\ref{eq_lim_bis_gen_1111}), we deduce that 
			\begin{multline}\label{eq_final_proof_bismut}
				\lim_{t \to 0} \Big( \log\big( \| \upsilon(t) \|_Q(g^{TX_t}, h^{\xi}) \big)
				+
				\frac{\rk{\xi}}{6} \log \Big( \Big \| \frac{s_{\Delta}(t)}{t^k}  \Big \|_{\Delta}^{{\rm{div}}} \Big) \Big)
				=
				 \log\big( \| \upsilon(0) \|_Q^{\xi}(g^{TY_0}, \rho^* (h^{\xi})) \big)
				\\
				-
				\frac{\rk{\xi}}{6} 
				\log  \Big( \big \| 
					\otimes_{i = 1}^{k} 
					(
						dz^i_{0}|_{P_{2i - 1}}
						\otimes 
						dz^i_{1}|_{P_{2i}}
					) 
				\big \|_{\Sigma_{X/S}/X} \Big) 
				\\
				+ \Big( 2 \zeta'(-1) - \frac{\log(2\pi)}{2} \Big) \cdot k \cdot \rk{\xi},
			\end{multline}
			which means, in particular, that the norm $\norm{\cdot}_{\mathscr{L}}$ extends continuously over $S$. 
			Moreover, as by (\ref{eq_isom_poinc_bismut_case}), the isomorphism (\ref{isom_compact_case_bismut}) is given in our situation by
			\begin{equation}
				\Big( 
				\upsilon^{12} 
				\otimes 
				\Big(
				\frac{s_{\Delta}}{t^k}
				\Big)^{2 \cdot \rk{\xi}}
				\Big)|_{|\Delta|} 
				\mapsto 
				\upsilon(0)^{12} 
				\otimes 
				\Big( 
				\otimes_{i = 1}^{k} 
				\big(
					dz^i_{0}|_{P_{2i - 1}}
					\otimes 
					dz^i_{1}|_{P_{2i}}
				\big) 
				\Big)^{-2 \cdot \rk{\xi}},
			\end{equation}
			the continuous extension satisfies (\ref{eq_bismut_rest}) by (\ref{eq_final_proof_bismut}). 
			\par Now let's prove (\ref{eq_bismut_rest}) for metric $g^{TX}_{0}$, which is not necessarily Kähler.
			We note that $\pi$ is locally projective (cf. Bismut-Bost \cite[Proposition 3.4]{BisBost}), thus for some small neighborhood $U$ of $0 \in S$, we may find a Kähler metric $g^{TX}$ over $\pi^{-1}(U)$. 
			As the statement of Theorem \ref{thm_bismut_rest} is local over the base, without losing the generality, we suppose from now on that $g^{TX}$ is defined over $X$.
			We denote by $\norm{\cdot}_{\mathscr{L}}^{0}$ the norm on $\mathscr{L}$, induced by $g^{TX}_{0}$. 
			The idea of the proof is to use the statement (\ref{eq_final_proof_bismut}) and the anomaly formula of Bismut-Gillet-Soulé \cite{BGS3} (cf. Theorem \ref{thm_anomaly_cusp} for $m = 0$) to relate the norms $\norm{\cdot}_{\mathscr{L}}^{0}$ and $\norm{\cdot}_{\mathscr{L}}$, and to study the limit of the right-hand side of this formula near the locus of singular curves.
			\par 
			We denote by $\norm{\cdot}_{\Sigma_{X/S}/X}^{0}$ the norm on the line bundle $\otimes_{i = 1}^{k}(\omega_{Y_0}|_{P_{2i - 1}} \otimes \omega_{Y_0}|_{P_{2i}})$, induced by $g^{TX}_{0}$ as in (\ref{eq_norm_det_ai_defn}). 
			Similarly to (\ref{eq_defn_ai_bi_ci}), we denote by $a_i^{0}$, $b_i^{0}$, $c_i^{0}$ the analogical functions associated with $g^{TX}_{0}$. 
			\par We argue that without losing the generality, we may suppose that $a_i^{0}, c_i^{0} = 1, b_i^{0} = 0$. 
			This is true since we could fix a Riemannian metric $g^{TX}_{*}$ which is compatible with the complex structure satisfying this assumption and then simply apply Theorem \ref{thm_bismut_rest} twice for $g^{TX}_{*}$ and $g^{TX}$ and for $g^{TX}_{*}$ and $g^{TX}_0$.
			By combining the two results, we would get the original statement.
			\par
			Now, by (\ref{eq_norm_det_ai_defn}), we trivially have
			\begin{equation}\label{eq_norms_det_compar1}
				2\log \Big( \norm{\cdot}_{\Sigma_{X/S}/X} / \norm{\cdot}_{\Sigma_{X/S}/X}^{0}  \Big)
				=	
				- \sum_{i = 1}^{k} \log \big( a_i c_i - |b_i|^2 \big)(Q_i).
			\end{equation}
			Let the differential form $F$ on $X$ be given by
			\begin{equation}
				F = \widetilde{\td} \big(TX/S, g^{TX/S}, g^{TX/S}_{0} \big) \ch \big(\xi, h^{\xi} \big),
			\end{equation}
			where $TX/S$ is the vertical tangent bundle of $\pi$, and $g^{TX/S}, g^{TX/S}_{0}$ are the Hermitian norms on $TX/S$ induced by $g^{TX}, g^{TX}_{0}$.
			By the anomaly formula of Bismut-Gillet-Soulé \cite{BGS3} (cf. Theorem \ref{thm_anomaly_cusp} for $m = 0$), over $X \setminus \Sigma_{X/S}$, we have
			\begin{equation}\label{eq_anomal_appl_norml}
				\log \Big( \norm{\cdot}_{\mathscr{L}} / \norm{\cdot}_{\mathscr{L}}^{0}  \Big)(t) 
				= 
				6 \int_{X_t} F.
			\end{equation}
			Now, as the map $\pi$ is a submersion away from $\Sigma_{X/S}$, and the metrics $g^{TX}$, $g^{TX}_{0}$ are smooth over $X$, by (\ref{eq_final_proof_bismut}), (\ref{eq_norms_det_compar1}) and (\ref{eq_anomal_appl_norml}), to prove Theorem \ref{thm_bismut_rest}, it is enough to prove that for any $i = 1,\ldots, k$, the following holds
			\begin{equation}\label{eq_enough_prove_bismut}
				\lim_{\epsilon \to 0} \lim_{t \to 0} \int_{X_t \cap U(Q_i, \epsilon)} F 
				= 
				\frac{\rk{\xi}}{6} 
				\log \big( a_i c_i - |b_i|^2 \big)(Q_i).
			\end{equation}
			For brevity, we fix $1 \leq i \leq k$, and denote $z_0 := z_0^{i}$, $z_1 := z_1^{i}$.
			As $z_0 \frac{\partial}{\partial z_0} - z_1 \frac{\partial}{\partial z_1}$ is a local holomoprhic frame of $TX/S$, locally around $Q_i$, we have
			\begin{equation}
				g^{TX/S} \Big( z_0 \frac{\partial}{\partial z_0} - z_1 \frac{\partial}{\partial z_1}, 
				z_0 \frac{\partial}{\partial z_0} - z_1 \frac{\partial}{\partial z_1} \Big)
				=
				a_i |z_0|^{2} + c_i |z_1|^{2} - b_i z_0 \overline{z}_1 -  \overline{b}_i z_1 \overline{z}_0.
			\end{equation}
			By using the fact that $z_0 z_1 = t$ over $X_t$, we deduce that locally around $Q_i$, we have
			\begin{multline}\label{eq_c1_form_pullback_smooth}
				c_1(TX/S, g^{TX})|_{X_t} =
				\frac{\partial \dbar}{2 \pi \imun} \Big( \log 
				\big(
					a_i |z_0|^{2} + c_i |z_1|^{2} - b_i z_0 \overline{z}_1 -  \overline{b}_i z_1 \overline{z}_0
				\big) 
				\Big)
				\\
				= \frac{4 (a_i c_i - |b_i|^2) |z_0|^2 |t|^2}{(a_i |z_0|^{4} + c_i |t|^{2} - b_i z_0^2 \overline{t} - \overline{b}_i \overline{z}_0^2 t )^2} \frac{dz_0 d \overline{z}_0}{2 \pi \imun} + 
				o \bigg(
					\Big( \frac{|t|^2}{|z_0|^{6}} + \frac{|z_0|^2}{|t|^{2}} \Big)dz_0 d\overline{z}_0
				\bigg).
			\end{multline}
			As we are only interested in the limit (\ref{eq_enough_prove_bismut}), the calculation localizes around $Q_i$, and only the highest order terms matter. Thus, in the calculations, we may suppose that $a_i, b_i, c_i$ are constants equal to $a_i(Q_i), b_i(Q_i), c_i(Q_i)$ respectively, and we may suppress the $o$-term.
			\par 
			By making the change of variables $y_0  = z_0 |t|^{-1/2}$ and using (\ref{ch_bc_0}), (\ref{ch_bc_2}), we deduce that as $t \to 0$ and $\epsilon \to 0$, we have
			\begin{multline}
				\int_{X_t \cap U(Q_i, \epsilon)} F 
				\sim
				-
				\frac{\rk{\xi}}{12}
				\int_{|t|^{1/2} / \epsilon < |y_0| < \epsilon / |t|^{1/2}}
				\bigg(
				\log 
					\frac{a_i |y_0|^{4} + c_i - b_i y_0^2 - \overline{b}_i \overline{y}_0^2 }{|y_0|^{4} + 1}
				\bigg)
				\\
				\cdot
				\bigg(
					\frac{(4 a_i c_i  - 4 |b_i|^2) |y_0|^2}{(a_i |y_0|^{4} + c_i - b_i y_0^2 - \overline{b}_i \overline{y}_0^2 )^2}
					+
					\frac{4 |y_0|^2}{(|y_0|^{4} + 1)^2}
				\bigg)
				\frac{dy_0 d \overline{y}_0}{2 \pi \imun}.
			\end{multline}
			We make the change of variables $x_0 := y_0^2$. The variable $x_0$ “turns around" $\comp$ twice, which kills one of the two additional factors $2$ appearing in the denominator in the following identity
			\begin{multline}\label{eq_bismut_thm_final_statement}
				\int_{X_t \cap U(Q_i, \epsilon)} F 
				\sim
				- \frac{\rk{\xi}}{24}
				\int_{|t|^{1/4} / \epsilon^{1/2} < |x_0| < \epsilon^{1/2} / |t|^{1/4}}
				\bigg(				
				\log 
					\frac{a_i |x_0|^{2} + c_i - b_i x_0 - \overline{b}_i \overline{x}_0 }{|x_0|^{2} + 1}
				\bigg)
				\\
				\cdot
				\bigg(
					\frac{(4 a_i c_i - 4 |b_i|^2) }{(a_i |x_0|^{2} + c_i - b_i x_0 - \overline{b}_i \overline{x}_0 )^2}
					+
					\frac{4}{(|x_0|^{2} + 1)^2}
				\bigg)
				\frac{dx_0 d \overline{x}_0}{2 \pi \imun}.
			\end{multline}
			By making the change of variables $x_0 := x_0 \cdot \exp(\arg(b_i))$, we see that the right-hand side of (\ref{eq_bismut_thm_final_statement}) depends only on $|b_0| \in \real$. 
			Thus, without losing the generality, we may assume that $b_0 \in \real$.
			By writing $x_0 = x + \imun y$, from (\ref{eq_bismut_thm_final_statement}), we see
			\begin{multline}\label{eq_bismut_thm_final_statement2}
				\int_{X_t \cap U(Q_i, \epsilon)} F 
				\sim
				\frac{\rk{\xi}}{24 \pi}
				\int_{|t|^{1/4} / \epsilon^{1/2} < x^2 + y^2 < \epsilon^{1/2} / |t|^{1/4}}
				\bigg(
				\log 
					\frac{a_i (x^2 + y^2) + c_i - 2 b_i x}{x^2 + y^2 + 1}
				\bigg)
				\\
				\cdot
				\bigg(
					\frac{(4 a_i c_i - 4 b_i^2) }{(a_i (x^2 + y^2) + c_i - 2 b_i x)^2}
					+
					\frac{4}{(x^2 + y^2 + 1)^2}
				\bigg)
				dx dy.
			\end{multline}
			Now, we remark that by switching to polar coordinates, changing the integration over radius by the integration over its square and applying tedious derivation by parts, for $a, c > 0$, $b \in \real$, $ac - b^2 > 0$, we get the following identity
			\begin{multline}\label{eq_double_integral_difficult}
				\int_{-\infty}^{+\infty} 
				\int_{-\infty}^{+\infty} 
				\bigg(
				\log
					\frac{a (x^2 + y^2) + c  - 2 b x }{x^2 + y^2 + 1}
				\bigg)
				\bigg(
					\frac{4 ac - 4 b^2}{(a (x^2 + y^2) + c  - 2 b x )^2}
					+
					\frac{4}{(x^2 + y^2 + 1)^2}
				\bigg)
				dx d y 
				\\
				= 4 \pi \log(ac - b^2).
			\end{multline}
			By (\ref{eq_bismut_thm_final_statement2}) and (\ref{eq_double_integral_difficult}), we get (\ref{eq_enough_prove_bismut}), which finishes the proof of Theorem \ref{thm_bismut_rest}.
		\end{proof}

	\subsection{Univesality in restriction theorem, proof of Theorems  \ref{thm_full_comp_pert}, \ref{thm_imm_thm_upto_const}}\label{sect_final_sect}	
	The goal of this section is to prove Theorems \ref{thm_full_comp_pert}, \ref{thm_imm_thm_upto_const}. 
	\par
	In fact, during the proof of Theorem \ref{thm_imm_thm_upto_const}, we will establish Theorem \ref{thm_full_comp_pert} for some metric constructed in a very special way. After this, to prove Theorem \ref{thm_full_comp_pert} we will only need to prove that its statement doesn't depend on the choice of the metric. To see this, we use the anomaly formula, Theorem \ref{thm_anomaly_cusp}, and relative compact perturbation theorem, Theorem \ref{thm_comp_appr}.
	\par 
	The proof of Theorem \ref{thm_imm_thm_upto_const} is based on Theorem \ref{thm_bismut_rest}.
	\begin{proof}[Proof of Theorem \ref{thm_imm_thm_upto_const}]
			The main idea of the proof is to construct a metric $\norm{\cdot}_{X/S}^{\varkappa}$ on $\omega_{X/S}$ over $X \setminus |\pi^{-1}(\Delta)|$, for which the assumptions (\ref{suppos_s3}), (\ref{suppos_restr}) hold and to show that Theorem \ref{thm_imm_thm_upto_const} holds for $\norm{\cdot}_{X/S}^{\varkappa}$ but instead of $C_{-n}$, we have an undetermined constant $A_{-n}$. 
			Then, by the anomaly formula and Proposition \ref{prop_int_loglog}, we deduce that Theorem \ref{thm_imm_thm_upto_const} holds for any norm $\norm{\cdot}_{X/S}^{\omega}$, satisfying the assumptions (\ref{suppos_s3}), (\ref{suppos_restr}).
			\par The fact that the constant $A_{-n}$ is universal is a consequence of a number of facts, which we now state.
			First, from Proposition \ref{prop_coord}, we know that the complex structure near $\Sigma_{X/S}$ is universal, and doesn't depend on  $\pi$, $X$, $\ldots$
			Second, the metric $\norm{\cdot}_{X/S}^{\varkappa}$ in the neighborhood of $\Sigma_{X/S}$ is also universal by the construction, and the construction of it in this neighborhood is local and, thus, doesn't depend on $\pi$, $X$, $\ldots$ as well.
			Third, we prove that the metric $\norm{\cdot}_{X/S}^{\varkappa}$ can be modified in the neighborhood of $\Sigma_{X/S}$ to another metric $\norm{\cdot}_{X/S}^{\sim}$ which satisfies the assumptions of Theorem \ref{thm_bismut_rest}.
			Fourth, remark that the constant analogical to $A_{-n}$ in Theorem \ref{thm_bismut_rest} is universal (it is equal to $24 \zeta'(-1) - 6 \log(2 \pi)$).
			Finally, the difference between the Quillen metric on the singular fiber associated with $\norm{\cdot}_{X/S}^{\sim}$ according to the procedure as in Theorem \ref{thm_bismut_rest} and the Quillen metric on the same fiber associated with $\norm{\cdot}_{X/S}^{\varkappa}$ according to the procedure as in Theorem \ref{thm_imm_thm}, is universal by Theorem \ref{thm_comp_appr}.
			\par 
			More precisely, we proceed in the following way. First, we construct a Riemannian metric $g^{TX}_{\sim}$ compatible with the complex structure on $X$. 
			This metric will be trivial in the neighborhood of $\Sigma_{X/S}$. 
			By using Theorem \ref{thm_bismut_rest}, we will study the asymptotics of the norm $\norm{\cdot}_{\mathscr{L}_n}^{\sim}$ induced on the line bundle $\mathscr{L}_n$ by $\norm{\cdot}_{X/S}^{\sim}$.
			Then by modifying locally this metric in the neighborhood of $\Sigma_{X/S}$, we construct a family of metrics $g^{TX_t}_{\varkappa}$ on $X_t$ for $t \in S \setminus |\Delta|$, which degenerates to the hyperbolic metric at the singular fiber through the family of degenerating hyperbolic cylinders.
			The construction of the metric $g^{TX_t}_{\varkappa}$ is highly motivated by Theorem \ref{thm_wolp_exp}. 
			By applying anomaly theorem and the previous result on the asymptotics of $\norm{\cdot}_{\mathscr{L}_n}^{\sim}$, we compute the asymptotics of the norm $\norm{\cdot}_{\mathscr{L}_n}^{\varkappa}$ induced on the line bundle $\mathscr{L}_n$ by $g^{TX_t}_{\varkappa}$. As our construction is local around $\Sigma_{X/S}$, the asymptotic is independent of the geometry of $\pi : X \to S\ldots$
			\par 
			Now, let local coordinates $z^{j}_{0}, z^{j}_{1}$ (cf. (\ref{eq_defn_pi})) around $Q_j$ be as in (\ref{eq_proj_gen_gft}). 
			We suppose that $D(1) \subset \Im (z^{j}_{0}) \cap \Im (z^{j}_{1})$. 
			This is merely a question of normalization, as for any coordinates $z^{j}_{0}, z^{j}_{1}$ of $X$ and $t$ of $S$, and for any $a \in \comp$, we may change the coordinates by $a \cdot z^{j}_{0}, a \cdot z^{j}_{1}$ and $a^2 \cdot t$, and the identity  (\ref{eq_proj_gen_gft}) would be preserved.
		We specify the function $\nu$ from (\ref{eq_defn_nu}) as follows
		\begin{equation}\label{eq_defn_nu_specific}
			\nu(x) = 
			\begin{cases} 
      			\hfill 0, & \text{ for } x \in X \setminus (\cup_{i = 1}^{k} U(Q_i, 1)), \\
      			\hfill 1- \nu_0(|z^{i}_{0}|^2 + |z^{i}_{1}|^2 ),  & \text{ for } x \in U(Q_i, 1). 
 			\end{cases}
		\end{equation}
		By (\ref{defn_nu}), the function (\ref{eq_defn_nu_specific}) satisfies (\ref{eq_defn_nu}) for $\epsilon = 1/2$.
			\par 
			We note that $\pi$ is locally projective (cf. Bismut-Bost \cite[Proposition 3.4]{BisBost}), thus, there is a neighborhood $U$ of $0 \in S$ such that there is a Kähler metric $g^{TX}_{0}$ over $\pi^{-1}(U)$. 
			As the statement of Theorem \ref{thm_bismut_rest} is local over the base, without losing the generality, we may suppose from now on that $g^{TX}_{0}$ is defined over $X$.
			\begin{sloppypar}
				We define the Riemannian metric $g^{TX}_{\sim}$ over $X$ so that it coincides with $g^{TX}_0$ over $X \setminus (\cup_{i = 1}^{k} U(Q_i, 1))$, and over $U(Q_i, 1)$ it is given by 
				\begin{equation}\label{defn_gtx_1}
					g^{TX}_{\sim} = (1 - \nu) \cdot g^{TX}_0 + \nu  \cdot ( |dz^i_{0}|^2 + |dz^i_{1}|^2 ).
				\end{equation}
				We note that $g^{TX}_{\sim}$ is not necessarily Kähler, but it is trivially compatible with the complex structure of $X$.
				We denote by $g^{TX_t}_{\sim}$ the induced Riemannian metric on $X_t$, $t \in S \setminus |\Delta|$, and by $g^{TY_0}_{\sim}$ the induced Riemannian metric on $Y_0$, constructed as in Section \ref{sect_bismut_proof}, i.e.  by $g^{TY_0}_{\sim} = \rho^*(g^{TX}_{\sim})$, where $\rho : Y_0 \to X_0$ is the normalization map. 
				Since $g^{TX}_{\sim}$ is compatible with the complex structure of $X$, we see that the metrics $g^{TX_t}_{\sim}$, $g^{TY_0}_{\sim}$ are Kähler. 
			\end{sloppypar}
			\begin{sloppypar}
				We endow $\omega_{X/S}$ with the Hermitian norm $\norm{\cdot}_{X/S}^{\sim, {\rm{ind}}}$ induced by $g^{TX}_{\sim}$ over $X \setminus \Sigma_{X/S}$. 
			Let $\widetilde{\nu}: X \to [0, 1]$ be defined as $\nu$ in (\ref{eq_defn_nu_specific}), where in place of $\nu_0(\cdot)$, we put $\nu_0(4 \cdot)$. 
			Then $\widetilde{\nu}(x) = 1$ for $x \in X \setminus (\cup_{i = 1}^{k} U(Q_i, 1/2))$.
			We define the Hermitian norm $\norm{\cdot}_{X/S}^{\sim}$ on $\omega_{X/S}$ over $X$ as follows. Over $X \setminus (\cup_{i = 1}^{k} U(Q_i, 1/2))$, we demand it to be equal to $\norm{\cdot}_{X/S}^{\sim, {\rm{ind}}}$, and over $U(Q_i, 1/2)$, we define it by
			\begin{equation}\label{eq_defn_tilde_norm}
				\big\| d z^i_{0} / z^i_{0} \big\|_{X/S}^{\sim} = (1 - \widetilde{\nu}) \cdot \big\| d z^i_{0} / z^i_{0} \big\|_{X/S}^{\sim, {\rm{ind}}} + \widetilde{\nu}.
			\end{equation}
			Trivially, the Hermitian norm $\norm{\cdot}_{X/S}^{\sim}$ on $\omega_{X/S}$ is smooth over $X$.
			Moreover, it is trivial on $\cup_{i = 1}^{k} U(Q_i, 1/4)$.
			The norm $\norm{\cdot}_{Y_0}^{\sim} := \rho^* (\, \norm{\cdot}_{X/S}^{\sim})$ over $\omega_{Y_0}(D)$ is characterized as follows: over $Y_0 \setminus ( \cup_{i = 1}^{k} (\{ |z^i_{0}| < 1/2 \} \cup \{ |z^i_{1}| < 1/2 \}))$, it is induced by $g^{TY_0}_{\sim}$ as in Construction \ref{const_norm_div}, over $\{ |z^i_{j}| < 1/2 \}$, $j = 0,1$, $i = 1, \ldots, k$, it is given by
			\begin{equation}
				\big\| d z^i_{j} \otimes s_{D_{Y_0}} / z^i_{j} \big\|_{Y_0}^{\sim} 
				= 
				(1 - (\widetilde{\nu} \circ \rho)) \cdot \big\| d z^i_{j} / z^i_{j} \big\|_{X/S}^{\sim, {\rm{ind}}} + (\widetilde{\nu} \circ \rho).
			\end{equation}
			\end{sloppypar}
			\begin{figure}[h]
				\centering
				\includegraphics[width=0.7\textwidth]{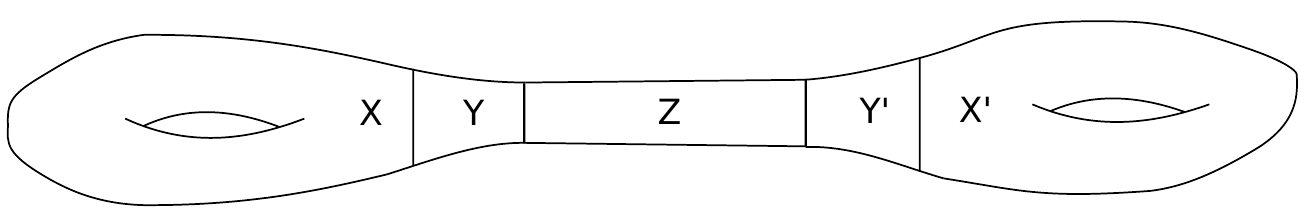}	
				\caption{The metric $g^{TX_t}_{\sim}$. Over the regions $X, X'$, it is induced by $g^{TX}_{0}$. Over the regions $Y, Y'$, it is an interpolation between $g^{TX}_{0}$ and $|dz_0|^2 + |dz_1|^2$, and over the region $Z$, it is given by $|dz_0|^2 + |dz_1|^2$, where $z_0 z_1 = t$.
				}
				\label{fig_graft}
			\end{figure}
			\par 
			We endow $\mathscr{L}_n$ with the metric $\norm{\cdot}_{\mathscr{L}_n}^{\sim}$, induced by the Quillen norm $\norm{\cdot}_Q(g^{TX_t}_{\sim}, h^{\xi} \otimes (\, \norm{\cdot}_{X/S}^{\sim})^{2n} )$ and the singular norm (\ref{defn_norm_D}).
			We endow $\mathscr{L}_n'$ with the norm $\norm{\cdot}_{\mathscr{L}_n'}^{\sim}$, induced by the Quillen norm $\norm{\cdot}_Q(g^{TY_0}_{\sim}, \rho^*(h^{\xi}) \otimes (\, \norm{\cdot}_{Y_0}^{\sim})^{2n} )$ and the norm $\norm{\cdot}_{\Sigma_{X/S}/X}^{\sim}$ (see (\ref{eq_norm_det_ai_defn})) on $\otimes_{i = 1}^{k}(\omega_{Y_0}|_{P_{2i - 1}} \otimes \omega_{Y_0}|_{P_{2i}} )$. By  (\ref{defn_gtx_1}),  the norm $\norm{\cdot}_{\Sigma_{X/S}/X}^{\sim}$ is characterized by the identity
			\begin{equation}\label{eq_char_norm_sigma}
				\big\|
					\otimes_{i = 1}^{k}(d z^i_{0}|_{P_{2i - 1}} \otimes d z^i_{1}|_{P_{2i}} )
				\big\|_{\Sigma_{X/S}/X}^{\sim}
				=
				1.
			\end{equation}
			The metrics $g^{TX_t}_{\sim}$, $\norm{\cdot}_{X/S}^{\sim}$ and $g^{TY_0}_{\sim}$, $\norm{\cdot}_{Y_0}^{\sim}$ satisfy the hypothesis of Theorem \ref{thm_bismut_rest}. Let
			\begin{equation}\label{eq_an_prone}
				A_{-n}' := 24 \zeta'(-1) - 6 \log(2\pi).
			\end{equation}
			By Theorem \ref{thm_bismut_rest} and (\ref{eq_char_norm_sigma}), for a smooth local frame $\upsilon$ of $\mathscr{L}_n$, under the isomorphism (\ref{isom_main}), we have
			\begin{equation}\label{eq_final_proof_bismut}
				\lim_{t \to 0} \Big( \log\big( \|  \upsilon(t) \|_{\mathscr{L}_n}^{\sim} \big)
				-
				2 k \cdot \rk{\xi} \cdot \log |t| \Big)
				=
				 \log\big( \| \upsilon(0) \|_{\mathscr{L}'_n}^{\sim} \big)
				+
				k \cdot \rk{\xi} \cdot A_{-n}'.
			\end{equation}
			\par Let's now modify the metric $g^{TX_t}_{\sim}$ to $g^{TX_t}_{\varkappa}$ so that it would satisfy the assumptions of the Theorem \ref{thm_imm_thm}.
			We define the metrics $g^{TX_t}_{\varkappa}$ on $X_t$, $t \in S \setminus |\Delta|$, as follows: over $X_t \setminus (\cup_{i = 1}^{k} U(Q_i, 1/2))$ it coincides with $g^{TX_t}_{\sim}$, and over $U(Q_i, 1/2)$ it is given by
			\begin{equation}\label{eq_defn_gtxt_1}
				g^{TX_t}_{\varkappa} := (1 - \widetilde{\nu}) \cdot g^{TX_t}_{\sim} + \widetilde{\nu} \cdot  g^{\text{Cyl}}_{i, t},
			\end{equation}
			where the metric $g^{\text{Cyl}}_{j, t}$ was defined in (\ref{defn_g_cyl1}).
			We also define the metric $g^{TY_0}_{\varkappa}$ as follows: over $Y_0 \setminus (\cup_{i = 1}^{k} U(Q_i, 1/2))$ it coincides with $g^{TY_0}_{\sim}$, and over $U(Q_i, 1/2)$ it is given by
			\begin{equation}\label{eq_defn_tgty_0_2}
				g^{TY_0}_{\varkappa}
				:= 
				(1 - (\widetilde{\nu} \circ \rho)) \cdot g^{TY_0}_{\sim} 
				+
				(\widetilde{\nu}  \circ \rho) 
				\cdot
				\big( g^{\rm{Poinc}}_{i, 0} +  g^{\rm{Poinc}}_{i, 1} \big),
			\end{equation}
			where the metrics $g^{\rm{Poinc}}_{i, 0}$, $g^{\rm{Poinc}}_{i, 1}$ are the metrics induced by the Poincaré metric (\ref{reqr_poincare}) with respect to the coordinates $z^{i}_{0}$ and $z^{i}_{1}$. 
			We denote by $\norm{\cdot}_{X/S}^{\varkappa}$ the Hermitian norm on $\omega_{X/S}$ induced by $g^{TX_t}_{\varkappa}$.
			By (\ref{eq_defn_gtxt_1}) (cf. Proposition \ref{prop_gft_good}), we see that the Hermitian norm $\norm{\cdot}_{X/S}^{\varkappa}$ extends continuously over $X \setminus \Sigma_{X / S}$, and the assumptions (\ref{suppos_s3}) are satisfied.
			We define the norm $\norm{\cdot}_{Y_0}^{\varkappa}$ on $\omega_{Y_0}(D)$ as follows
			\begin{equation}\label{eq_norm_y_0_1_last_pf}
				\norm{\cdot}_{Y_0}^{\varkappa} = \rho^* (\, \norm{\cdot}_{X/S}^{\varkappa}).
			\end{equation}
			Then we see trivially that $\norm{\cdot}_{X/S}^{\varkappa}$ satisfies assumptions (\ref{suppos_restr}), and by (\ref{eq_defn_gtxt_1}), (\ref{eq_defn_tgty_0_2}), the associated metric on $Y_0 \setminus D_{Y_0}$, constructed as in Section 1, coincides with $g^{TY_0}_{\varkappa}$.
			\par Let's pause and explain this construction. The metrics $g^{TX_t}_{\varkappa}$ degenerate near the singular fibers to a metric with cusps in the similar way as the hyperbolic metrics (see Theorem \ref{thm_wolp_exp}).
			The advantage of the metrics $g^{TX_t}_{\varkappa}$ over the hyperbolic one is that over the region $\cup_{i = 1}^{k} U(Q_i, 1/2)$, it is independent of any exterior data (as $\pi: X \to S$), and over $X_t \setminus (\cup_{i = 1}^{k} U(Q_i, 1/2))$, the metric $g^{TX_t}_{\varkappa}$ coincides with a metric $g^{TX_t}_{\sim}$, for which Theorem \ref{thm_bismut_rest} holds.
			This means that to get the asymptotic near the singular fibers of the Hermitian norm $\norm{\cdot}_{\mathscr{L}_n}^{\varkappa}$ on the holomorphic line bundle $\mathscr{L}_n$, induced by the Quillen norm $\norm{\cdot}_Q(g^{TX_t}_{\varkappa}, h^{\xi} \otimes (\, \norm{\cdot}_{X/S}^{\varkappa})^{2n} )$ and the singular norm (\ref{defn_norm_D}), it is enough to apply the anomaly formula (the contribution of which is local near the set of singular points) and use (\ref{eq_final_proof_bismut}). 
			Let's make it more precise.
			\begin{figure}[h]
				\centering
				\includegraphics[width=0.7\textwidth]{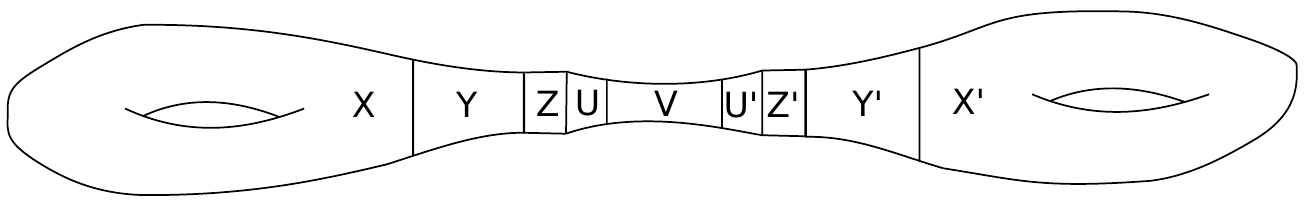}	
				\caption{The metric $g^{TX_t}_{\varkappa}$. Over the regions $X, Y, Z, Z', Y', X'$ it coincides with $g^{TX_t}_{\sim}$.
				Over the regions $U, U'$, it is an interpolation between $g^{TX_t}_{\sim}$ and the hyperbolic cylinder metric, (\ref{defn_g_cyl1}), and over the region $V$, it coincides with the hyperbolic cylinder metric, (\ref{defn_g_cyl1}).
				}
				\label{fig_graft}
			\end{figure}
			\par 
			Let $\norm{\cdot}^{W, \varkappa}_{Y_0}$ be the Wolpert norm on $\otimes_{i = 1}^{k}(\omega_{Y_0}|_{P_{2i - 1}} \otimes \omega_{Y_0}|_{P_{2i}} )$.
			By Definition \ref{defn_wolpert_norm} and (\ref{eq_defn_tgty_0_2})
			\begin{equation}\label{eq_char_wolp_norm_last}
				\big\|
					\otimes_{i = 1}^{k}(d z^i_{0}|_{P_{2i - 1}} \otimes d z^i_{1}|_{P_{2i}} )
				\big\|^{W, \varkappa}_{Y_0}
				=
				1.
			\end{equation}
			\par 
			We endow the complex line $\mathscr{L}_n'$ with the Hermitian norm $\norm{\cdot}_{\mathscr{L}_n'}^{\varkappa}$, induced by the Quillen norm $\norm{\cdot}_Q(g^{TY_0}_{\varkappa}, \rho^*(h^{\xi}) \otimes (\, \norm{\cdot}_{Y_0}^{\varkappa})^{2n} )$ and the Wolpert norm $\norm{\cdot}^{W, \varkappa}_{Y_0}$.
			\par 
			Recall that the regularized integral was defined in (\ref{eq_defn_reg_integ}).
			For $n \in \nat$ we define $A_{-n}'' \in \real$ as follows
			\begin{multline}\label{eq_a_n_2_prime}
				A_{-n}'' := 6 \cdot k^{-1} \cdot \int_{Y_0}^{\textbf{r}} \Big( 
			 					\widetilde{\td} \big(\omega_{Y_0}^{-1}, g^{TY_0}_{\sim}, g^{TY_0}_{\varkappa} \big)  \ch \big(\omega_{Y_0}(D)^n,  (\, \norm{\cdot}^{\sim}_{Y_0})^{2n} \big) 
			 					\\
								+			 		 					
					 			\td \big(\omega_{X/S}^{-1}, g^{TY_0}_{\varkappa} \big) \widetilde{\ch} \big(\omega_{Y_0}(D)^n, (\, \norm{\cdot}^{\sim}_{Y_0})^{2n}, (\, \norm{\cdot}^{\varkappa}_{Y_0})^{2n} \big) \Big).
			\end{multline}
			As we will see further, $A_{-n}''$ depends only on $n \in \nat$ and the choice of $\nu_0$, and does not depend on any exterior data as $\pi : X \to S \ldots$
			We would like to prove that the following holds
			\begin{equation}\label{eq_anom_thrg_g_lim}
				\lim_{t \to 0}
				\Big(
				\log \big( 
				\norm{\cdot}_{\mathscr{L}_n}^{\varkappa} / \norm{\cdot}_{\mathscr{L}_n}^{\sim}
				\big)(t)
				+
				k \cdot \rk{\xi} \cdot \log |t|
				\Big)
				=
				k \cdot \rk{\xi} \cdot A_{-n}''.
			\end{equation}
			\par Assume we have accomplished proving (\ref{eq_anom_thrg_g_lim}).
			Then, by Theorem \ref{thm_comp_appr}, (\ref{eq_char_norm_sigma}), (\ref{eq_defn_tgty_0_2}), (\ref{eq_norm_y_0_1_last_pf}) and (\ref{eq_char_wolp_norm_last}), we deduce that the quantity 
			\begin{equation}\label{eq_limit_almost_final_last}
				(k \cdot \rk{\xi} )^{-1} \log \big( 
				\norm{\cdot}_{\mathscr{L}_n'}^{\varkappa} / \norm{\cdot}_{\mathscr{L}_n'}^{\sim}
				\big)
				=:
				A_{-n}'''
			\end{equation}
			depends only the choice of the function $\nu_0$.
			Thus, by (\ref{eq_final_proof_bismut}), (\ref{eq_anom_thrg_g_lim}) and (\ref{eq_limit_almost_final_last}), we deduce that under the isomorphism (\ref{isom_main}), the following holds
			\begin{equation}\label{eq_thm_for_1_norms}
				\norm{\cdot}_{\mathscr{L}_n}^{\varkappa}|_{|\Delta|} = \exp(k \cdot \rk{\xi} \cdot A_{-n}) \cdot \norm{\cdot}_{\mathscr{L}'_n}^{\varkappa} \quad \text{with} \quad A_{-n} := A_{-n}' + A_{-n}'' - A_{-n}'''.
			\end{equation}
			The essential difference between (\ref{eq_final_proof_bismut}) and (\ref{eq_thm_for_1_norms}) is that (\ref{eq_final_proof_bismut}) is a statement in realms of Theorem \ref{thm_bismut_rest}, and (\ref{eq_thm_for_1_norms}) is a statement in realms of Theorem \ref{thm_imm_thm}, which is exactly what we need.
			\par Now, let $\norm{\cdot}_{X/S}$, $\norm{\cdot}_{Y_0}$ be any norms which satisfy the assumptions of Theorem \ref{thm_imm_thm}.
			We denote by $\norm{\cdot}_{\mathscr{L}_n}$, $\norm{\cdot}_{\mathscr{L}_n'}$ the Hermitian norms on $\mathscr{L}_n, \mathscr{L}_n'$, defined as in Theorem \ref{thm_imm_thm}.
			By Proposition \ref{prop_int_loglog} and the anomaly formula of Bismut-Gillet-Soulé \cite{BGS3} (cf. Theorem \ref{thm_anomaly_cusp} for $m = 0$), applied for the line bundles $\mathscr{L}_n$ and $\mathscr{L}_n'$, we get 
			\begin{equation}\label{eq_limit_anomaly}
				\lim_{t \to 0}
				\log \big( \norm{\cdot}_{\mathscr{L}_n}^{\varkappa} / \norm{\cdot}_{\mathscr{L}_n} \big)(t) = \log \big( \norm{\cdot}_{\mathscr{L}_n'}^{\varkappa} / \norm{\cdot}_{\mathscr{L}_n'} \big).
			\end{equation}
			From  (\ref{eq_thm_for_1_norms}) and (\ref{eq_limit_anomaly}), we deduce that the norm $\norm{\cdot}_{\mathscr{L}_n}$ extends continuously over $S$ and under the isomorphism (\ref{isom_main}), the following holds
			\begin{equation}\label{eq_thm_for_all_norms}
				\norm{\cdot}_{\mathscr{L}_n}|_{\Delta} = \exp(k \cdot \rk{\xi} \cdot A_{-n}) \cdot \norm{\cdot}_{\mathscr{L}'_n},
			\end{equation}
			in other words Theorem \ref{thm_imm_thm} holds, but instead of $C_n$, we have an undetermined constant $A_{-n}$.
			\par \textbf{Now, to prove (\ref{eq_anom_thrg_g_lim})}, we apply the anomaly formula of Bismut-Gillet-Soulé \cite{BGS3} (cf. Theorem \ref{thm_anomaly_cusp} for $m = 0$), and by the triviality of $(\xi, h^{\xi})$ in the neighborhood of $\Sigma_{X/S}$, we deduce that for any $t \in S \setminus |\Delta|$, we have
			\begin{equation}\label{eq_anom_thrg_g}
				\log \Big( 
				\norm{\cdot}_{\mathscr{L}_n}^{\varkappa} / \norm{\cdot}_{\mathscr{L}_n}^{\sim}
				\Big)(t)
				=
				6 \cdot \rk{\xi} \cdot \int_{X_t} G,
			\end{equation}
			where the differential form $G$ is given by
			\begin{multline}\label{eq_g_expr_explc}
				G = \Big( 
			 					\widetilde{\td} \big(\omega_{X/S}^{-1}, (\, \norm{\cdot}_{X/S}^{\sim, {\rm{ind}}})^{-2}, (\, \norm{\cdot}^{\varkappa}_{X/S})^{-2} \big)  \ch \big(\omega_{X/S}^n,  (\, \norm{\cdot}^{\sim}_{X/S})^{2n} \big) 
			 					\\
								+			 		 					
					 			\td \big(\omega_{X/S}^{-1}, (\, \norm{\cdot}^{\varkappa}_{X/S})^{-2} \big) \widetilde{\ch} \big(\omega_{X/S}^n, (\, \norm{\cdot}_{X/S}^{\sim})^{2n}, (\, \norm{\cdot}_{X/S}^{\varkappa})^{2n} \big) \Big)^{[2]}.
			\end{multline} 
			We decompose
			$
				G =  G_1 +  G_2
			$,
			where 
			\begin{align}
				& 
				G_1 = \widetilde{\td} \big(\omega_{X/S}^{-1}, (\, \norm{\cdot}_{X/S}^{\sim, {\rm{ind}}})^{-2}, (\, \norm{\cdot}^{\varkappa}_{X/S})^{-2} \big)^{[2]}, \label{eq_g_1_eqn_full}\\
				& \nonumber
				G_2 =
				\Big( 
			 					\widetilde{\td} \big(\omega_{X/S}^{-1}, (\, \norm{\cdot}_{X/S}^{\sim, {\rm{ind}}})^{-2}, (\, \norm{\cdot}^{\varkappa}_{X/S})^{-2} \big)^{[0]}  \ch \big(\omega_{X/S}^n,  (\, \norm{\cdot}^{\sim}_{X/S})^{2n} \big) 
			 					\\ 
			 	&
			 	\qquad \qquad \qquad \qquad
			 	+			 		 					
				\td \big(\omega_{X/S}^{-1}, (\, \norm{\cdot}^{\varkappa}_{X/S})^{-2} \big) \widetilde{\ch} \big(\omega_{X/S}^n, (\, \norm{\cdot}_{X/S}^{\sim})^{2n}, (\, \norm{\cdot}_{X/S}^{\varkappa})^{2n} \big) \Big)^{[2]}.\label{eq_g_2_eqn_full}
			\end{align}
			By (\ref{eq_g_1_eqn_full}), (\ref{eq_g_2_eqn_full}) and the fact that the norms  $\norm{\cdot}_{X/S}^{\sim, {\rm{ind}}}$,  $\norm{\cdot}_{X/S}^{\sim}$, $\norm{\cdot}_{X/S}^{\varkappa}$ coincide over $X \setminus (\cup_{i = 1}^{k} U(Q_i, 1/2))$, we conclude that $G_i$, $i = 1, 2$ have support over $\cup_{i = 1}^{k} U(Q_i, 1/2)$.
			Moreover, over $U(Q_i, 1)$, by (\ref{defn_gtx_1}), (\ref{eq_defn_tilde_norm}), (\ref{eq_defn_gtxt_1}), the following identities hold
			\begin{equation}\label{eq_explc_eval_norm_pf_end}
			\begin{aligned}
				&  \Big( \Big\|  \frac{d z^i_{0}}{z^i_{0}} - \frac{d z^i_{1}}{z^i_{1}} \Big\|_{X/S}^{\sim, {\rm{ind}}} \Big)^{2} \big|_{X_t}
				= 
				2 \Big( |z^i_{0}|^2 + |z^i_{1}|^2 \Big)^{-1},
				\\
				&  \Big( \Big\| \frac{d z^i_{0}}{z^i_{0}} - \frac{d z^i_{1}}{z^i_{1}} \Big\|_{X/S}^{\sim} \Big)^{2} \big|_{X_t}
				= 
				2 \Big( (1 - \widetilde{\nu}) \cdot (|z^i_{0}|^2 + |z^i_{1}|^2) + \widetilde{\nu}  \Big)^{-1},
				\\
				&
				\Big( \Big\| \frac{d z^i_{0}}{z^i_{0}} - \frac{d z^i_{1}}{z^i_{1}} \Big\|_{X/S}^{\varkappa} \Big)^{2}  \big|_{X_t}
				= 
				2
				\bigg(
				(1 - \widetilde{\nu}) \cdot \big( |z^i_{0}|^2 + |z^i_{0}|^2 \big) 
				\\
				&
				\qquad \qquad \qquad \qquad \qquad \qquad \qquad \qquad \qquad \qquad
				+  
				\frac{\pi^2 \widetilde{\nu}}{(\log |t|)^2} \bigg(
				\sin 
				\frac{\pi \log |z^i_{0}|}{\log |t|}
			\bigg)^{-2}
			 \bigg)^{-1}.
			\end{aligned}
			\end{equation}
			\par Now, as the norm $\norm{\cdot}_{X/S}^{\sim}$ is smooth over $X$, the norm $\norm{\cdot}_{X/S}^{\varkappa}$ is good on $X \setminus \pi^{-1}(|\Delta|)$ with singularities along $\pi^{-1}(\Delta)$ and it has log-log growth along $\Sigma_{X/S}$, by Propositions \ref{prop_int_loglog}, \ref{prop_gft_good} and (\ref{ch_bc_0}), (\ref{ch_bc_2}), (\ref{eq_g_2_eqn_full}), (\ref{eq_explc_eval_norm_pf_end}), we conclude that
			\begin{equation}\label{eq_lim_g_2}
				\lim_{t \to 0} \int_{X_t} G_2 = k \cdot A_{-n}^{II},
			\end{equation}
			where the constant $A_{-n}^{II} \in \real$ is defined by
			\begin{multline}\label{eq_aII_defn}
				A_{-n}^{II} :=
				k^{-1} \cdot \int_{Y_0} \Big(
			 					\widetilde{\td} \big(\omega_{Y_0}^{-1}, g^{TY_0}_{\sim}, g^{TY_0}_{\varkappa} \big)^{[0]}  \ch \big(\omega_{Y_0}(D)^n,  (\, \norm{\cdot}^{\sim}_{Y_0})^{2n} \big) 
			 	\\ 
			 	+			 		 					
				\td \big(\omega_{Y_0}^{-1}, g^{TY_0}_{\varkappa} \big) \widetilde{\ch} \big(\omega_{Y_0}(D)^n, (\, \norm{\cdot}_{Y_0}^{\sim})^{2n}, (\, \norm{\cdot}_{Y_0}^{\varkappa})^{2n} \big) \Big).
			\end{multline}
			By (\ref{eq_explc_eval_norm_pf_end}), the discussion before and (\ref{eq_aII_defn}), we see that $A_{-n}^{II}$ depends only on $n \in \nat$ and not on $\pi : X \to S \ldots$
			\par 
			For $n \in \nat$, we note
			\begin{equation}\label{eq_aI_defn}
				A_{-n}^{I} :=
				k^{-1} \cdot \int_{Y_0}^{\textbf{r}} 
			 					\widetilde{\td} \big(\omega_{Y_0}^{-1}, g^{TY_0}_{\sim}, g^{TY_0}_{\varkappa} \big) + 1.
			\end{equation}
			Similarly to (\ref{eq_aII_defn}), by (\ref{eq_explc_eval_norm_pf_end}), we see that $A_{-n}^{I}$ depends only on $n \in \nat$ and not on $\pi : X \to S \ldots$
			We would like to prove that
			\begin{equation}\label{eq_lim_g_1}
				\lim_{t \to 0} \Big( \int_{X_t} G_1 + k \cdot \frac{\log |t|}{6} \Big) = k \cdot A_{-n}^{I}.
			\end{equation}
			Clearly, the statements (\ref{eq_lim_g_2}), (\ref{eq_aII_defn}) and (\ref{eq_aI_defn}), (\ref{eq_lim_g_1}) imply (\ref{eq_anom_thrg_g_lim}) for $A_{-n}'' =  A_{-n}^{I} +  A_{-n}^{II}$.
			\par To show (\ref{eq_lim_g_1}), we remark that by (\ref{eq_c1_form_pullback_smooth}) and (\ref{eq_explc_eval_norm_pf_end}), for $t \in S \setminus |\Delta|$, we have
			\begin{equation}\label{eq_c1_form_pullback_smooth1}
				c_1(\omega_{X/S}^{-1}, (\, \norm{\cdot}_{X/S}^{\sim, {\rm{ind}}})^{-2})|_{X_t \cap U(Q_i, 1/2)}
				= \frac{4 |z_0^{i}|^2 |t|^2}{(|z_0^{i}|^{4} + |t|^{2})^2} \frac{dz_0^{i} d \overline{z}_0^{i}}{2 \pi \imun}.
			\end{equation}
			By the fact that the norm $\norm{\cdot}_{X/S}^{\sim, {\rm{ind}}}$ coincides with $\norm{\cdot}_{X/S}^{\varkappa}$ away from $U(Q_i, 1/2)$, by Green identities and (\ref{ch_bc_0}), (\ref{ch_bc_2}), (\ref{eq_c1_form_pullback_smooth1}), we see that the following identity holds
			\begin{multline}\label{eq_decomp_f1}
				\int_{X_t} G_1 = \frac{k}{12} \int_{2|t| < |z_0^{i}| < 1/2} \log \big( |z_0^{i}|^2 + |t/z_0^{i}|^{2} \big)  \frac{4 |z_0^{i}|^2 |t|^2}{(|z_0^{i}|^{4} + |t|^{2})^2} \frac{dz_0^{i} d \overline{z}_0^{i}}{2 \pi \imun} 
				\\
				-
				 \frac{1}{6} 
				 \sum_{i = 1}^{k}
				 \int_{2|t| < |z_0^{i}| < 1/2} 
				 \log \Big( \frac{1}{\sqrt{2}} \Big\| \frac{d z^i_{0}}{z^i_{0}} - \frac{d z^i_{1}}{z^i_{1}} \Big\|_{X/S}^{\varkappa} \Big)
				 c_1(\omega_{X/S}, (\, \norm{\cdot}_{X/S}^{\varkappa})^2)
				.
			\end{multline}
			By making change of variables $y := z_0^{i} \cdot  |t|^{-1/2}$, we see that
			\begin{multline}\label{eq_last_thm_decomp222}
				\int_{2|t| < |z_0^{i}| < 1/2} \log \big( |z_0^{i}|^2 + |t /  z_0^{i}|^{2} \big)  \frac{4 |z_0^{i}|^2 |t|^2}{(|z_0^{i}|^{4} + |t|^{2})^2} \frac{dz_0^{i} d \overline{z}_0^{i}}{2 \pi \imun}
				\\
				=
				(\log |t| ) \cdot \int_{2|t|^{1/2} < |y| < |t|^{-1/2}/2}   \frac{4 |y|^2}{(|y|^{4} + 1)^2} \frac{dy d \overline{y}}{2 \pi \imun}
				\\
				+
				\int_{2|t|^{1/2} < |y| < |t|^{-1/2}/2} \log ( |y|^2 + |y|^{-2} )  \frac{4 |y|^2}{(|y|^{4} + 1)^2} \frac{dy d \overline{y}}{2 \pi \imun}.
			\end{multline}
			Also, we see that
			\begin{equation}
			\begin{aligned}
				&\label{eq_lim_log_diverg}
				\int_{y \in \comp} \frac{4 |y|^2}{(|y|^{4} + 1)^2} \frac{dy d \overline{y}}{2 \pi \imun} 
				= 
				- \int_{0}^{+ \infty} \frac{8 r^3 dr}{(r^4 + 1)^2}
				=
				- 2,
				\\
				&
				\int_{y \in \comp} \log ( |y|^2 + |y|^{-2} )  \frac{4 |y|^2}{(|y|^{4} + 1)^2} \frac{dy d \overline{y}}{2 \pi \imun}
				=
				- \int_{0}^{+ \infty} \frac{8 r^3 \log(r^2 + r^{-2}) dr}{(r^4 + 1)^2}
				=
				- 2.
			\end{aligned}
			\end{equation}
			Also, we see easily by (\ref{eq_explc_eval_norm_pf_end}) that
			\begin{equation}
			\begin{aligned}
				\lim_{t \to 0}
				\int_{2|t| < |z_0^{i}| < 1/2} 
				&
				 \log \Big( \frac{1}{\sqrt{2}} \Big\| \frac{d z^i_{0}}{z^i_{0}} - \frac{d z^i_{1}}{z^i_{1}} \Big\|_{X/S}^{\varkappa} \Big)
				 c_1(\omega_{X/S}, (\, \norm{\cdot}_{X/S}^{\varkappa})^2)
				 \\
				 =
				 &
				 -
				 \frac{1}{2}
				 \int_{0 < |z_0^{i}| < 1/2} 
				 \log \Big(
				(1 - \widetilde{\nu})  |z^i_{0}|^2
				+  
				\frac{\widetilde{\nu}}{(\log |z^i_{0}|)^{2}}
				\Big)
				 c_1(\omega_{Y_0}(D), (\, \norm{\cdot}_{Y_0}^{\varkappa})^2)
				 \\
				 &
				 -
				 \frac{1}{2}
				 \int_{0 < |z_1^{i}| < 1/2} 
				 \log \Big(
				(1 - \widetilde{\nu})  |z^i_{1}|^2
				+  
				\frac{\widetilde{\nu}}{(\log |z^i_{1}|)^{2}}
				\Big)
				 c_1(\omega_{Y_0}(D), (\, \norm{\cdot}_{Y_0}^{\varkappa})^2).
			\end{aligned}
			\end{equation}
			However, over the set $|z_0^{i}| < 1/2$, by (\ref{eq_explc_eval_norm_pf_end}), we have
			\begin{equation}\label{eq_repr_rapp_gty}
				\log \Big(
				(1 - \widetilde{\nu})  |z^i_{0}|^2
				+  
				\frac{\widetilde{\nu}}{(\log |z^i_{0}|)^{2}}
				\Big) - 2 \log |z^i_{0}|
				=
				\log \big( g^{TY_0}_{\varkappa} / g^{TY_0}_{\sim} \big).
			\end{equation}
			Similar identity holds over the set $|z_1^{i}| < 1/2$.
			Thus, by (\ref{eq_repr_rapp_gty}) and the fact that $c_1(\omega_{Y_0}^{-1}, g^{TY_0}_{\sim}) = 0$ over $\{ |z_0^{i}| < 1/2 \}$, we deduce
			\begin{multline}
				\frac{1}{2}
				\sum_{j = 0}^{1}
				\sum_{i = 1}^{k}
				\int_{0 < |z_j^{i}| < 1/2} 
				 \log \Big(
				(1 - \widetilde{\nu})  |z^i_{j}|^2
				+  
				\frac{\widetilde{\nu}}{(\log |z^i_{j}|)^{2}}
				\Big)
				 c_1(\omega_{Y_0}(D), (\, \norm{\cdot}_{Y_0}^{\varkappa})^2)
				 \\
				 =
				6
				\int_{Y_0}^{\textbf{r}}
				\widetilde{\td} \big(\omega_{Y_0}^{-1}, g^{TY_0}_{\sim}, g^{TY_0}_{\varkappa} \big)
				 +
				 k
				 \cdot
				 \sum_{j = 0}^{1}
				 \int_{0 < |z^i_{j}| < 1/2}^{\textbf{r}}
				 \log |z^i_{j}|
				 c_1(\omega_{Y_0}(D), (\, \norm{\cdot}_{Y_0}^{\varkappa})^2).
			\end{multline}
			However, by (\ref{eq_explc_eval_norm_pf_end}) and Green identities one deduces that for a standard Laplacian $\Delta = \frac{\partial^2}{\partial x^2} + \frac{\partial^2}{\partial y^2}$ over $\comp$, we have
			\begin{multline}\label{eq_regular_computation}
				\int_{\epsilon < |z^i_{0}| < 1/2}
				 \log |z^i_{0}|
				 \cdot
				 c_1(\omega_{Y_0}(D), (\, \norm{\cdot}_{Y_0}^{\varkappa})^2)
				 \\
				 =
				 \frac{1}{4 \pi}
				 \int_{\epsilon < |z^i_{0}| < 1/2}
				 \log |z^i_{0}|
				 \cdot
				 \Delta
				 \log \Big(
				 (1 - \widetilde{\nu})
				 +  
				 \frac{\widetilde{\nu}}{|z^i_{0} \log |z^i_{0}||^2}
				 \Big)
				 dx dy
				\\
				=
				- \epsilon \cdot \Big( \frac{1}{\epsilon} \cdot  \log ( \epsilon | \log \epsilon |)
				-
				(\log \epsilon ) \cdot \Big( \frac{1}{\epsilon} + \frac{1}{\epsilon |\log \epsilon|} \Big)
				\Big)
				=
				1 -  \log| \log \epsilon|.
			\end{multline}
			Thus, we deduce by (\ref{eq_regular_computation}) that
			\begin{equation}\label{eq_regular_computation2}
				 \int_{|z^i_{0}| < 1/2}^{\textbf{r}}
				 \log |z^i_{0}|
				 c_1(\omega_{Y_0}(D), (\, \norm{\cdot}_{Y_0}^{\varkappa})^2) = 1.
			\end{equation}
			By (\ref{eq_decomp_f1})-(\ref{eq_regular_computation2}) the proof of (\ref{eq_lim_g_1}), and thus (\ref{eq_anom_thrg_g_lim}), is complete.
		\end{proof}
		\begin{rem}\label{rem_comp_pert_special_case}
			By Theorem \ref{thm_imm_thm},  (\ref{eq_an_prone}), (\ref{eq_a_n_2_prime}), (\ref{eq_limit_almost_final_last}) and (\ref{eq_thm_for_1_norms}), we deduce that 
			\begin{equation}
			\begin{aligned}
				& 12 \log \big(
				\norm{\cdot}_Q(g^{TY_0}_{\varkappa}, \rho^* (h^{\xi}) \otimes (\, \norm{\cdot}_{Y_0}^{\varkappa})^{2n} )
				/
				\norm{\cdot}_Q(g^{TY_0}_{\sim}, \rho^* (h^{\xi}) \otimes (\, \norm{\cdot}_{Y_0}^{\sim})^{2n} )
				\big)
				\\
				& 
				\qquad \qquad
				=
				6
				 \int_{Y_0}^{\textbf{r}} \Big( 
			 					\widetilde{\td} \big(\omega_{Y_0}^{-1}, g^{TY_0}_{\sim}, g^{TY_0}_{\varkappa} \big)  
			 					\ch (\rho^*(\xi), \rho^* (h^{\xi}))
			 					\ch \big(\omega_{Y_0}(D)^n,  (\, \norm{\cdot}^{\sim}_{Y_0})^{2n} \big) 
			 					\\
								& 
								\qquad \qquad \phantom{= 6 \,  \int_{Y_0}^{\textbf{r}} \Big(}
								+			 		 					
					 			\td \big(\omega_{X/S}^{-1}, g^{TY_0}_{\varkappa} \big)	 					
					 			\ch (\xi, h^{\xi})
					 			\widetilde{\ch} \big(\omega_{Y_0}(D)^n, (\, \norm{\cdot}^{\sim}_{Y_0})^{2n}, (\, \norm{\cdot}^{\varkappa}_{Y_0})^{2n} \big) \Big)
				\\
				& 
				\qquad \qquad \phantom{= \,}
				+
				k \cdot \rk{\xi}
				\Big(
				24 \zeta'(-1) - 6 \log(2\pi)
				+
				1
				-
				C_{-n}
				\Big).				
				\end{aligned}
			\end{equation}
			In other words, we see that Theorem \ref{thm_full_comp_pert} holds for $\overline{M} = Y_0$, $g^{T \overline{M}} := g^{TY_0}_{\sim}$, $g^{T M} := g^{TY_0}_{\varkappa}$, $h^{\xi} = h^{\xi}_{0} := \rho^*(h^{\xi})$.
		\end{rem}
		\begin{proof}[Proof of Theorem \ref{thm_full_comp_pert}.]
			Let's say few words about the general strategy of the proof. First, for fixed $n \in \nat$, $(\overline{M}, D_M, g^{TM})$, $g^{TM}$, $g^{T \overline{M}}$, $\norm{\cdot}_{\overline{M}}$, $h^{\xi}$, $h^{\xi}_0$, we define $E_{-n} \in \real$ by equality (\ref{eq_comp_pert}).
			Then by using anomaly formula, Theorem \ref{thm_anomaly_cusp}, we prove that $E_{-n}$, defined by this formula, is independent of $g^{TM}$, $g^{T \overline{M}}$, $\norm{\cdot}_{\overline{M}}$, $h^{\xi}$, $h^{\xi}_0$. Then by using relative compact perturbation theorem, Theorem \ref{thm_comp_appr}, we prove that $E_{-n}$ is independent of $(\overline{M}, D_M)$. Finally, by using Remark \ref{rem_comp_pert_special_case}, we conclude.
			Now let's explain each of this steps in detail.
			\par 
			First, in the proof of independence of $E_{-n}$, defined by (\ref{eq_comp_pert}), on $g^{TM}$, $g^{T \overline{M}}$, $\norm{\cdot}_{\overline{M}}$, $h^{\xi}$, $h^{\xi}_0$, we will only restrict ourselves to the proof of independence on $g^{T \overline{M}}$, as the other cases are identical.
			\par 
			Now, to show independence on $g^{T \overline{M}}$, by the anomaly formula of Bismut-Gillet-Soulé \cite{BGS3} (cf. Theorem \ref{thm_anomaly_cusp} for $m = 0$), it is enough to prove that for any two Kähler metrics $g^{T \overline{M}}_{1}$, $g^{T \overline{M}}_{2}$ over $\overline{M}$, the following holds 
			\begin{multline}
				\int_{M}^{\textbf{r}} 
		 		\Big[ 
		 			\Big(
	 				\widetilde{\td} \big(\omega_{\overline{M}}^{-1}, g^{T\overline{M}}_{1}, g^{TM} \big) - \widetilde{\td} \big(\omega_{\overline{M}}^{-1}, g^{T\overline{M}}_{2}, g^{TM} \big) 
	 				\Big) \ch \big(\xi, h^{\xi}_0 \big)  \ch \big(\omega_M(D)^n, \norm{\cdot}_{\overline{M}}^{2n} \big) 
	 			\Big]^{[2]}
	 			\\
	 			=
	 			\int_{\overline{M}}
		 		\Big[ 
	 				\widetilde{\td} \big(\omega_{\overline{M}}^{-1}, g^{T\overline{M}}_{1},  g^{T\overline{M}}_{2} \big) \ch \big(\xi, h^{\xi}_0 \big)  \ch \big(\omega_M(D)^n, \norm{\cdot}_{\overline{M}}^{2n} \big) 
	 			\Big]^{[2]}
	 			\\
	 			+
	 			\frac{\rk{\xi}}{6} \ln \Big( \norm{\cdot}^{D_M}_{\overline{M}, 1}/ \norm{\cdot}^{D_M}_{\overline{M}, 2} \Big),
			\end{multline}
			where the norms $\norm{\cdot}^{D_M}_{\overline{M}, 1}$, $\norm{\cdot}^{D_M}_{\overline{M}, 2}$ on the complex line $\otimes_{P \in D_M} \omega_{\overline{M}}|_{P}$ are induced by $g^{T\overline{M}}_{1}$ and $g^{T\overline{M}}_{2}$ respectively.
			Since the function $\log (g^{T\overline{M}}_{1} / g^{T\overline{M}}_{2})$ is integrable, it is enough to prove that 
			\begin{multline}
			\int_{M}^{\textbf{r}} 
		 		\Big[ 
		 			\Big(
	 				\widetilde{\td} \big(\omega_{\overline{M}}^{-1}, g^{T\overline{M}}_{1}, g^{TM} \big) - \widetilde{\td} \big(\omega_{\overline{M}}^{-1}, g^{T\overline{M}}_{2}, g^{TM} \big) 
	 				\Big)
	 			\Big]^{[2]}
	 			\\
	 			=
	 			\int_{\overline{M}}
		 		\Big[ 
	 				\widetilde{\td} \big(\omega_{\overline{M}}^{-1}, g^{T\overline{M}}_{1},  g^{T\overline{M}}_{2} \big) 
	 			\Big]^{[2]}
	 			+
	 			\frac{1}{6} \ln \Big( \norm{\cdot}^{D_M}_{\overline{M}, 1}/ \norm{\cdot}^{D_M}_{\overline{M}, 2} \Big).
			\end{multline}
			By using (\ref{ch_bc_0}) and (\ref{ch_bc_2}), we see that it is enough to prove that 
			\begin{multline}
			\int_{M}^{\textbf{r}} 
				\Big(
		 		\log (g^{T\overline{M}}_{1} / g^{TM}) \big( c_1(\omega_{\overline{M}}^{-1}, g^{T\overline{M}}_{1}) + c_1(\omega_{\overline{M}}^{-1}, g^{TM})\big)
		 		\\
		 		-
		 		\log (g^{T\overline{M}}_{2} / g^{TM}) \big( c_1(\omega_{\overline{M}}^{-1}, g^{T\overline{M}}_{2}) + c_1(\omega_{\overline{M}}^{-1}, g^{TM})\big)
		 		\Big)
	 			\\
	 			=
	 			\int_{\overline{M}}
		 		\log (g^{T\overline{M}}_{1} / g^{T\overline{M}}_{2}) \big( c_1(\omega_{\overline{M}}^{-1}, g^{T\overline{M}}_{1}) + c_1(\omega_{\overline{M}}^{-1},  g^{T\overline{M}}_{2})\big)
	 			+
	 			2 \ln \Big( \norm{\cdot}^{D_M}_{\overline{M}, 1}/ \norm{\cdot}^{D_M}_{\overline{M}, 2} \Big).
			\end{multline}
			Recall the trivial identity 
			\begin{equation}\label{eq_rel_c_1_c_1_diff_metr}
				c_1(\omega_{\overline{M}}^{-1}, g^{T\overline{M}}_{1}) = c_1(\omega_{\overline{M}}^{-1}, g^{T\overline{M}}_{2}) + \frac{\partial \dbar}{2 \pi \imun} \log (g^{T\overline{M}}_{1} / g^{T\overline{M}}_{2}).
			\end{equation}
			Using (\ref{eq_rel_c_1_c_1_diff_metr}), we see that it is enough to prove
			\begin{multline}
			\int_{M}^{\textbf{r}} 
				\Big(
		 		\log (g^{T\overline{M}}_{1} / g^{TM}) \frac{\partial \dbar}{2 \pi \imun} \log (g^{T\overline{M}}_{1} / g^{T\overline{M}}_{2})
		 		+
		 		 \log (g^{T\overline{M}}_{1} / g^{T\overline{M}}_{2}) c_1(\omega_{\overline{M}}^{-1}, g^{TM}) 
		 		\Big)
	 			\\
	 			=
	 			\int_{\overline{M}}
		 		\log (g^{T\overline{M}}_{1} / g^{T\overline{M}}_{2}) c_1(\omega_{\overline{M}}^{-1},  g^{T\overline{M}}_{2})
	 			+
	 			2 \ln \Big( \norm{\cdot}^{D_M}_{\overline{M}, 1}/ \norm{\cdot}^{D_M}_{\overline{M}, 2} \Big),
			\end{multline}
			which is a simple application of Green identities.
			\par Now, as we already know that the constant $E_{-n}$, defined by (\ref{eq_comp_pert}), doesn't depend on $g^{TM}$, $g^{T \overline{M}}$, $\norm{\cdot}_{\overline{M}}$, $h^{\xi}$, $h^{\xi}_0$, to show that it is independent of $(\overline{M}, D_M)$, it is enough to apply relative compact perturbation theorem \ref{thm_comp_appr} for a flattening $g^{T\overline{M}}$ of $g^{TM}$ and for $h^{\xi} = h^{\xi}_{0}$ trivial around the cusps.
			\par Now we know that the constant $E_{-n}$, defined by (\ref{eq_comp_pert}), doesn't depend on $m \in \nat$, $g^{TM}$, $g^{T \overline{M}}$, $\norm{\cdot}_{\overline{M}}$, $h^{\xi}$, $h^{\xi}_0$, $(\overline{M}, D_M)$. Thus, it is a universal constant, depending only on $n \in \nat$. To show that it is indeed equal to (\ref{eq_defn_e_k}), it is enough to verify it at least for one example, which is done in Remark \ref{rem_comp_pert_special_case}. This concludes the proof.
		\end{proof}
 		
		\bibliographystyle{abbrv}

\Addresses

\end{document}